\documentclass[10pt]{amsart}
\usepackage{latexsym, amssymb,  eucal}

\newtheorem{thm}{Theorem}[subsection]
\newtheorem{lemma}[thm]{Lemma}
\newtheorem{prop}[thm]{Proposition}
\newtheorem{cor}[thm]{Corollary}

\newtheorem{fact}[thm]{Fact}
\newtheorem{rmk}[thm]{Remark}

\newtheorem{exercise}[thm]{Exercise}
\newtheorem{question}[thm]{Question}

\theoremstyle{definition}
\newtheorem{df}[thm]{Definition}
\newtheorem{ex}[thm]{Example}

\newcommand{\BM}{\mathbb{M}}  
\newcommand{\BN}{\mathbb{N}}
\newcommand{\BD}{\mathbb{D}}
\newcommand{\BJ}{\mathbb{J}}
\newcommand{\BK}{\mathbb{K}}
\newcommand{\BS}{\mathbb{S}}
\newcommand{\BR}{\mathbb{R}}

\newcommand{\cu}[1]{\mathcal{#1}}

\newcommand{\sa}[1]{\mathsf{#1}}

\newcommand{\tp}{\operatorname{tp}}

\newcommand{\dcl}{\operatorname{dcl}}
\newcommand{\acl}{\operatorname{acl}}
\newcommand{\cl}{\operatorname{cl}}
\newcommand{\met}{\operatorname{met}}
\newcommand{\Th}{\operatorname{Th}}

\newcommand{\dist}{\operatorname{dist}}
\newcommand{\rs}{\operatorname{rs}}

\newcommand{\range}{\operatorname{range}}
\renewcommand{\hat}{\widehat}

\makeatletter

\def\indsym#1#2{%
  \setbox0=\hbox{$\m@th#1x$}%
  \kern\wd0%
  \hbox to 0pt{\hss$\m@th#1\mid$\hbox to 0pt{$\m@th#1^{#2}$}\hss}%
  \lower.9\ht0\hbox to 0pt{\hss$\m@th#1\smile$\hss}%
  \kern\wd0}
\newcommand{\ind}[1][]{\mathop{\mathpalette\indsym{#1}}}

\def\nindsym#1#2{%
  \setbox0=\hbox{$\m@th#1x$}%
  \kern\wd0%
  \hbox to 0pt{\hss$\m@th#1\not$\kern1.4\wd0\hss}
  \hbox to 0pt{\hss$\m@th#1\mid$\hbox to 0pt{$\m@th#1^{\,#2}$}\hss}%
  \lower.9\ht0\hbox to 0pt{\hss$\m@th#1\smile$\hss}%
  \kern\wd0}

\def\dotminussym#1#2{%
  \setbox0=\hbox{$\m@th#1-$}%
  \kern.5\wd0%
  \hbox to 0pt{\hss\hbox{$\m@th#1-$}\hss}%
  \raise.6\ht0\hbox to 0pt{\hss$\m@th#1.$\hss}%
  \kern.5\wd0}
\newcommand{\dotminus}{\mathbin{\mathpalette\dotminussym{}}}

\def\dotlesym#1#2{%
  \setbox0=\hbox{$\m@th#1-$}%
  \kern.5\wd0%
  \hbox to 0pt{\hss\hbox{$\m@th#1\le$}\hss}%
  \raise 1\ht0\hbox to 0pt{\hss$\m@th#1.$\hss}%
  \kern.5\wd0}
\newcommand{\dotle}{\mathbin{\mathpalette\dotlesym{}}}

\begin{document}

\title{Model Theory for Real-valued Structures}

\author{H. Jerome Keisler}

\address{University of Wisconsin-Madison, Department of Mathematics, Madison,  WI 53706-1388}
\email{keisler@math.wisc.edu}

\date{\today}

\begin{abstract}
We consider general structures where formulas have truth values in the real unit interval as in continuous model theory, but whose predicates
and functions need not be uniformly continuous with respect to a distance predicate.
Every general structure can be expanded to a
pre-metric structure by adding a distance predicate that is a uniform limit of formulas.  Moreover, that distance predicate is unique up to uniform equivalence.
We use this to extend the central notions in the model theory of metric structures to general structures, and show that many model-theoretic results from
the literature about metric structures have natural analogues for general structures.
 \end{abstract}

\maketitle

\setcounter{tocdepth}{2}

\tableofcontents

\section{Introduction}

We show that much of the model theory of metric structures carries over to general structures that
still have truth values in $[0,1]$.
A general ($[0,1]$-valued) structure is like a classical first-order structure, except that the predicate symbols have truth values in $[0,1]$.
 It has a vocabulary  consisting of predicate, function, and constant symbols.  General structures are called models of $[0,1]$-valued logic
 in [CK66], and are called $[0,1]$-valued structures in [AH] and [Ca].

Continuous model theory, as developed in [BBHU], [BU], and [Fa], has been
highly successful in the study  of metric structures, which are more elaborate than general $[0,1]$-valued structures.
A metric (or pre-metric) structure has, in addition to a vocabulary, a metric signature that specifies a
distance predicate symbol that is a metric (or pseudo-metric), and a modulus of uniform continuity for each predicate and function symbol.
The notions of formula, truth value of a formula, and sentence, are the same for general structures as for pre-metric structures.

We show that almost all of the model-theoretic properties of metric structures in [BBHU] extend in a unique way (called the absolute version)
to general structures, and have natural characterizations in terms of the general structure itself.  We also show that almost all of the
 results about metric structures in [BBHU] imply analogous results about general structures.  We get similar results for the infinitary
 continuous model theory in [Ea15], and the model theory for unbounded metric structures in [BY08].
We show, for instance, that the property of stability has an absolute version, and a general structure is stable if and only if it has a stable independence relation.
One can readily find (e.g. using Theorem \ref{t-interpret-upgrade} below) examples of stable general structures that
are quite different from the familiar examples of stable metric structures in [BBHU].

The formalism for metric structures in [BU] and [BBHU] is, to quote from [BU]. ``an immediate generalization of
classical first order logic, more natural and less technically involved than previous formalisms.''\footnote{Previous formalisms included metric open Hausdorff cats in [BY03a]
and Henson's Banach space logic in [He], [HI].}
The main advantage of their
formalism is that it is easy to describe a metric structure by specifying the universe, functions, constants, relations,
and moduli of uniform continuity.
Here we go a step further---the notion of a general structure is even less technically involved than the notion of a metric structure,
since it does not require one to specify a  metric signature giving a distance predicate and moduli of uniform continuity.

Formally, the key concepts that make things work in this paper are the notion of a pre-metric expansion of a theory (Definition \ref{d-metric-expansion}),
and the notion of an absolute version of a property of pre-metric structures (Definition \ref{d-absolute-version}).

Let $T$ be a  theory (or set of sentences) with vocabulary $V$.  A pre-metric expansion of $T$ is a metric theory $T_e$ that makes every general model $\cu M$ of $T$
into a pre-metric model $\cu M_e$ of $T_e$ in the following way.  $T_e$ has a metric signature over $V\cup\{D\}$ with distance $D$.
Moreover, there is a sequence $\langle d_m\rangle_{m\in\BN}$ of formulas of $V$, called an approximate distance for $T_e$, such that
for each general model $\cu M$ of $T$, $\cu M_e=(\cu M,D)$ with $D=\lim_{m\to\infty} d_m$.

By adapting arguments of Iovino [I94] and Ben Yaacov [BY05] to our setting, and using Theorem 4.23 of [BU], one can prove that every complete theory of
general structures with a countable vocabulary has a pre-metric
expansion.\footnote{In [BU], Theorem 4.23 was used to show that their formalism of continuous model theory is equivalent to the formalism of open Hausdorff cats in [BY03a].}
The Expansion Theorem \ref{t-metric-expansion-exists}  improves that by showing that every (not necessarily complete) theory $T$ with a countable vocabulary has a pre-metric
expansion in which each of the formulas $d_m(x,y)$ defines a pseudo-metric in $T$.  We will see in Section 4 that the Expansion Theorem has far-reaching consequences.
Proposition \ref{p-homeomorphic} shows that the pre-metric expansion of a theory is unique up to uniform equivalence (but far from unique).

We say that a property $\cu P$ of general structures with parameters is an absolute version of a property $\cu Q$
of pre-metric structures if whenever $\cu M$ is a general model of a theory $T$
and $T_e$ is a pre-metric expansion of $T$, $\cu M$ has property $\cu P$ if and only if  $\cu{M}_e$ has property $\cu{Q}$.
If  $\cu{Q}$ has an absolute version, its absolute version is unique.
 We regard the absolute version of a property of pre-metric structures as the ``right''
extension of that property to general structures.

General structures correspond to first order structures without equality in the same way that metric structures with a distance predicate
correspond to first order structures with equality.  In  first order model theory, each structure without equality can be expanded to a pre-structure with equality
in a unique way, so there is very little difference between structures without equality and structures with equality.

In $[0,1]$-valued model theory, every general structure can be expanded to a pre-metric structure by taking a pre-metric expansion.
In some cases, such as Banach spaces, there is a natural  distance predicate that can be regarded as the analogue of equality.
But in other cases, a general structure is easily described, but every pre-metric expansion is complicated.
In general, there may not be a natural distance, but since the pre-metric expansion is unique up to uniform equivalence, there will always
be a natural notion of uniform convergence.

In continuous model theory, one can either focus on general structures with a vocabulary,
or focus on metric structure with a metric signature.
When the goal is to apply the methods of continuous model theory to a family of structures with a natural metric, it is better to
focus on metric structures as in [BBHU].  But when the goal is to discover potential new applications of the results in
continuous model theory, it may be better to focus on general structures.

In Section 2 we lay the groundwork by developing some basic model-theory for general structures, and also give a brief review of pre-metric structures.
In Section 3 we introduce pre-metric expansions and prove their existence, and then introduce the notion of an absolute property.
In Section 4 we extend a variety of deeper properties and results from metric structures to general structures.

I thank Isaac Goldbring and Ward Henson, and Jos\'{e} Iovino for helpful comments related to this paper.

\section{Basic Model Theory for General Structures}

We assume that the reader is familiar with the model theory of metric structures as developed in the paper [BBHU], and we will freely use notation from that paper.
Our focus in this section will be on the definitions.
Only brief comments will be given on the proofs, which will be similar to the proofs of the corresponding
classical results that can be found, for example, in [CK12], as well as to proofs in [BBHU], and in the much earlier monograph [CK66] that
treated metric structures with the discrete metric.

\subsection{General Structures}

The space of truth values will be $[0,1]$, with $0$ representing truth.  A \emph{vocabulary} $V$ consists of a
 set of predicate symbols $P$ of finite arity, a set of function symbols $F$ of finite arity, and a set of constant symbols $c$.
A \emph{general} ($[0,1]$-valued) structure $\cu{M}$ consists of a vocabulary $V$, a non-empty universe set $M$, an element $c^{\cu{M}}\in M$ for each constant symbol $c$, a
mapping $P^\cu{M}\colon M^n\to[0,1]$ for each predicate symbol $P$ of arity $n$, and a mapping $F^\cu{M}\colon M^n\to M$ for each function symbol $F$ of arity $n$.
 A general structure determines a vocabulary, but does not determine a metric signature.

The \emph{formulas} are as in [BBHU], with the connectives\footnote{we consider the elements of $[0,1]$ to be $0$-ary connectives, so each $r\in[0,1]$ is a formula.}
 being all continuous functions from finite powers of $[0,1]$ into $[0,1]$, and the quantifiers $\sup_x,\inf_x$.
The  \emph{truth value} of a formula $\varphi(\vec{x})$ at a tuple $\vec{a}\in M^{|\vec x|}$ in a general structure $\cu{M}$ is an element of $[0,1]$ denoted by $\varphi^\cu{M}(\vec{a})$.  It is defined in the usual way by induction on the complexity of formulas.
The syntax and semantics of general structures described here are the same as in [AH], and are the same as the \emph{restricted continuous logic} in [Ca],
except that there the value $1$ denotes truth.

We deviate slightly from [BBHU] by defining a \emph{theory} to be a set of sentences (rather than a set of statements of the form $\varphi=0$).
Similarly, we define an \emph{$n$-type} over a set of parameters $A$ to be a set of formulas with free variables in $\vec x=\langle x_1,\ldots,x_n\rangle$ and
parameters from $A$.
A \emph{general model}\footnote{To avoid confusion, we will always write ``general model'' instead of just ``model'', because in the literature on continuous model theory,
``model'' is used to  mean what we call ``metric model''.} of a theory $T$ is a general structure $\cu M$ such that
$\varphi^{\cu M}=0$ for each $\varphi\in T$.
 An $n$-tuple $\vec b$ in a general structure $\cu{M}$ \emph{satisfies} a formula $\varphi(\vec x)$ in $\cu{M}$ if $\varphi^{\cu{M}}(\vec b)=0$,
and \emph{satisfies}, or \emph{realizes}, an $n$-type $p$  if $\vec b$ satisfies every formula $\varphi(\vec x)\in p$.

Hereafter, $\cu{M}, \cu{N}$ will denote general structures with universe sets $M,N$ and vocabulary $V$, and $S, T, U$ will denote
sets of sentences (that is, theories) with the vocabulary $V$.

The notions of substructure (denoted by $\subseteq$),  elementary equivalence (denoted by $\equiv$), elementary substructure and
extension (denoted by $\prec$ and $\succ$), and elementary chain
are as defined as in [BBHU], but applied to general structures
as well as metric structures.  $\cu M\models T$ means that $\cu M$ is a general model of $T$, $\cu M\models \varphi(\vec b)$
means that $\varphi^{\cu M}(\vec b)=0$, and $T\models U$ means that every general model of
$T$ is a general model of $U$.  $T$ and $U$ are \emph{equivalent} if they have the same general models.
The complete theory $\Th(\cu{M})$ of $\cu M$ is the set of all sentences true in $\cu{M}$.
Thus $\Th(\cu{M})=\Th(\cu{M}')$ if and only if $\cu{M}\equiv\cu{M}'$.
We state two elementary results.

\begin{fact}  \label{f-Skolem}  (Downward L\"{o}wenheim-Skolem)  For every general structure $\cu{M}$, there is a general structure $\cu{M}'\prec\cu{M}$ such that $|M'|\le\aleph_0+|V|$.
\end{fact}

\begin{fact}  \label{f-elementary-chain} (Elementary Chain Theorem) The union of an elementary chain of general structures
$\langle \cu{M}_\alpha\colon\alpha<\beta\rangle$ is an elementary extension of each $\cu{M}_\alpha$.
\end{fact}

By an \emph{embedding} $h\colon\cu{M}\to\cu{N}$ we mean a function $h\colon M\to N$  such that $h(c^\cu{M})=c^\cu{N}$ for each constant symbol $c\in V$,
and for every $n$ and $\vec{a}\in M^n$, $h(F^\cu{M}(\vec{a}))=F^\cu{N}(h(\vec{a}))$ for every function symbol $F\in V$ of arity $n$, and
$P^\cu{M}(\vec{a})=P^\cu{N}(h(\vec{a}))$ for every predicate symbol $P\in V$ of arity $n$.  We say that $\cu{M}$ is \emph{embeddable in}
$\cu{N}$ if there is an embedding $h\colon\cu{M}\to\cu{N}$.  Note that the image of an embedding $h\colon\cu{M}\to\cu{N}$ is a substructure of $\cu{N}$.
An \emph{elementary embedding} $h\colon\cu{M}\prec\cu{N}$ is an embedding that preserves the truth value of every formula.
We say that $\cu{M}$ is \emph{elementarily embeddable} in $\cu{N}$ if there exists an $h\colon\cu{M}\prec\cu{N}$.

In [BBHU], the reduction of a pre-metric structure was defined by identifying elements that are at distance zero from each other.
Some care is needed to choose the right notion of reduction for general structures.
In first order logic without equality, the reduction of a structure is formed by identifying two elements that
cannot be distinguished by atomic formulas.  We do the analogous thing here for general structures.

\begin{df}  \label{d-reduced}
For $a,b\in M$, we write $a\doteq^{\cu{M}} b$ if for every atomic formula $\varphi(x,\vec{z})$ and tuple $\vec{c}\in M^{|\vec z|}$, $\varphi^{\cu{M}}(a,\vec{c})=\varphi^{\cu{M}}(b,\vec{c}).$
$\cu{M}$ is \emph{reduced} if whenever $a\doteq^\cu{M} b$ we have $a=b$.
\end{df}

The relation $\doteq^{\cu{M}}$ is a very old idea that goes back to Leibniz around 1840, and is called \emph{Leibniz equality}.

\begin{rmk}  \label{r-doteq}
For any general structure $\cu M$ and $a,b\in M$, $a\doteq^{\cu M} b$ if and only if in $\cu M$, $(a,b)$ satisfies the set of formulas
$$\{\sup_{\vec z}|\varphi(a,\vec z)-\varphi(b,\vec z)|\colon\varphi \mbox{ is atomic}\}.$$
\end{rmk}

The \emph{reduction map} for $\cu M$ is the mapping that sends each element of $M$ to its equivalence class under $\doteq^{\cu M}$.
The \emph{reduction} of the general structure $\cu{M}$ is the reduced structure $\cu{N}$ such that $N$ is the set of equivalence classes of elements of
$M$ under $\doteq^{\cu{M}}$, and the reduction map for $\cu{M}$ is an embedding of $\cu M$ onto $\cu N$.
We say that $\cu{M},\cu{M}'$ are \emph{isomorphic}, in symbols $\cu {M}\cong\cu{M}'$, if there is an embedding from the reduction of $\cu{M}$ onto the reduction of $\cu{M}'$.
We write $h\colon \cu M\cong\cu N$ if $h\colon \cu M\to\cu N$, and for each $b\in N$ there exists $a\in M$ such that $h(a)\doteq^{\cu N} b$.

\begin{rmk}  \label{r-reduction}
\noindent\begin{itemize}
\item $\cong$ is an equivalence relation on general structures.
\item Every general structure is isomorphic to its reduction.
\item $\cu M,\cu N$ are isomorphic if and only their reductions are isomorphic.
\item  If there is an embedding of $\cu{M}$ \emph{onto} $\cu{N}$, then $\cu{M}\cong\cu{N}$.
\item $\cu{M}\cong\cu{N}$ implies $\cu{M}\equiv\cu{N}$.
\item $\cu M\cong\cu N$ if and only there exists $h$ such that $h\colon \cu M\cong\cu N$.
\end{itemize}
\end{rmk}

If $V^0\subseteq V$, and $\cu{M}^0$ is obtained from $\cu{M}$ by forgetting every symbol of $V\setminus V^0$, we call $\cu{M}$ an
\emph{expansion} of $\cu{M}^0$ to $V$, and call $\cu{M}^0$  \emph{the $V^0$-part} of $\cu{M}$.

\begin{rmk}  \label{r-formula-expansion}
Suppose $V^0\subseteq V$, and $\cu{M}^0$ is the $V^0$-part of $\cu M$.  Then for every formula $\varphi(\vec x)$ in the vocabulary $V^0$,
and every tuple $\vec b\in M^{|\vec x|}$, we have $\varphi^{\cu M}(\vec b)=\varphi^{\cu{M}^0}(\vec b)$.
\end{rmk}

\begin{rmk}  \label{r-reduced-part}  If $\cu{M}^0$ is reduced and $\cu M$ is an expansion of $\cu M^0$, then $\cu M$ is reduced.
\end{rmk}

\begin{proof} For all $a,b\in M$, $a\doteq^{\cu{M}} b$ implies $a\doteq^{\cu{M}^0} b$.
\end{proof}

\subsection{Ultraproducts}

The ultraproduct of an indexed family of general structures will be defined below as the reduction of the pre-ultraproduct, which is a general structure whose universe is the cartesian product.
As we will see later, this will be a direct generalization of the ultraproduct of metric structures as defined in [BBHU].
Recall that for any ultrafilter $\cu{D}$ over a set $I$ and function $g\colon I\to[0,1]$, there is a unique value $r=\lim_{\cu{D}} g$ in $[0,1]$
such that for each neighborhood $Y$ of $r$, the set of $i\in I$ such that $g(i)\in Y$ belongs to $\cu{D}$.

\begin{df} \label{d-ultraproduct} Let $\cu{D}$ be an ultrafilter  over a set $I$ and $\cu{M}_i$ be a general structure for each $i\in I$.
The \emph{pre-ultraproduct} $\prod^{\cu{D}}\cu{M}_i$ is the  general structure $\cu{M}'=\prod^{\cu{D}}\cu{M}_i$  such that:
\begin{itemize}
\item $M'=\prod_{i\in I} M_i$, the cartesian product.
\item For each constant symbol $c\in V$, $c^{\cu{M}'}=\langle c^{\cu{M}_i}\rangle_{i\in I}$.
\item For each $n$-ary function symbol $G\in V$ and $n$-tuple $\vec{a}$ in $M'$,
$$G^{\cu{M}'}(\vec{a})=\langle G^{\cu{M}_i}(\vec{a}(i))\rangle_{i\in I}.$$
\item For each $n$-ary predicate symbol $P\in V$ and $n$-tuple $\vec{a}$ in $M'$,
$$P^{\cu{M}'}(\vec{a})=\lim_{\cu{D}}\langle P^{\cu{M}_i}(\vec{a}(i))\rangle_{i\in I}.$$
\end{itemize}

The \emph{ultraproduct} $\prod_{\cu{D}}\cu{M}_i$ is the reduction of the pre-ultraproduct $\prod^{\cu{D}}\cu{M}_i$.
For each $a\in M'$ we also let $a_{\cu{D}}$ denote the equivalence class of $a$ under $\doteq^{\cu{M}'}$.
\end{df}

The following fact is the analogue for general structures of the fundamental theorem of {\L}o{\'s}.

\begin{fact} \label{f-Los}  Let $\cu{M}_i$ be a general structure for each $i\in I$, let $\cu{D}$ be an ultrafilter over $I$, and let
$\cu{M}=\prod_{\cu{D}}\cu{M}_i$ be the ultraproduct.  Then for each formula $\varphi$ and tuple $\vec{b}$ in the cartesian product $\prod_{i\in I} M_i$,
$$\varphi^{\cu{M}}(\vec{b}_{\cu{D}})=\lim_{\cu{D}}\langle\varphi^{\cu{M}_i}(\vec{b}_i)\rangle_{i\in I}.$$
\end{fact}

The special case of Fact \ref{f-Los} where each $\cu M_i$ has a symbol $=_i$ for the discrete metric was already stated and proved in [CK66].
The general case can be obtained from that special case by observing that if each $\cu M_i$ has vocabulary $V$ and $\cu N_i=(\cu M_i,=_i)$,
then $\prod_{\cu D}\cu M_i$ is the reduction of the $V$-part of $\prod_{\cu D}\cu N_i$.
Alternatively, Fact \ref{f-Los} can be proved from scratch by induction on the complexity of formulas, as is done in first order logic.

If $\cu{M}_i=\cu{M}$ for all $i\in I$, the ultraproduct $\prod_{\cu{D}}\cu{M}_i$ is called the \emph{ultrapower} of $\cu{M}$ modulo $\cu{D}$,
and is denoted by $\cu{M}^I/{\cu{D}}$.

\begin{cor}  \label{c-ultrapower}  For each $\cu{M}$ and ultrafilter $\cu{D}$, $\cu{M}$ is elementarily embeddable into $\cu{M}^I/{\cu{D}}$.
\end{cor}

\begin{cor}  \label{c-ultraproduct-isomorphism}  Suppose  $\cu M_i\cong\cu N_i$ for each $i\in I$, and $\cu D$ is an ultrafilter over $I$.
Then $\prod_{\cu D} \cu M_i\cong\prod_{\cu D} \cu N_i$.

\end{cor}

\begin{proof}  By Remark \ref{r-reduction}, for each $i\in I$ there is a map $h_i\colon \cu M_i\cong\cu N_i$.  Let $h\colon \prod_{i\in I}M_i\to\prod_{i\in I} N_i$
be the mapping such that for $a\in\prod_{i\in I} M_i$,
$h(a)=\langle h_i(a_i)\rangle_{i\in I}$.  For each continuous formula $\varphi(\vec v)$ and tuple $\vec a\in (\prod_{i\in I}M_i)^{|\vec v|}$, it follows from Fact
 \ref{f-Los} that
$$\varphi^{\prod^{\cu D}\cu M_i} (\vec a) = \lim_{\cu D}(\varphi^{\cu M_i}(\vec a_i))=\lim_{\cu D}(\varphi^{\cu N_i}(h_i(\vec a_i)))=
\varphi^{\prod^{\cu D}\cu N_i} (h(\vec a)).$$
Therefore $h\colon \prod^{\cu D}\cu M_i\to \prod^{\cu D}\cu N_i$.

By Fact \ref{f-Los}, if $b, c\in\prod_{i\in I} N_i$ and $b_i\doteq^{\cu N_i} c_i$
for each $i\in I$, then $b\doteq^{\prod^{\cu D}\cu N_i}c$.  It follows that $h\colon \prod^{\cu D}\cu M_i\cong \prod^{\cu D}\cu N_i$,
so by Remark \ref{r-reduction}, $\prod_{\cu D} \cu M_i\cong\prod_{\cu D} \cu N_i$.
\end{proof}

The following is proved in the usual way, using Fact \ref{f-Los}.

\begin{fact}  \label{f-compactness}  (Compactness)  If every finite subset of  $T$ has a general model, then $T$ has a general model.
\end{fact}

\subsection{Definability and Types}

We  write
\begin{itemize}
\item $r\dotplus s$ for $\min(r+s,1)$,
\item $r\dotle s$ for  $\max(r-s,0)$,
\item $r\dotle s\dotle t$ for $\max(r\dotle s,s\dotle t).$
\end{itemize}
Thus $r\le s$ if and only if $r\dotle s = 0$.
In the literature, $r\dotle s$ is sometimes written $r\dotminus s$.
Note that for any general structure $\cu M$, the following are equivalent:
\begin{itemize}
\item $\cu M\models \sup_{\vec x} [\varphi(\vec x)\dotle\psi(\vec x)\dotle\theta(\vec x)]$.
\item $(\forall \vec a\in M^{|\vec x|})[\varphi^{\cu M}(\vec a)\le\psi^{\cu M}(\vec a) \mbox{ and } \psi^{\cu M}(\vec a)\le\theta^{\cu M}(\vec a)]$.
\end{itemize}

In what follows, all formulas mentioned are understood to be in the vocabulary of a theory $T$.
Let $\vec x$ be a tuple of variables, and $\vec y$ be a finite or infinite sequence of variables,
where all the symbols $x_i$ and $y_j$ are distinct.
Given a formula $\theta(\vec x,\vec y)$, we let $\sup_{\vec y}\theta(\vec x,\vec y)$ denote the formula
$\sup_{\vec u}\theta(\vec x,\vec y)$ where $\vec u$ is the (necessarily finite) tuple of variables from $\vec y$ that occur freely in $\theta$.

We say that a sequence $\langle\varphi_m(\vec x,\vec y)\rangle_{m\in\BN}$ of formulas is \emph{Cauchy} in $T$
if for each $\varepsilon >0$ there exists $m$ such that for all $k\ge m$,
$$ T\models \sup_{\vec x}\sup_{\vec y}|\varphi_m(\vec x,\vec y)-\varphi_k(\vec x,\vec y)|\dotle\varepsilon.$$
\emph{Cauchy in} $\cu M$ means Cauchy in the complete theory $\Th(\cu M)$.

If $\langle\varphi_m(\vec x,\vec y)\rangle_{m\in\BN}$ is Cauchy in $T$, then for each general model $\cu{M}$ of $T$ there is a
unique mapping from $M^{|\vec x|}\times M^{|\vec y|}$ into $[0,1]$, denoted by $[\lim \varphi_m]^{\cu{M}}$, such that
$$(\forall \vec b\in M^{|\vec x|}) (\forall \vec c\in M^{|\vec y|})[\lim \varphi_m]^\cu{M}(\vec b,\vec c)=\lim_{m\to\infty}\varphi_m^{\cu{M}}(\vec b,\vec c).$$

We say that $\langle\varphi_m(\vec x,\vec y)\rangle_{m\in\BN}$ is \emph{exponentially Cauchy} in $T$ if whenever $m\le k$ we have
$$ T\models \sup_{\vec x}\sup_{\vec y}|\varphi_m(\vec x,\vec y)-\varphi_k(\vec x,\vec y)|\dotle 2^{-m}.$$
Note that every exponentially Cauchy sequence of formulas in $T$ is  Cauchy, and every Cauchy sequence in $T$ has an exponentially Cauchy subsequence.

\begin{df}   \label{d-definable}
We say that a mapping $\sa P\colon M^{|\vec x|}\times M^{|\vec y|}\to[0,1]$ is \emph{defined} by $\langle\varphi_m(\vec x,\vec y)\rangle_{m\in\BN}$
in a general structure $\cu M$,
and is \emph{definable in} $\cu M$, if $\langle\varphi_m(\vec x,\vec y)\rangle_{m\in\BN}$ is Cauchy in $\cu M$
and $\sa P=[\lim \varphi_m]^{\cu M}$.
\end{df}

Note that  for each general structure $\cu M$,
$$\delta_{\cu M}(\varphi,\psi):=\sup_{\vec x}\sup_{\vec y}|\varphi(\vec x,\vec y)-\psi(\vec x,\vec y)|^{\cu M}$$
is a pseudo-metric on the set of all formulas with free variables from $(\vec x,\vec y)$, and the above definition says that
$\langle\varphi_m(\vec x,\vec y)\rangle_{m\in\BN}$ is Cauchy in $T$
if and only if it is Cauchy with respect to $\delta_{\cu M}$ uniformly for all $\cu M\models T$.

We often consider the case where $\vec y$ is empty, or equivalently, where only finitely many of the variables in $\vec y$
actually occur in some $\varphi_m(\vec x,\vec y)$.  In that case, we have the notion of a Cauchy sequence of formulas
$\langle\varphi_m(\vec x)\rangle_{m\in \BN}$ and a
definable mapping $\sa P\colon M^{|\vec x|}\to[0,1]$
in $\cu M$.

We now introduce complete $n$-types.
For each theory $T$ and $n\in\BN$, a \emph{complete $n$-type over} $T$ is an $n$-type $p(\vec x)$  that is maximal with respect to being satisfiable in a general model of $T$.
$S_n(T)$ denotes the set of all complete $n$-types over $T$.  In particular, $S_0(T)$ is the set of all complete extensions of $T$.  For each
$p\in S_n(T)$ and formula $\varphi(\vec x)$ with $n$ free variables, we let $\varphi(\vec x)^p$ be the unique $r$ such that $|\varphi(\vec x)-r|\in p$.
The \emph{logic topology} on $S_n(T)$ is the topology whose closed sets are the sets of the form $\{p\in S_n(T)\colon \Gamma(\vec x)\subseteq p\}$
for some $n$-type $\Gamma(\vec x)$.  It follows from the Compactness Theorem that:

\begin{fact} \label{f-logic-topology-compact}
For each theory $T$ and $n\in\BN$,  the logic topology on $S_n(T)$ is compact.
\end{fact}

Given a general structure $\cu M$ and a set $A\subseteq M$, let $\cu M_A=(\cu M,a)_{a\in A}$.  The  \emph{complete type} of an
$n$-tuple $\vec{b}$ over $A$ in $\cu{M}$
is the set $\tp_{\cu M}(\vec{b}/A)$ of all formulas satisfied by $\vec{b}$ in $\cu M_A$.
The set $S_n(\Th(\cu M_A))$ of all complete $n$-types over $A$ realized in models of $\Th(\cu M_A)$
is denoted by $S_n^{\cu M}(A)$.
Thus the logic topology on $S_n^{\cu M}(A)$ is compact.

We say that a mapping $\sa P(\vec x,\vec y)$ is \emph{definable over} $A$ in $\cu M$ if $\sa P$ is definable in $\cu M_A$.

The following lemma gives a relationship between $\sa P(\vec x)$ being definable over a countable sequence of parameters,
and $\sa P(\vec x,\vec y)$ being definable without parameters.

\begin{lemma}  \label{l-definable-iff}
Let $B\subseteq M$.  A mapping $\sa P(\vec x)$ is definable over $B$ in $\cu M$ if and only if there is a  countable sequence $\vec b$ of elements of $B$
and an exponentially Cauchy sequence  $\langle\varphi(\vec x,\vec y)\rangle$ of formulas in $\cu M$ such that
 for all $\vec a\in M^{|\vec x|}$ we have $\sa P(\vec a)=\sa Q(\vec a,\vec b)$ where $\sa Q$ is the mapping defined by $\langle\varphi(\vec x,\vec y)\rangle$ in $\cu M$.
\end{lemma}

\begin{proof}
 It is clear that if there is such a sequence of formulas, then $\sa P(\vec x)$ is definable over  $B$ in $\cu M$.
Suppose $\sa P(\vec x)$ is definable over  $B$ in $\cu M$, by a sequence $\langle\theta_m(\vec x)\rangle_{m\in\BN}$
with parameters in $B$.  By taking a subsequence if necessary, we may insure that for some sequence $\vec b$
of elements of $B$, each $\theta_m$ has parameters from $\vec b$, and for each $m$ we have
\begin{equation}  \label{e-theta}
\cu M\models\sup_{\vec x}|\theta_m(\vec x,\vec b)-\theta_{m+1}(\vec x,\vec b)|\le 2^{-(m+1)}.
\end{equation}
We now use the forced convergence trick in [BU] to find a new sequence of formulas
$\varphi_m(\vec x,\vec y)$ in such a way that
\begin{equation} \label{e-varphi}
\cu M\models\sup_{\vec x}\sup_{\vec y}|\varphi_m(\vec x,\vec y)-\varphi_{m+1}(\vec x,\vec y)|\le 2^{-(m+1)},
\end{equation}
and that $\varphi_m^{\cu M}(\vec a,\vec b)=\theta_m^{\cu M}(\vec a,\vec b) $ whenever $\vec a\in M^{|\vec x|}$ and $m\in\BN$.
To do that we inductively define $\varphi_0=\theta_0$, and
$$\varphi_{m+1}=\max(\varphi_m-2^{-(m+1)},\min(\varphi_m+2^{-(m+1)},\theta_{m+1})).$$
Condition (\ref{e-varphi}) implies that $\varphi_m(\vec x,\vec y)$ is exponentially Cauchy in $\cu M$,
and  defines a mapping $\sa Q$ in $\cu M$ such that $\sa P(\vec a)=\sa Q(\vec a,\vec b)$  for all $\vec a\in M^{|\vec x|}$.
\end{proof}

The following result is the analogue for general structures of Theorem 9.9 of [BBHU].

\begin{prop}  \label{p-definable-continuous}
Let $\cu M$ be a general structure, $A\subseteq M$, and  $\sa P\colon M^n\to[0,1]$.  The following are equivalent:
\begin{itemize}
\item[(i)] $\sa P$ is a definable predicate over $A$ in $\cu M$.
\item[(ii)]
There is a unique function $\Psi\colon S_n^{\cu M}(A)\to[0,1]$ such that
$\sa P(\vec c)=\Psi(\tp_{\cu M}(\vec c/A))$ for all $\vec c\in M^n$.  Furthermore, $\Psi$ is continuous in the logic topology on $S_n^{\cu M}(A)$
\end{itemize}
\end{prop}

\begin{proof}
The proof that (ii) $\Rightarrow$ (i) is the same as the proof of Theorem 9.9 in [BBHU], which did not use a metric on $\cu M$.

(i) $\Rightarrow$ (ii):  Suppose that $\sa P$ is defined by $\langle\varphi_m\rangle_{m\in\BN}$ over $A$ in $\cu M$.
Then $\langle\varphi_m\rangle_{m\in\BN}$ is Cauchy in $\Th(\cu M_A)$.  There is a unique $\Psi\colon S_n^{\cu M}(A)\to[0,1]$
such that $\Psi(\tp_{\cu M}(\vec c/A))=[\lim \varphi_m]^{\cu M_A}(\vec c)=\sa P(\vec c)$ for each $\vec c\in M^n$.
For each closed interval  $I\subseteq[0,1]$, there is
a $n$-type $\Gamma_I(\vec x)$ with parameters in $A$ such that for all $\vec c\in M^n$, $\Psi(\tp_{\cu M}(\vec c/A))\in I$ if and only if
$\cu M_A\models \Gamma_I(\vec c)$.
Therefore $\Psi^{-1}(I)$ is closed in $S_n^{\cu M}(A)$, so $\Psi$ is continuous.
\end{proof}

\begin{rmk}  Let $A\subseteq M$ and $\vec b$ in $M^n$.
\begin{itemize}
\item If $\cu N\succ\cu M$ then $\tp_{\cu M}(\vec b/A)=\tp_{\cu N}(\vec b/A)$.
\item  $\varphi^{\cu{M}}(\vec b,A)=r$ if and only if
$|\varphi(\vec x,A)-r|$ belongs to $\tp_{\cu M}(\vec b/A)$.
\end{itemize}
Using the compactness theorem, we have
\begin{itemize}
\item If $\cu N\succ\cu M$ then $S_{n}^{\cu M}(A)=S_{n}^{\cu N}(A)$.
\item There exists $\cu N\succ\cu M$ in which  every type in $\bigcup_n S_n^{\cu M}(A)$ is realized.
\end{itemize}
\end{rmk}

\subsection{Saturated and Special Structures}

For an infinite cardinal $\kappa$, we say that a general structure $\cu{M}$ is \emph{$\kappa$-saturated} if for every set $A\subseteq M$ of
cardinality $|A|<\kappa$, every set of formulas in the vocabulary of $\cu M$ with one free variable and parameters from $A$
that is finitely satisfiable in $\cu{M}_A$ is satisfiable in $\cu{M}_A$.

\begin{rmk}  \label{r-saturated-types}
If $\cu M$ is $\kappa$-saturated, $A\subseteq M$, and $|A|<\kappa$, then every complete type $p\in S_{n}(A)$ is realized in $\cu M$.
\end{rmk}

\begin{rmk}  \label{r-reduction-sat}
$\cu M$ is $\kappa$-saturated if and only if the reduction of $\cu M$ is $\kappa$-saturated.
\end{rmk}

\begin{df}
By a \emph{special cardinal} we mean a cardinal $\kappa$ such that $2^\lambda\le\kappa$ for all $\lambda<\kappa$.
We say that $\cu{M}$ is \emph{special} if $|M|$ is an uncountable special cardinal and  $\cu{M}$ is the union of an elementary chain of structures
$\langle \cu{M}_\lambda\colon \lambda<|M|\rangle$ such
that each $\cu{M}_\lambda$ is $\lambda^+$-saturated.  $\cu{M}$ is \emph{$\kappa$-special} if $\kappa$ is special and $\cu{M}$ is the reduction of a special
structure of cardinality $\kappa$.
\end{df}

Note that every strong limit cardinal is special, and if $2^\lambda=\lambda^+$ then $2^\lambda$ is special.
Thus every $\kappa$-special structure is reduced and has cardinality $\le\kappa$.
 A $\kappa$-special structure may have cardinality less than $\kappa$.  For example, every metric structure with a compact metric is $\kappa$-special for
 every special cardinal $\kappa$ but has cardinality at most $2^{\aleph_0}$.

\begin{rmk}  \label{r-special-expansion}  If $\cu{M}$ is $\kappa$-special and  $V^0\subseteq V$, then the reduction of the $V^0$-part of $\cu{M}$ is $\kappa$-special.
\end{rmk}

\begin{rmk}  \label{r-special-inaccessible}  If $\kappa$ is an uncountable inaccessible cardinal, then $\cu M$ is $\kappa$-special if and only if
$\cu M$ is a reduced $\kappa$-saturated structure of cardinality $\kappa$.
\end{rmk}

\begin{fact}  \label{f-special-unique}  (Uniqueness Theorem for Special Models) If $T$ is complete and $\cu{M}, \cu{N}$ are $\kappa$-special models of $T$,
then $\cu{M}$ and $\cu{N}$ are isomorphic.
\end{fact}

An easy consequence is:

\begin{fact}  \label{f-special-isomorphism}  Suppose $\cu M$ is $\kappa$-special, $A\subseteq M$, and $|A|<$ cofinality of $\kappa$.  Then
for all tuples $\vec b, \vec c$ in $M$, $\tp_{\cu M}(\vec b/A)=\tp_{\cu M}(\vec c/A)$ if and only if $\cu M$ has an automorphism that
sends $\vec b$ to $\vec c$ and is the identity on $A$.
\end{fact}

 The following result is proved using the Compactness, Downward L\"{o}wenheim-Skolem, and
Elementary Chain Theorems.

\begin{fact} \label{f-special-exist}  (Existence Theorem for Special Models)
If $\kappa$ is a special cardinal and $\aleph_0+|V|<\kappa$, then every reduced structure $\cu{M}$ such that
$|M|\le\kappa$ has a $\kappa$-special elementary extension.
\end{fact}

 The statement of Fact \ref{f-special-exist} above differs slightly from the
 corresponding result in [CK66], because we do not have a symbol for the equality relation here.

\subsection{Pre-metric Structures}

 A \emph{metric signature} $L$ over $V$ specifies a distinguished binary predicate symbol $d\in V$ for distance, and equips each predicate
 or function symbol $S\in V$ with a modulus of uniform continuity   $\triangle_S\colon(0,1]\to(0,1]$ with respect to $d$.

In the literature on metric structures, one usually fixes a metric signature $L$ once and for all, but here we focus on general structures that are not equipped with metric signatures, and often consider many metric signatures at the same time.  For that reason, we will officially define a pre-metric structure to be pair consisting of a general structure and a metric signature.

A \emph{pre-metric structure} $\cu M_+=(\cu M,L)$  consists of a general structure $\cu{M}$ with  vocabulary $V$,
and a metric signature $L$ over $V$, such that $d^{\cu{M}}$ is a pseudo-metric
on $M$, and for each predicate symbol $P$ and function symbol $F$ of arity $n$, and for all $\vec{a},\vec{b}\in M^n$ and  $\varepsilon\in(0,1]$,
$$  \max_{k\le n}d^\cu{M}(a_k,b_k)<\triangle_P(\varepsilon)\Rightarrow|P^\cu{M}(\vec{a})-P^\cu{M}(\vec{b})|\le \varepsilon.$$ and
$$  \max_{k\le n}d^\cu{M}(a_k,b_k)<\triangle_F(\varepsilon)\Rightarrow d^\cu{M}(F^\cu{M}(\vec{a}),F^\cu{M}(\vec{b}))\le \varepsilon.$$
A \emph{metric structure}  is a pre-metric structure $(\cu M,L)$ such that $(M,d^\cu{M})$ is a complete metric space.
The paper [BBHU] emphasized metric structures, but in this paper we will focus more on pre-metric than metric structures.

Given a pre-metric structure $\cu M_+=(\cu M,L)$, we will call $\cu M$  the \emph{downgrade of} $\cu M_+$,
and call $\cu M_+$  the \emph{upgrade of} $\cu M$ to $L$.
Note that two different pre-metric structures can have the same downgrade,
because they may have different metric signatures.

We will slightly abuse notation and say that $\cu M$ \emph{is a pre-metric (or metric) structure for} $L$ when $(\cu M,L)$ is a pre-metric (or metric) structure.
We say that a pre-metric structure $\cu M_+$ for $L$ is $\kappa$-saturated if its downgrade is $\kappa$-saturated.  Similarly for other properties
of general structures, such as being reduced, or being an ultraproduct of a family of structures.

 We remind the reader that every pre-metric structure for $L$ has a unique completion up to isomorphism,
that this completion is a metric structure for $L$, and that every pre-metric structure is elementarily embeddable in its completion.

\begin{rmk} \label{r-saturated-metric}  Every pre-metric structure for $L$ that is reduced and $\aleph_1$-saturated
 is a metric structure for $L$.
\end{rmk}

In [BBHU], when $\cu{M}_+$ is a metric structure, a mapping from $M^n$ into $[0,1]$ that is uniformly continuous with respect to $d^{\cu M}$
is called a \emph{predicate on $\cu{M}_+$}.

We will make frequent use of the following result from [BBHU].

\begin{fact} \label{f-t3.5} (Theorem 3.5 in [BBHU])  For every metric signature $L$  and formula $\varphi(\vec x)$ in the vocabulary of $L$,
there is a function $\triangle_\varphi\colon(0,1]\to(0,1]$ that is a modulus of uniform continuity for the mapping $\varphi^{\cu M}\colon M^{|\vec x|}\to[0,1]$
in every pre-metric structure $\cu M_+$ for $L$.
\end{fact}

\begin{fact}  \label{f-metric-theory}  (See [Ca], page 112, and Definition 2.4 in [AH].)
For each metric signature $L$ over $V$, there is a theory $\met(L)$ whose general models are exactly
the downgrades of pre-metric structures for $L$.  Each sentence in $\met(L)$ consists of finitely many $\sup$ quantifiers followed by a quantifier-free formula.
\end{fact}

The sentences in $\met(L)$ formally express that $d$ is a pseudo-metric, and that the functions and predicates in $V$ respect the moduli of uniform continuity for $L$.

\begin{cor}  Every ultraproduct of pre-metric structures for $L$ is a pre-metric structure for $L$.
\end{cor}

\begin{proof}  By Facts \ref{f-Los} and  \ref{f-metric-theory}.
\end{proof}

A \emph{metric theory} $(T,L)$ consists of a metric signature $L$ and a set $T$ of sentences in the vocabulary of $L$ such that $T\models\met(L)$.
A pre-metric (or metric) \emph{model} of a metric theory $(T,L)$ is a pre-metric (or metric) structure $\cu M_+=(\cu M,L)$ for $L$ such that $\varphi^{\cu M}=0$
for each sentence $\varphi\in T$.

When $(T,L)$ is a metric theory, we also say that \emph{$T$ is a metric theory with signature $L$}.
Note that for every metric theory $(T,L)$, pre-metric structure $(\cu M,L)$, and continuous sentence $\varphi$, we have
$T\models \cu M$ iff $(T,L)\models (\cu M,L)$, and $T\models\varphi$ iff $(T,L)\models\varphi$. Also, if $(\cu M,L)$ is a pre-metric model of a metric theory $(T,L)$,
then $\cu M$ is a general model of $T$.

Part (i) of the next lemma shows that when $\cu M_+$ is a metric structure, the definition of a definable predicate here agrees with the notion of a definable predicate in [BBHU].
\begin{lemma}  \label{l-def-predicate-uniform}
\noindent\begin{itemize}
\item[(i)] In every pre-metric structure $\cu M_+$ for $L$, every mapping that is definable in the sense of Definition \ref{d-definable} is uniformly continuous.
\item[(ii)] For every metric theory $T$ with signature $L$,  and sequence of formulas
$\langle \varphi_m(\vec x)\rangle_{m\in\BN}$ that is Cauchy in $T$, there is a function $\triangle_{\langle\varphi\rangle}$
that is a modulus of uniform continuity for the definable predicate $[\lim\varphi_m]^{\cu M}$ in every pre-metric model $\cu M_+$ of $T$.
\end{itemize}
\end{lemma}

\begin{proof}   (ii) trivially implies (i).  We prove (ii).  Given $\varepsilon >0$, take the least $m$ so that  for all $n\ge m$,
$$T\models\sup_{\vec x}|\varphi_m(\vec x)-\varphi_n(\vec x)|\dotle\varepsilon/3.$$
Let $\triangle_{\langle\varphi\rangle}(\varepsilon)=\triangle_{\varphi_m}(\varepsilon/3)$ .  Suppose $\cu M_+$ is a pre-metric model of $ T$,
and $\max_{k\le n}d^\cu{M}(a_k,b_k)<\triangle_{\langle\varphi\rangle}(\varepsilon)$.  Then
$$ |[\lim\varphi_m]^{\cu M}(\vec a)-[\lim\varphi_m]^{\cu M}(\vec b)|\le$$
$$|[\lim\varphi_m]^{\cu M}(\vec a)-\varphi_m^{\cu M}(\vec a)| + |\varphi_m^{\cu M}(\vec a)-\varphi_m^{\cu M}(\vec b)| +|\varphi_m^{\cu M}(\vec b)-[\lim\varphi_m]^{\cu M}(\vec b)|\le$$
$$\varepsilon/3 + \varepsilon/3 +\varepsilon/3 =\varepsilon.$$
\end{proof}

The next fact follows at once from Theorem 3.7 in [BBHU].

\begin{fact}  \label{f-reduction-distance}
If $\cu M_+$ is a pre-metric structure, $a\doteq^{\cu{M}}b$ if and only if $d^{\cu{M}}(a,b)=0$.
\end{fact}

As mentioned above, in [BBHU], the notion of a reduction of a pre-metric structure $\cu M_+$ was defined as the structure obtained by identifying elements $x,y$
when $d(x,y)=0$, rather than when $x\doteq y$.  So the reduction of a pre-metric structure
as defined here is the same as the reduction of a pre-metric structure as defined in [BBHU].  It  follows that
the ultraproduct of a family of metric structures with signature $L$ as defined here is the same as
the ultraproduct of a family of metric structures with signature $L$ as defined in [BBHU].

\begin{fact}  \label{f-metric-ultraproduct}
(By Proposition 5.3 in [BBHU].)  Every ultraproduct of metric structures for $L$ is a metric structure for $L$.
\end{fact}

\begin{cor}  \label{c-metric-compactness}  (Metric Compactness Theorem)  If $T$ is a metric theory with signature $L$, and every finite subset of $T$ has a metric model,
then $T$ has a metric model.
\end{cor}

In [BBHU], the definitions of a pre-metric structure and of an ultraproduct
 were  designed to insure that Fact \ref{c-metric-compactness} and the Metric Compactness Theorem hold.
There were pitfalls to avoid in extending the notion of an ultraproduct to general structures.  For example, if one adds
a discontinuous predicate or function to a pre-metric structure and tries to define an ultraproduct by identifying elements at distance $0$
from each other, the added predicate or function symbol would be undefined in the ultraproduct.

Such pitfalls were avoided here by defining the reduction of a general structure to be the general structure formed by identifying $x$ with $y$
when $x\doteq^{\cu M} y$.  Using that notion of reduction, we obtained a well-defined notion of ultraproduct that coincides with the
notion in [BBHU] for  metric structures, satisfies the theorem of {\L}o{\'s} (Fact \ref{f-Los}) in all cases, and leads to the
Compactness Theorem for general structures.

\subsection{Some Variants of Continuous Model Theory}

In this subsection we discuss some cases from the literature where variants of continuous model theory have been developed
in order to study special classes of general structures that share some features with pre-metric structures.  In each of these cases,
one can instead work with the model theory of general structures as developed here.

\begin{ex} \label{e-Eagle} (Infinitary Continuous Logic) Christopher Eagle, in [Ea14] and [Ea15], developed an infinitary logic $\cu{L}_{\omega_1\omega}$
that is analogous to continuous logic, but has much greater power of expression and fails to satisfy the compactness theorem.
The  $\cu{L}_{\omega_1\omega}$-formulas in a vocabulary $V$ are built from the atomic formulas using the connectives and quantifiers of continuous logic
and, in addition, the operations $\sup_m\varphi_m$ and $\inf_m\varphi_m$ whenever $\langle\varphi_m\rangle_{m\in\BM}$ is a sequence of formulas
in which only finitely many variables occur freely.  Given an $\cu L_{\omega_1\omega}$-formula $\varphi(\vec x)$ and a pre-metric structure $\cu M$,
the truth value function $\varphi^{\cu M}\colon M^{|\vec x|}\to[0,1]$ is defined in [Ea15] as one would expect by induction on complexity.

We adopt here the same definition for the truth value $\varphi^{\cu M}$ of an $\cu{L}_{\omega_1\omega}$-formula in a general structure $\cu M$.  Then
$(\cu{M},\varphi^{\cu{M}})$ is a general structure whose vocabulary has an extra $|\vec x|$-ary predicate symbol. However, even if $\cu M$ is a metric structure,
$(\cu{M},\varphi^{\cu{M}})$ is not necessarily a metric structure, and $\varphi^{\cu{M}}$  may even be discontinuous.  We will revisit this concept in Subsection \ref{s-infinitary}.
\end{ex}

\begin{ex} (Geodesic Logic) The paper [Cho17] introduced a variant of continuous model theory, called \emph{geodesic logic}, whose structures
are pre-metric structures with extra functions that are possibly discontinuous.  He used this to obtain approximate fixed point results
and metastability results for certain discontinuous functions.

A \emph{geodesic signature} $G$ consists of a metric signature $H$ over a vocabulary $V$, and a set $Y$ of extra symbols for
possibly discontinuous functions, including a set $\{L_t\colon t\in[0,1]\}$ of  binary function symbols (the linear structure).
 A \emph{geodesic structure} with  signature $G$ is a general structure in our present sense,
that has a vocabulary $V\cup Y$ whose $V$-part is
a pre-metric structure with signature $H$, that  satisfies the following sentences for each $s,t\in[0,1]$:
\begin{eqnarray}
 \sup_x\sup_y[|d(L_s(x,y),L_t(x,y))- |s-t|d(x,y)|\dotle 0], \label{geo1}\\
 \sup_x\sup_y[d(L_t(x,y),L_{1-t}(y,c))\dotle 0]. \label{geo2}
 \end{eqnarray}
Thus a geodesic structure is just a general model of the theory
$$T_G=\met{H}\cup\{(\ref{geo1}),(\ref{geo2})\}.$$
For example, in our setting, Theorem 6.6 (a) of [Cho17] can be stated as follows:

\emph{Let $\lambda\in[0,1]$ and let $G$ be a geodesic signature with distance predicate $d$, and an extra unary function symbol $F$.
Let $\cu M$ be a general model of the theory
\begin{equation} \label{geo}
T_G\cup\{\sup_x[d(Fx,FL_\lambda(x,Fx))\dotle d(x,Fx)]\}.
\end{equation}
In $\cu M$,  when $x_{n+1}=L_\lambda(x_n,Fx_n)$ for all $n$, we have $\lim_{n\to\infty}d(x_n,x_{n+1})=0$.}

The compactness theorem can then be used to get a uniform metastability bound  across all general models of (\ref{geo}) \ (Theorem 6.6 (b) of [Cho17]).
 \end{ex}

\begin{ex}  \label{e-manysorts}  (Sorted Vocabularies)
Many-sorted metric structures are prominent in the literature.  Here we will work exclusively with general structures,
as defined at the beginning of this section, but introduce the notions of a sorted vocabulary,
and of a general structure that respects the sorts in that vocabulary.  A general structure that respects
sorts cannot be a pre-metric structure.
We have the flexibility of starting with a general structure that respects sorts, and adding a new unsorted distance predicate to
form a pre-metric structure that does not respect sorts.
As we will see later, that flexibility will be useful when we consider such topics as unbounded metrics and imaginary elements.
\end{ex}

A \emph{sorted vocabulary} $W$ consists of a set of sorts, sets of finitary predicate  and function symbols,
and a set of constant symbols.  Each argument of a predicate or function symbol is equipped with a sort, and each function and
constant symbol is equipped with a sort for its value. Moreover, $W$  contains a unary predicate symbol $U_\BS$ associated with each sort $\BS$,
and a constant symbol $u$ (for ``unsorted'') that has no sort.

  We say that a general structure $\cu M$ with vocabulary $W$ \emph{respects sorts} if:
\begin{itemize}
\item  The universe $M$ of $\cu M$ contains a family of pairwise disjoint non-empty universe sets $\BS^{\cu M}$ corresponding to the sorts $\BS$ of $W$.
\item For each sort $\BS$ of $W$ and $a\in M$,  $U_\BS^{\cu M}(a)=0$ when $a\in \BS^{\cu M}$, and $U_\BS^{\cu M}(a)=1$ when $a\notin \BS^{\cu M}$.
\item For each $k$-ary predicate symbol $P$ of $W$ and $\vec a\in M^k$,  $P^{\cu M}(\vec a)=1$ when at least one argument is of the wrong sort.
\item For each $k$-ary function symbol $F$ of $W$ with value sort $\BS$ and $\vec a\in M^k$, $F^{\cu M}(\vec a)$  belongs to $\BS^{\cu M}$ when each argument is of the
correct sort, and is $u^{\cu M}$ when at least one argument is of the wrong sort.
\item The value $c^{\cu M}$ of each constant symbol $c$ of sort $\BS$ is an element of $\BS^{\cu M}$.
\item  The constant $u^{\cu M}$ does not belong to any of the sets $\BS^{\cu M}$.
\end{itemize}

The unsorted constant symbol $u$ serves two purposes: It allows one to interpret
each function symbol as a total function rather than as a partial function from a product of sorts to a sort, and it insures that each structure that
respects sorts has an unsorted element.

\begin{lemma}  \label{l-sortless}
Let $W$ be a sorted vocabulary and let $\cu M$ be a general structure with vocabulary $W$.
\begin{itemize}
\item[(i)] If $\cu M$  respects sorts, then the reduction of $\cu M$ also respects sorts.
\item[(ii)]  If $\cu M$ is reduced and respects sorts, then $u^{\cu M}$ is the unique element of $M$ that does not belong to $\BS^{\cu M}$ for any sort $\BS$ of $W$.
\item[(iii)]  There is a set $\rs(W)$ of sentences such that for every reduced structure $\cu M$ with vocabulary $W$, $\cu M$ respects sorts
if and only if $\cu M\models\rs(W)$.
\end{itemize}
\end{lemma}

\begin{proof}  (i) is clear.

(ii): Let $a$ be an element of $M$ that does not belong to $\BS^{\cu M}$ for any sort $\BS$ of $W$.  Then for any tuple $\vec c$ in $M$ and atomic formula
$\varphi(x,\vec c)$ in which the variable $x$ occurs, we have $\varphi^{\cu M}(a,\vec c)=1$ and $\varphi^{\cu M}(u^{\cu M},\vec c)=1$.
Thus $a\doteq^{\cu M} u^{\cu M}$, so (ii) holds.

(iii): Each of the requirements for respecting sorts can be expressed in $\cu M$ be a set of sentences.  For example, the requirement
that two sorts $\BS_1$ and $\BS_2$ are pairwise disjoint is expressed by the sentence
$$\sup_x(1\dotle\max(U_{\BS_1}(x),U_{\BS_2}(x))).$$
By Remark \ref{r-doteq}, $F(x,\vec y)\doteq u$ can be expressed in $\cu M$ by a set  of formulas.
One can use that to show that when $\cu M$ is reduced, the requirement that $F^{\cu M}(x,\vec y)=u^{\cu M}$ when $x$ does not have sort
$\BS$ can be expressed by a set of sentences.
The other requirements are similar.  With a bit more work, one can take each sentence in $\rs(W)$ to be a
finite set of $\sup$ quantifiers followed by a quantifier-free formula.
 \end{proof}

\begin{df}  \label{d-sorted-metric}
Given a sorted vocabulary $W$,
a \emph{sorted metric signature} $L$ over  $W$ consists of a distinguished distance predicate symbol $d_\BS$
for each sort $\BS$ of $W$, and a modulus of uniform continuity for each predicate and function symbol with respect to these distance predicates.
By a \emph{sorted metric (or pre-metric) structure}
 with the sorted signature $L$ we mean a reduced general structure $\cu M$ with vocabulary $W$ that respects sorts,
 such that in $\cu M$, each $d_\BS$ is a complete metric (or metric) on $\BS$, and each symbol of $W$ satisfies the modulus of uniform
 continuity given by $L$ with respect to these metrics.
\end{df}

We will return to sorted metric structures in Subsection \ref{s-manysorted}.

\begin{ex} \label{e-bounded}  (Bounded Continuous Logic)

The version of continuous logic as defined in Section 2 of [BBHU], which we will here call \emph{bounded continuous logic},
 is slightly broader than the $[0,1]$-valued version.  (See also [BY09] and [Fa]).
 A \emph{bounded vocabulary} is a vocabulary $V$ equipped with a bounded real interval $[r^P,s^P]$ where $r^P<s^P$, for each predicate symbol $P\in V$.
The metric signatures, continuous formulas, and metric structures with the bounded vocabulary $V$ are as in the $[0,1]$-valued case, except that  the bounds
of uniform continuity are maps from $(0,\infty)$ into $(0,\infty)$,
the distinguished metric $d$ is a metric with values in $[0,s^d]$, the logical connectives are continuous functions from $\BR^n$ into $\BR$, and the
predicates $P$ have truth values in the bounded interval $[r^P,s^P]$.  The truth values of formulas in a bounded metric structure are defined by induction
on complexity of formulas in the usual way, and the truth values are in $\BR$.
Many-sorted bounded metric structures are treated in a similar way.

One can regard $[0,1]$-valued continuous logic as the special case of bounded continuous logic where
where each predicate symbol $P$, including the distinguished metric $d$, has the interval $[0,1]$.
The model theory of $[0,1]$-valued metric structures carries over to bounded metric structures in a routine way.
As pointed out in [BBHU], there is no real loss of generality in working  with $[0,1]$-valued continuous logic, and for
simplicity, [BBHU] develops things in detail only for the $[0,1]$-valued case.

One can convert any
bounded metric structure $\cu M$ to a $[0,1]$-valued metric structure $\cu M^1$ by replacing each
predicate  $P^\cu M$ by the $[0,1]$-valued predicate $P^{\cu M^1}$ where
$$P^{\cu M^1}(\vec v)=(P^\cu M(\vec v)-r^P)/(s^P-r^P).$$
It then turns out that a predicate $P\colon[0,1]^n\to[0,1]$ is definable in $\cu M$ if and only if it is definable in $\cu M^1$.
In that sense, the $[0,1]$-valued language has the same expressive power as the bounded language.  But in many applications, the formulas
in the bounded language will be easier to understand than the corresponding formulas in the $[0,1]$-valued language.

In keeping with the perspective of this paper, one can also consider bounded general structures with a given bounded vocabulary, by simply dropping the
requirement that there is a distinguished metric and a signature giving moduli of uniform continuity.
However, we retain the bounded vocabulary that equips each predicate symbol $P$ with a bounded real interval.
Formulas and truth values are defined in the same way for bounded general structures as they are for bounded metric structures in [BBHU].

\end{ex}

\begin{ex}  \label{e-unbounded} (Unbounded Continuous Logic)
Many important mathematical structures, such as Banach spaces, valued fields, and $C^*$-algebras, are structures with an unbounded complete metric $d$,
functions that are uniformly continuous with respect to $d$, predicates with truth values in $\BR$ that are uniformly continuous with respect to $d$, and constant symbols.
We will call such structures   \emph{unbounded metric structures}.

The inductive definition of the truth value of a formula in a bounded metric structure does not automatically carry over to unbounded metric structures, because the
$\sup$ of a set of values may
not exist in $\BR$.  In the literature, there are at least three approaches to the study of an unbounded metric structure $\cu N$ by applying the existing model theory
to some metric structure that is associated with $\cu N$.

One approach is to add a constant symbol $0$ for a distinguished element  of $\cu N$, and look at the many-sorted  bounded metric structure $\cu M$ that has a sort $\BS_m$
for the closed $m$-ball around $0$ for each $0<m\in\BN$, and, for each $0<m<k$, a function $i_{mk}$ of sort $\BS_m\to\BS_k$
for the inclusion map.  For simplicity, suppose $\cu N$ has no function symbols.  For every constant symbol $c$ and $n$-ary predicate symbol $P$ of $\cu N$, and every sort
$\BS_m$,  $\cu M$ will have a constant symbol $c_m$ and
an $n$-ary predicate symbol $P_m$ of sort $\BS_m$.  The uniform continuity property of  $\cu N$ guarantees that for each predicate symbol $P$ of $\cu N$,
and each $m>0$, $P^\cu N$ maps the closed $m$-ball around $0$  into a bounded interval $[r^P_m,s^P_m]$.  For each predicate symbol $P$, the bounded vocabulary of
$\cu M$ will be equipped with both the arity of $P$ and the bounded interval $[r^P_m,s^P_m]$.

In [Fa], $C^*$-algebras are treated as many-sorted bounded metric structures as in the preceding paragraph.
In order to make things work properly, axioms are explicitly added to guarantee that
the inclusion map sends the sort $\BS_m$ onto the closed $m$-ball in $\BS_k$.
In [BBHU], Sections 15 and 17, Hilbert spaces and Banach lattices are treated in a similar way
(see also Remark 4.6 in [BU], and [BY09]).  In those cases, in any metric model of the many-sorted theory,
the inclusion map sends the sort $\BS_m$ onto the closed $m$-ball around $0$ in $\BS_k$.

A second, and simpler, approach is to look at the single-sorted $[0,1]$-valued metric structure
formed by restricting the universe of $\cu N$ to the unit ball around $0$, and normalizing each predicate symbol so that it takes truth values in $[0,1]$.
[BBHU] show that for Hilbert and $L^p$ spaces, the many-sorted  metric structure has essentially the same model-theoretic properties as the single-sorted
structure on the unit ball, and are able to successfully simplify things by working exclusively with the unit ball.

However, as pointed out in [BY08] and [BY14],
there are other cases, such as in non-archimedean valued fields, where neither of the above approaches is adequate.
The many-sorted structure $\cu M$ may have definable predicates that would not be considered definable in the original structure $\cu N$.
For example, if $m<k$, the distance $\dist(x,\sa B)$ from a point $x\in\BS_{k}$ to the set $\sa B=\range(i_{mk})$ might not be definable in $\cu N$, but is
always definable in $\cu M$ by the formula $\inf_y d_k(x,i_{mk}(y))$.

To deal with that problem, Ben Yaacov [BY08]
developed a variant of continuous logic with an unbounded metric and modified notions of formula and ultraproduct.
That logic has a Lipschitz \emph{gauge function} $\nu\colon\cu N\to[0,\infty)$
that is thought of as the size of an element, and determines the closed $m$-balls $B_m=\{x\colon\nu(x)\le m\}$.
In most examples, the gauge $\nu(x)$ is $d(x,0)$, the distance from $x$ to a distinguished point $0$.  The universe of the modified pre-ultraproduct is the set of elements $x$
of the usual pre-ultraproduct such that $\langle\nu_i(x_i)\rangle_{i\in I}$ is bounded.  This means that the gauge of $x$ is finite,
which in most examples means that $d(x,0)$ is finite.
Using that logic, [BY08] converted an unbounded metric structure
$\cu N$ into an ordinary  $[0,1]$-valued metric structure $\cu N^\infty$  via his \emph{emboundment} construction.
The intuitive idea is to add a point at infinity to $\cu N$, and to add a new metric to form a metric structure $\cu N^\infty$
such that the theories of $\cu N$ and $\cu N^\infty$ have the same model-theoretic properties.
The reason for doing that is to
study the unbounded metric structure $\cu N$ by applying ordinary continuous model theory to $\cu N^\infty$.

The logic developed in [BY08] is in many ways equivalent to the logic initiated in 1976 by Henson [He], and developed further by Henson and Iovino [HI],
and  Dueñez and Iovino [DI].  That approach uses a similar notion of unbounded metric structure and modified ultraproduct as [BY08],
but with a different notion of formula, and a semantics based on approximate truth.  Another form of unbounded continuous logic, that uses
a similar notion of unbounded metric structure and modified ultraproduct but a different notion of formula, is developed in [Lu].

The bounded continuous logic of [BBHU] is better suited for applications than either of the logics developed in [BY08] or in [HI] and [DI], because  the notions of
formula and truth value  are simpler and more natural, and the formulas are easier to understand, in [BBHU] than in the other approaches.  For that reason,
it desirable to find a way to use bounded continuous logic to study unbounded metric structures when possible.

We will return to unbounded metric structures in Subsection \ref{s-unbounded} below.
\end{ex}

\section{Turning General Structures into Metric Structures} \label{s-turning}

In this section we define the key notion of a pre-metric expansion, and show that for every theory $T$
there exists a pre-metric expansion of $T$.

\subsection{Definitional Expansions}

In this subsection we introduce definitional expansions of a theory, and in the next subsection we will introduce pre-metric expansions as a special case.

\begin{df} \label{d-definitional-expansion}  Let $T$ be a general theory in a vocabulary $V$,  let $D$ be a  predicate symbol
that may or may not belong to $V$, and let $V_D=V\cup\{ D\}$.
 A \emph{definitional expansion} of  $T$ over $V_D$ is a  theory $T_e$ with vocabulary $V_D$
such that for some sequence $\langle d\rangle=\langle d_m(\vec x)\rangle_{m\in\BN}$ of formulas of $V$ that is Cauchy in $T$:
\begin{itemize}
\item[(i)]   For every general model $\cu M$ of $T$, $\cu M_e=(\cu M,[\lim d_m]^{\cu M})$ is the a unique expansion of $\cu M$ to a general model of $T_e$.
\item[(ii)]  Every general model of $T_e$ is equal to $(\cu M,[\lim d_m]^{\cu M})$ where $\cu M$ is a general model of $T$.
\end{itemize}
\end{df}

We say that the sequence $\langle d \rangle$ \emph{approximates} $D$ in $T_e$.  Note that if $\langle d \rangle$ \emph{approximates} $D$ in $T_e$,
then so does every subsequence of $\langle d\rangle$.  Therefore, for every definitional expansion $T_e$ of $T$ over $V_D$,
there is an exponentially Cauchy sequence of formulas of $V$ that approximates $D$ in $T_e$.

\begin{lemma}  \label{l-Te-axioms}  (Axioms for $T_e$)
Suppose $T_e$ is a definitional expansion of $T$, and $\langle d_m(x,y)\rangle_{m\in\BN}$ is exponentially Cauchy in $T$ and approximates $D$ in $T_e$.
Let $T'_e$ be the union of $T$ and the set of sentences
\begin{equation} \label{e-Te}
\{\sup_x\sup_y (|d_m(x,y)-D(x,y)|\dotle 2^{-m})\colon m\in\BN\}.
\end{equation}
Then $T_e$ is equivalent to $T'_e$.
\end{lemma}

\begin{proof}  Let $\cu N$ be a general structure with vocabulary $V_D$, and let $\cu M$ be the $V$-part of $\cu N$.
Then $\cu N$ satisfies (\ref{e-Te}) if and only if $D^{\cu N}=[\lim d_m]^{\cu M}$.
By \ref{d-definitional-expansion} (i), every general model of $T'_e$ is a general model of $T_e$.
By \ref{d-definitional-expansion} (ii), every general model of $T_e$  is a general model of $T'_e$.
\end{proof}

\begin{rmk} \label{r-definitional-expansion} Suppose $T_e$ is a definitional expansion of $T$ with vocabulary $V_D$.
\begin{itemize}
\item[(i)] For each sequence $\langle d_m\rangle_{m\in\BN}$ of formulas of $V$ that approximates $D$ in $T_e$,
and each general model $\cu M$ of $T$,  $D^{\cu M_e}=[\lim d_m]^{\cu M}$ is defined by $\langle d\rangle$ in $\cu M$.
\item[(ii)] If $T_e,T_f$ are definitional expansions of $T$ with vocabulary $V_D$, and there is a sequence of formulas of $V$ that approximates $D$
in both $T_e$ and $T_f$, then $T_e$ and $T_f$ are equivalent.
\item[(iii)]  If $T\subseteq U$, then $U_e:=T_e\cup U$ is a definitional expansion of $U$, every sequence that approximates $D$ in $T_e$ approximates
$D$ in $U_e$, and for every general model $\cu M$ of $U$, the definitional expansion $\cu M_e$ of $\cu M$ is the same for $U_e$ as for $T_e$.
\item[(iv)]  For each general model $\cu M$ of $T$, $\Th(\cu M_e)$ is a definitional expansion of $\Th(\cu M)$.
\end{itemize}
\end{rmk}

The next remark says that definitional expansions are preserved under the addition of constant symbols.

\begin{rmk}  \label{r-definitional-expansion-parameters}
Suppose $V'=V\cup C$ where $C$ is a set of constant symbols.
If $T_e$ is a definitional expansion of the theory $T$ with vocabulary $V_D$, then
$T_e$ is still a definitional expansion of $T$ with vocabulary $V'_D$.
\end{rmk}

\begin{lemma}  \label{l-reduced-Me}
Suppose $T_e$ is a definitional expansion of a theory $T$, and $\cu M$ is a general model of $T$.
\begin{itemize}
\item[(i)]  $(\doteq^{\cu M}) = (\doteq^{\cu M_e})$.
\item[(ii)] $\cu{M}_e$ is reduced if and only if $\cu{M}$ is reduced.
\item[(iii)]  If  $\cu M'$ is the reduction of $\cu M$, then $\cu{M}'_e$ is the reduction of $\cu{M}_e$.
\item[(iv)]  If $\cu{M}'$ is a general model of $T$ and $h\colon\cu{M}\cong\cu{M}'$, then $h\colon\cu{M}_e\cong\cu{M}'_e$.
\end{itemize}
\end{lemma}

\begin{proof}  Let $\langle d_m(\vec x)\rangle_{m\in\BN}$ approximate $D$ in $T_d$.

(i): It is trivial that $x\doteq^{\cu{M}_e} y$ implies $x\doteq^{\cu{M}} y$.  Suppose $x\doteq^{\cu{M}} y$.
Consider an atomic formula $D(\vec{\tau}(x,\vec z))$, where $\vec \tau$ is a tuple of terms, and none of the variables $x, y,\vec{z}$ occur freely
in the any of the approximating formulas $d_m$.
Then for all $\vec z$ in $\cu M_e$ we have
$$ D(\vec{\tau}(x,\vec z))=\lim_{m\to\infty} d_m(\vec{\tau}(x,\vec z))=\lim_{m\to\infty} d_m(\vec{\tau}(y,\vec z))=D(\vec{\tau}(y,\vec z)),$$
so $x\doteq^{\cu{M}_e} y.$  Therefore  $(\doteq^{\cu{M}_e})=(\doteq^{\cu{M}})$.

(ii): Expansions of reduced structures are always reduced.  Suppose $\cu M_e$ is reduced. By the proof of (ii), if $x\doteq^{\cu M}y$ then $x\doteq^{\cu N} y$, and
since $\cu{M}_e$ is reduced, $x=y$.  Hence $\cu{M}$ is reduced.

(iii)
For each $x\in M$ let $x'\in M'$ be the equivalence class of $x$ under $\doteq^{\cu M}$.  Then
$$D^{\cu{M}_e}(\vec x)=\lim_{m\to\infty} d_m^{\cu{M}}(\vec x)=\lim_{m\to\infty} d_m^{\cu{M}'}(\vec x')=D^{\cu{M}'_e}(\vec x'),$$
so $\cu{M}'_e$ is the reduction of $\cu{M}_e$.

(iv):  By (ii), we may assume that $\cu M, \cu M'$ are reduced.  Then any isomorphism $h\colon\cu{M}\cong\cu{M}'$ sends $[\lim d_m]^{\cu M}$ to $[\lim d_m]^{\cu M'}$.
\end{proof}

\begin{prop}  \label{p-expansion-ultraproduct}  (Definitional expansions commute with ultraproducts.)
Suppose $T_e$ is a definitional expansion of $T$, $\cu D$ is an ultrafilter over a set $I$, and $\cu M_i\models T$ for each $i\in I$.
Then $(\prod^{\cu D}\cu M_i)_e=\prod^{\cu D}((\cu M_i)_e)$, and $(\prod_{\cu D}\cu M_i)_e=\prod_{\cu D}((\cu M_i)_e)$.
\end{prop}

\begin{proof}  By Fact \ref{f-Los}, $\prod^{\cu D}\cu M_i$, $\prod_{\cu D}\cu M_i$ are models of $T$, and
$\prod^{\cu D}((\cu M_i)_e)$, $\prod_{\cu D}((\cu M_i)_e)$ are models of $T_e$.
$(\prod^{\cu D}\cu M_i)_e=\prod^{\cu D}((\cu M_i)_e)$ because $\prod^{\cu D}\cu M_i$ is the $V$-part of $\prod^{\cu D}((\cu M_i)_e)$.
Therefore, by  Lemma \ref{l-reduced-Me} (ii), $(\prod_{\cu D}\cu M_i)_e=\prod_{\cu D}((\cu M_i)_e)$.
\end{proof}

\subsection{Pre-metric Expansions}

We assume hereafter that $D$ is a binary  predicate symbol.

\begin{df} \label{d-metric-expansion}
Let $T$ be a theory in a vocabulary $V$.
We say that $T_e$ is a pre-metric expansion of $T$ (with signature $L_e$) if:
\begin{itemize}
\item[(i)]  $(T_e,L_e)$ is a metric theory whose signature $L_e$ has vocabulary $V_D$ and distance predicate $D$.
\item[(ii)] There is a Cauchy sequence $\langle d\rangle$ of formulas in $T$ such that
the general models of $T_e$ are exactly the structures
of the form $\cu M_e=(\cu M,[\lim d_m]^{\cu M})$, where $\cu M$ is a general model of $T$.
\end{itemize}
\end{df}

Since $(T_e,L_e)$ is a metric theory, we have $T_e\models \met(L_e)$, and every general model of $T_e$ is the downgrade of a pre-metric structure with signature $L_e$.
Condition (ii) in the above definition just says that $T_e$ is a definitional expansion of $T$.
So each of the results \ref{l-Te-axioms} -- \ref{p-expansion-ultraproduct} hold for pre-metric expansions because they hold for all definitional expansions.

We call $(\cu M_e,L_e)$ the \emph{pre-metric expansion of} $\cu M$ for $T_e$, and call $\cu M$ the \emph{non-metric part} of $\cu M_e$.
We abuse notation by using $\cu M_e$ to denote both the pre-metric expansion $(\cu M_e,L_e)$ of $\cu M$, and its downgrade  $\cu M_e$.
We call a sequence $\langle d\rangle$ that approximates $D$ in $T_e$ an \emph{approximate distance for} $T_e$, and also for $\cu M_e$.
Every pre-metric expansion of $T$ has approximate distances, but they are not unique.

Here are three rather trivial examples of pre-metric expansions.

\begin{ex} \label{e-no-predicates}
Suppose $V$ has no predicate symbols.  Then up to equivalence, the unique pre-metric expansion of $T$ is the theory $T_e=T\cup\{\sup_x\sup_y D(x,y)\dotle 0\}$
in which the distance between any two elements is $0$.
Every reduced model of $T$ or of $T_e$ is a one-element structure.
\end{ex}

\begin{ex}  \label{e-metric-expansion-trivial}    Suppose the predicate symbol $D$ already belongs to $V$,
$L$ is a metric signature over $V$ with distance predicate $D$, and $T\models\met(L)$.  Let $T_e$ be the metric theory with signature $L$
and the same set of sentences as $T$.  Then $T_e$
is a pre-metric expansion of $T$,  and for each general model $\cu M$ of $T$, the upgrade of $\cu M$ to $L$ is the pre-metric expansion of $\cu M$ for $T_e$.
\end{ex}

\begin{ex} \label{e-metric-expsansion-discrete}
If there is a formula $D(x,y)$ in the vocabulary $V$ such that $D^{\cu M}$ is $\doteq^{\cu M}$ in every general model $\cu M$ of $T$, then $T$ has a pre-metric expansion
with the distinguished distance $D$ and the trivial moduli of uniform continuity.  We call such a pre-metric expansion \emph{discrete}.
A theory $T$ has a discrete pre-metric expansion if and only if
 $T$ has a Cauchy sequence of formulas $\langle d_m(x,y)\rangle_{m\in\BN}$ such that
in every reduced model $\cu N$ of $T$, $[\lim d_m]^{\cu{N}}$ is the discrete metric $=^{\cu N}$.
(Hint: Take  $d_m(x,y)$ to be exponentially Cauchy, so that $d_2$ is always within $1/4$ of the limit, and take $D(x,y)=C(d_2(x,y))$ for an appropriate connective $C$.)
\end{ex}

Note that a classical structure without equality (where every atomic formula has truth value $0$ or $1$) will not necessarily have a discrete pre-metric expansion.
But by the Expansion Theorem below, it will have a pre-metric expansion if $V$ is countable.

\medskip

\emph{From this point on, except when we specify otherwise,
every general structure we consider will be understood to be a general structure whose vocabulary has at most countably many predicate and function symbols.
We let $V$ be a vocabulary that contains at most countably many predicate and function symbols (we make no restriction on the number of constant symbols).
We let $D$ be a binary predicate symbol and let $V_D=V\cup\{D\}$.}
\medskip

Formulas  in the vocabulary $V$ will be called \emph{$V$-formulas},
and general structures for $V$ will be called $V$-structures.  Similarly for terms and sentences.
Unless we say otherwise, $T$ will be a $V$-theory, that is, a set of $V$-sentences.

\begin{df}  \label{d-exact} Let $T$ be a complete $V$-theory.
We call $\langle d\rangle=\langle d_m(x,y)\rangle_{m\in\BN}$ an \emph{exact distance} in $T$ if
\begin{itemize}
\item [(i)] $\langle d\rangle$ is Cauchy in $T$.
\item[(ii)] $[\lim d_m]^{\cu M}$ is a pseudo-metric in every general model of $T$.
\item[(iii)] For each general model $\cu M\models T$ and $V$-formula $\varphi(\vec x)$,
the mapping $\varphi^{\cu M}\colon M^{|\vec x|}\to[0,1]$ is uniformly continuous in the pseudo-metric space $(M,[\lim d_m]^{\cu M})$.
\end{itemize}
\end{df}

\begin{prop}  \label{p-expansion-d-metric}
Suppose $T_e$ is a pre-metric expansion of a complete $V$-theory $T$ with approximate distance $\langle d\rangle=\langle d_m(x,y)\rangle_{m\in\BN}$.  Then
$\langle d\rangle$ is an exact distance in $T$.
\end{prop}

\begin{proof}  By the definition of a pre-metric expansion, Definition \ref{d-exact} (i) and (ii) hold, and
for every $\cu M\models T$, $\cu M_e=(\cu M,[\lim d_m]^{\cu M})$ is a pre-metric structure with signature $L_e$.  Then by Fact \ref{f-t3.5}, (iii) holds.
\end{proof}

We now state two results, Theorems \ref{t-Be05} and \ref{f-BU-Th4.23}, that follow from the proofs of results from [I94] in the context of Henson's Banach space
model theory [He], and from [BU] in the context of
open Hausdorff cats.   Those results together  imply that every complete theory has a pre-metric expansion.
In order to convert the proofs from [I94] and [BU] to proofs of Theorems \ref{t-Be05} and \ref{f-BU-Th4.23} in our present setting,
we would need a long detour through positive bounded formulas and open Hausdorff cats.  Instead,
 in the next subsection we will give a different and self-contained proof of a stronger result,
that every (not necessarily complete) theory has a pre-metric expansion
with an approximate distance $\langle d_m\rangle_{m\in\BN}$ such that each $d_m$ defines a pseudo-metric in every general model of $T$.

\begin{thm} \label{t-Be05}  (By the proof of Proposition 53 of [I94], Theorem 5.1 of [I99], or Theorem 2.20 of [BY05].)
Every complete $V$-theory $T$ has an exact distance.
\end{thm}

\begin{thm}  \label{f-BU-Th4.23}  (By Theorem 4.23 of [BU].)
Let $T$ be a complete $V$-theory.  If $T$ has an exact distance,
then $T$ has a pre-metric expansion.
\end{thm}

Note that by Proposition \ref{p-expansion-d-metric}, every complete $V$-theory has a pre-metric expansion if and only if
both Theorem \ref{t-Be05} and Theorem \ref{f-BU-Th4.23} hold.

\subsection{The Expansion Theorem}

We will show that every $V$-theory $T$ has a pre-metric expansion with an approximate distance $\langle d_m(x,y)\rangle_{m\in\BN}$
such that each $d_m$ defines a pseudo-metric in every general model of $T$.
In Lemma \ref{l-metric-expansion-relational} below  we  prove this in  the special case that $V$ has no function symbols.
We will use that to prove the general result in Theorem \ref{t-metric-expansion-exists}.
By Remark \ref{r-definitional-expansion} (iii), it is enough to show that the empty set of sentences with vocabulary $V$ has such a pre-metric expansion.

\begin{df}  A formula $\varphi(\vec x,\vec y)$ is \emph{pseudo-metric in} $T$ if $|\vec x|=|\vec y|$
and $\varphi^{\cu M}$ is a pseudo-metric on $M^{|\vec x|}$ for every general model $\cu M\models T$.
By a \emph{pseudo-metric approximate distance for} a pre-metric expansion $T_e$ of $T$, we mean an approximate distance $\langle d_m(x,y)\rangle_{m\in\BN}$
for $T_e$ such that each $d_m$ is a pseudo-metric formula in $T$.
\end{df}

\begin{rmk} \label{r-expansion-subtheory}
(i) If $\langle\varphi_m(\vec x,\vec y)\rangle_{m\in\BN}$ is Cauchy in $T$ and each formula $\varphi_m$ is pseudo-metric in $T$,
then $[\lim \varphi_m]^{\cu M}$ is a pseudo-metric on $M^{|\vec x|}$ for every $\cu M\models T$.

(ii)  If $T\subseteq U$ and $T_e$ is a pre-metric expansion of $T$, then $U_e=T_e\cup U$ is a pre-metric expansion of $U$ with the same
metric signature and approximate distance as $T_e$.
\end{rmk}

\begin{proof}  (i) is clear.  (ii) follows from Remark \ref{r-definitional-expansion} (iii) and part (i) above.
\end{proof}

\begin{lemma} \label{l-metric-expansion-relational}   If the vocabulary $V$ has countably many predicate symbols and no function symbols, then
every $V$-theory $T$ has a pre-metric expansion with a pseudo-metric approximate distance.
\end{lemma}

\begin{proof}   Let $x, y, u,z_1,z_2,\ldots$ be distinct variables.  Arrange all atomic formulas with no constant
symbols whose variables are among $u,z_1,z_2,\ldots$ in a countable list
$\alpha_0,\alpha_1,\ldots$.  For each $m$, let $\alpha_m(x,\vec z)$ and $\alpha_m(y,\vec z)$ be the formulas formed from $\alpha_m(u,\vec z)$
by replacing $u$ by $x$ and $y$ respectively.
 For each $m$, let $\beta_m(x,y)$ be the $V$-formula
$$ \beta_m(x,y) = \sup_{\vec {z}}|\alpha_m(x,\vec{z})-\alpha_m(y,\vec{z})|.$$
Let $d_0(x,y)=\beta_0(x,y)$, and for each $m>0$ let $d_m(x,y)$ be the $V$-formula
$$d_m(x,y)=\max(d_{m-1}(x,y),2^{-m}\beta_m(x,y)).$$
Then for every $V$-structure $\cu{M}$ and every $m$, $\beta_m^{\cu{M}}$ and $d_m^{\cu{M}}$ are pseudo-metrics on $M$.  Moreover,
$$\models d_m(x,y)\dotle d_{m+1}(x,y)\dotle d_m(x,y)\dotplus 2^{-(m+1)},$$
so $\langle d\rangle:=\langle d_m(x,y)\rangle_{m\in\BN}$ is exponentially Cauchy in $T$.
For every general model $\cu M$ of $T$, let $\cu M_e=(\cu M,D^{\cu M_e})$ where $D^{\cu M_e}=[\lim d_m]^{\cu M}$.
Since each $d_m^{\cu M}$ is a  pseudo-metric on $M$, the limit $D^{\cu M_e}$ is a pseudo-metric on $M$.

For every $\cu M\models T$, $\cu M_e$ is a general model of the $V_D$-theory
$$T_{\langle d\rangle}=\{\sup_x\sup_y[d_m(x,y)\dotle D(x,y)\dotle d_m(x,y)\dotplus 2^{-m}]\colon m\in\BN\}.$$
We will find a metric signature $L_e$ such that $T_e=T\cup T_{\langle d\rangle}$ is a pre-metric expansion of $T$ with approximate distance ${\langle d\rangle}$.
It suffices to specify a modulus of uniform continuity $\triangle_P$ for each $k$-ary predicate symbol $P\in V$ that is satisfied with
respect to $D$ in every general model of $T_{\langle d\rangle}$.
 For each $1\le i\le k$, let $P(u,\vec z)_i$ be the atomic formula obtained from $P(z_1,\ldots,z_k)$
by replacing $z_i$ by $u$.  Then  $P(u,\vec z)_i$ is $\alpha_{m_i}$ for some $m_i\in\BN$.
Let $m=\max(m_1,\ldots,m_k)$.  For each $1\le i\le k$,  we have
$$ T_{\langle d \rangle}\models\sup_{\vec z}|P(x,\vec z)_i-P(y,\vec z)_i|=\beta_{m_i}(x,y)\dotle 2^m d_m(x,y)\dotle 2^m D(x,y).$$
Therefore whenever $\vec x, \vec y$ differ only in the $i$-th argument we have
$$T_{\langle d \rangle}\models|P(\vec x)-P(\vec y)|\dotle 2^m D(x_i,y_i).$$
Since one can change any $k$-tuple $\vec x$ to $\vec y$ in $k$ steps by changing one variable at a time,
for every pair of $k$-tuples $\vec x,\vec y$ we have
$$T_{\langle d \rangle}\models|P(\vec x)-P(\vec y)|\dotle 2^m k \max( D(x_1,y_1),\ldots,D(x_k,y_k)).$$
It follows that $P$ has modulus of uniform continuity $\triangle_P(\varepsilon)=2^{-m}k^{-1}\varepsilon$ with respect to $D$ in each general model of $T_{\langle d \rangle}$.
Therefore $T_e=T\cup T_{\langle d\rangle}$ is a pre-metric expansion of $T$.
\end{proof}

In the general case that $V$ contains function symbols, we must also specify a modulus of uniform continuity for each function symbol with respect to $D$.
In a pre-metric structure with distance $d$, one can eliminate function symbols by using the formula $d(F(\vec x),y)$
for the graph of $F(\vec x)$.  This will not work in a general structure, because
the formula $D(F(\vec x),y)$ for the graph of  $F(\vec x)$ is not a $V$-formula.
We will circumvent that difficulty by an ``atomic Morleyization'', adding a new predicate symbol for each atomic $V$-formula.

\begin{thm}  \label{t-metric-expansion-exists} (Expansion Theorem) For every vocabulary $V$ with countably many predicate symbols and countably many
function symbols, every $V$-theory $T$ has a pre-metric expansion with a pseudo-metric  approximate distance.
\end{thm}

\begin{proof}  By Remark \ref{r-expansion-subtheory} (ii), we may assume that $T$ is the empty set of sentences.   Let $V'$ be the union of $V$ and a
$k$-ary  predicate symbol $P_\alpha$
for each atomic $V$-formula $\alpha(\vec x)$ with $k$ variables and no constant symbols.  Let $T'$ be the set of $V'$-sentences
$\sup_{\vec x} |P_{\alpha}(\vec x)-\alpha(\vec x)|$
for each $\alpha(\vec x)$.  Then $V'\setminus V$ is a countable set of new predicate symbols, and
every $V$-structure $\cu M$ has a unique expansion to a  general model $\cu{M}'$ of $T'$.
It can be shown by induction on complexity that every $V'$-formula $\varphi(\vec x)$ is $T'$-equivalent to a $V$-formula $\varphi^0(\vec x)$, and is also
$T'$-equivalent to a $V'$-formula $\varphi''(\vec x)$ that has no function symbols.

Let $V''$ be the set of all predicate and constant symbols  in $V'$, and for each  $V$-structure $\cu{M}$ let $\cu{M}''$ be the $V''$-part of $\cu{M}'$.
Let $T''$ be the set of all $V''$-sentences that hold in all general models of $T'$, which is the same as the set of all $V''$-sentences that hold in
$\cu{M}''$ for every  $V$-structure $\cu{M}$.  Note that $V''\subseteq V'$, $T''\subseteq T'$, and for each $V'$-formula $\varphi(\vec x)$,
the formula $\varphi''(\vec x)$ defined in the previous paragraph is a $V''$-formula.
By Lemma \ref{l-metric-expansion-relational}, $T''$ has a pre-metric expansion $T''_e$ with a pseudo-metric approximate distance $\langle d\rangle=\langle d_m(x,y)\rangle_{m\in\BN}$.
 Then for each $V$-structure $\cu{M}$,  $\cu{M}''_{e}=(\cu M'',[\lim d_m]^{\cu M''})$.

We next find a metric signature $L'_e$ over $V'_D$ such that $T'_e=T''_e\cup T'$ with signature $L'_e$
is a pre-metric expansion of $T'$ with the same approximate distance.  Since each $\cu M'$ is an expansion of $\cu M''$,
$d_m^{\cu M'}=d_m^{\cu M''}$, so $[\lim d_m]^{\cu M'}=[\lim d_m]^{\cu M''}$.
We must find a modulus of uniform continuity for each $k$-ary function symbol $F$ in $V$.
For each $m$, let $\theta_m(\vec x,y)$ be a $V''$-formula that is $T'$-equivalent to $d_m(F(\vec x),y)$.  Then
 $\langle \theta_m(\vec x,y)\rangle_{m\in\BN}$ is  Cauchy in $T''$,
and for each $\cu M\models T$ and all $(\vec b,c)\in M^{k+1}$,
$$D^{\cu M''_e}(F^{\cu M}(\vec b),c)=[\lim d_m]^{\cu M''}(F^{\cu M}(\vec b),c)=[\lim \theta_m]^{\cu M''}(\vec b,c)=[\lim \theta_m]^{\cu M''_e}(\vec b,c).$$
 By Lemma \ref{l-def-predicate-uniform},
there is a function $\triangle_F$ that is a modulus of uniform continuity for $[\lim \theta_m]^{\cu M''_e}$ in the pre-metric structure $\cu M''_e$
for every general model $\cu M$ of $T$.
We note that Proposition 9.23 of [BBHU] on definable functions holds for pre-metric structures as well as metric structures.
Therefore in $\cu M''_e$, the function $F^{\cu M}$ is definable and has the same modulus of uniform continuity $\triangle_F$.
 This shows that $\cu M'_e=(\cu M',[\lim d_m]^{\cu M'})$ is
a pre-metric structure with the metric signature $L'_e$ that agrees with $L''_e$ on $V''$ and gives  $F$ the modulus of uniform continuity $\triangle_F$.

Finally, we show that there is a pre-metric expansion $T_e$  of the empty theory $T$ with a pseudo-metric approximate distance.
We take $L_e$ to be the restriction of $L'_e$ to $V_D$ (so the symbols of $V_D$ have the same moduli of uniform continuity in $L_e$ as in $L'_e$).
Then $L_e$ is a metric signature with distance predicate $D$ over $V_D$.
For each $m$ let $\hat{d}_m(x,y)$ be a $V$-formula that is $T'$-equivalent to $d_m(x,y)$, and let $\langle \hat{d}\rangle=\langle \hat{d}_m(x,y)\rangle_{m\in\BN}.$
Since $\langle d\rangle$ is  Cauchy in $T'$ and each $d_m$ is pseudo-metric in every general model of $T'$,
and every $\cu M\models T$ has a unique expansion to a general model $\cu M'\models T'$,
$\langle \hat{d}\rangle$ is  Cauchy in $T'$ and each $\hat{d}_m$ is pseudo-metric in each general model of $T$.
We let $T_e$ be the union of $T$ and a set of sentences saying that $D(x,y)=\lim_{m\to\infty} \hat {d}_m(x,y)$ with signature $L_e$.
Then $T_e$ is a pre-metric expansion of $T$ with  pseudo-metric approximate distance $\langle \hat{d}\rangle$.
\end{proof}

Theorem \ref{t-metric-expansion-exists} shows that a pre-metric expansion with a pseudo-metric approximate distance always exists, but the distance predicate built in the proof
depends on an arbitrary enumeration of the atomic formulas, and  may not be natural.

\begin{cor}
Every complete $V$-theory $T$ has an exact distance $\langle d_m\rangle_{m\in\BN}$ in which each $d_m^{\cu M}$ is a pseudo-metric for each $\cu M\models T$.
\end{cor}

\begin{proof}
By Theorem \ref{t-metric-expansion-exists} and Proposition \ref{p-expansion-d-metric}.
\end{proof}


\begin{question}  \label{q-pseudo}
Does every pre-metric expansion of $T$ have a pseudo-metric approximate distance?
\end{question}

Added in October, 2020: With the permission of James E. Hanson, we report that he has shown that the answer to Question \ref{q-pseudo} is ``Yes''.
If $\langle d_k\rangle$ is an approximate distance for $T_e$, then $\langle e_k\rangle$ is a pseudo-metric approximate distance for $T_e$,
where $e_n(x,y)=\sup_x|d_k(x,z)-d_k(y,z)|$.

We now show that the pre-metric expansion of a $V$-structure $\cu{M}$ is unique up to a uniformly continuous homeomorphism.

\begin{prop}  \label{p-homeomorphic}  Let $T_e, T_f$ be pre-metric expansions of $T$.  There is a function $\triangle(\cdot)\colon (0,1]\to(0,1]$ such that for
every general model $\cu{M}$ of $T$, the identity function is a uniformly continuous homeomorphism from $(M,D^{\cu{M}_e})$ onto $(M,D^{\cu{M}_f})$
with modulus $\triangle(\cdot)$.
\end{prop}

\begin{proof}  We must find a function $\triangle(\cdot)$ that is a modulus of uniform continuity for $D^{\cu M_f}$
in the pre-metric structure $\cu M_e$ for each $\cu M\models T$.
Let $ \langle d_m\rangle_{m\in\BN}$ be an approximate distance for $f$.
For each $\cu M\models T$, $\cu M_e$ is a pre-metric structure with signature $L_e$.
$ \langle d_m\rangle_{m\in\BN}$ is Cauchy in the $V$-theory $T$.  $T$ is also a $V_D$-theory, and $D^{\cu M_f}=[\lim d_m]^{\cu M}=[\lim d_m]^{\cu M_e}$.
 By Lemma \ref{l-def-predicate-uniform} (ii),
there is a function $\triangle(\cdot)$ with the required property.
\end{proof}

Let us look at the pre-metric expansion process in the reverse direction.  The general structure $\cu M$ can by obtained from the pre-metric expansion
$\cu M_e$ by simplifying in two steps: first downgrade  $\cu M_e$ to a general structure $\cu N$ by forgetting the metric signature, and then take the $V$-part
of $\cu N$ by forgetting the distance to get the original general structure $\cu M$.   $\cu M$ is the non-metric part of $\cu M_e$.
We thus have two equivalence relations on the class of pre-metric structures with vocabulary $V_D$, the fine relation of having the same downgrade,
and the coarse relation of having the same non-metric part.

Formally, the class of general structures is disjoint from the class  of pre-metric structures because only the latter includes a metric signature.
But intuitively, we regard $\cu M_e$ and the downgrade of $\cu M_e$ as the same structure presented in two different ways.

\subsection{Absoluteness}

In this rather philosophical subsection we develop a framework that allows one to apply known model-theoretic results about metric structures to general structures.

\emph{For the remainder of this paper, we let $T$ be a  $V$-theory, $T_e$ be an arbitrary pre-metric expansion of $T$ with signature $L_e$,
$\cu M$ be an arbitrary general model of $T$, and $\cu M_e$ be the pre-metric expansion of $\cu M$ for $T_e$.}

By a \emph{general property}  we mean a class of general  structures that is preserved under automorphism (i.e., bijective embeddings), and by a
\emph{pre-metric property} we mean a class of pre-metric structures, again preserved under automorphisms.
When $\cu P$ is a  property, we interchangeably use the phrases ``$\cu N$ belongs to $\cu P$'', ``$\cu N$ has property $\cu P$'',  ``$\cu N$ satisfies $\cu P$'',
and ``$\cu P$ holds in $\cu N$''.

\begin{df}  \label{d-absolute-version}
An \emph{absolute version} of a pre-metric property $\cu{Q}$ is a general property $\cu P$ such that whenever $\cu M_e$
is a pre-metric expansion of  $\cu M$, $\cu M$ has property $\cu P$ if and only if $\cu M_e$ has property $\cu Q$.
\end{df}

The following corollary was stated in the Introduction.

\begin{cor}  \label{c-absolute-unique}
Every pre-metric property  $\cu{Q}$  has at most one absolute version.
\end{cor}

\begin{proof}  Suppose $\cu P$ and $\cu{P}'$ are absolute versions of $\cu Q$, and consider a general structure $\cu M$.
By Theorem \ref{t-metric-expansion-exists}, $\cu{M}$ has a pre-metric expansion $\cu{M}_e$.  Then the following are equivalent: $\cu M$ has property $\cu{P}$,
$\cu M_e$ has property $\cu Q$, $\cu M$ has property $\cu P'$.
\end{proof}

\begin{df} \label{d-absolute} We say that a general property  $\cu P$  is \emph{absolute} if whenever $\cu M_e$ is a pre-metric expansion of  $\cu M$,
$\cu{M}$ has property $\cu P$ if and only if the downgrade of $\cu{M}_e$ has property $\cu P$.
\end{df}

\begin{cor}  \label{c-absolute-version}
Suppose $\cu P$ is an absolute version of $\cu Q$.  Then $\cu P$ is absolute.
\end{cor}

\begin{proof}
By Example \ref{e-metric-expansion-trivial}, each pre-metric structure $\cu N$ is a pre-metric expansion of the downgrade of $\cu N$.
Therefore for each $\cu M$ and  $\cu M_e$, the following are equivalent:
$\cu M$ has property $\cu P$, $\cu M_e$ has property $\cu Q$, the downgrade of $\cu M_e$ has property $\cu P$.
\end{proof}

By Lemma \ref{l-reduced-Me} (i) we have:

\begin{prop}  \label{p-reduced-absolute}
The  property  \emph{[$\cu M$ is reduced]}  in Definition \ref{d-reduced} is absolute.
\end{prop}

In view of Remark \ref{r-definitional-expansion-parameters}, we often consider properties of a  general structure $\cu M$ with additional parameters from $M$.

\begin{prop}  \label{p-realize-absolute}
Let $p(\vec x,A)$ be a set of $V$-formulas with parameters from $M$.
  The property  \emph{[$p(\vec x,A)$ is realized in $\cu M_A$]} is absolute.
\end{prop}

\begin{proof}  Every pre-metric expansion of $\cu M$ is an expansion of $\cu M$.   It follows from Remark \ref{r-formula-expansion},
that if  $\cu M'$ is an expansion of $\cu M$, $A\subseteq M$, and $\vec b\subseteq M$, then $\vec b$ realizes $p(\vec x,A)$
in $\cu M'_A$ if and only if $\vec b$ realizes $p(\vec x,A)$ in $\cu M_A$.
\end{proof}

For the same reason, Proposition \ref{p-realize-absolute} also holds when $p(\vec x,A)$ is an infinitary $L_{\omega_1\omega}$-formula in the sense of [Ea15]

In view of Corollaries \ref{c-absolute-unique} and \ref{c-absolute-version}, we consider the absolute version of a pre-metric property $\cu Q$, if there is one, to be the
``right'' way to extend $\cu Q$ to all general structures.  When we have a name for structures with a pre-metric property $\cu Q$,
it will be convenient to use the same name for general structures that satisfy an absolute version of $\cu Q$.  So we adopt the following convention.

\begin{df} \label{d-general-Q}
A pre-metric property $\cu Q$ is said to be \emph{absolute} if $\cu Q$ has an absolute version.
If $\cu Q$ is an absolute pre-metric property, a general structure that satisfies the absolute version of $\cu Q$ will be called
a \emph{general structure with property} $\cu Q$.
\end{df}

\begin{prop}  \label{p-robust}
For any pre-metric property $\cu Q$, the following are equivalent:
\begin{itemize}
\item[(i)]  $\cu Q$ is absolute.
\item[(ii)]  The class of general structures with property $\cu Q$ is absolute.
\item[(iii)] For every general structure $\cu{M}$ and all pre-metric expansions $\cu{M}_e, \cu{M}_f$ of $\cu{M}$, $\cu{M}_e$ has property $\cu Q$
if and only if $\cu{M}_f$ has property $\cu Q$.
\end{itemize}
\end{prop}

\begin{proof}  (i) and (ii) are equivalent by Corollary \ref{c-absolute-version}.  Assume (i), so $\cu Q$ has an absolute version
$\cu{P}$.   Let $\cu{M}$ be a  general structure and let $\cu{M}_e$ and $\cu{M}_f$ be two pre-metric expansions of
$\cu{M}$.  Then the following are equivalent: $\cu M_e$ has $\cu{Q}$, the downgrade of $\cu{M}_e$ has $\cu P$, $\cu M$ has $\cu P$, the downgrade of $\cu M_f$ has $\cu P$,
$\cu M_f$ has $\cu Q$.  Therefore (iii) holds.

Assume (iii). Let $\cu{P}$ be the property such that for each general structure $\cu M$, $\cu P$ holds for  $\cu{M}$ if and only if $\cu{Q}$
holds for some pre-metric expansion of $\cu{M}$. By (iii), for each $\cu M$ and  $\cu M_e$,
$\cu P$ holds for $\cu M$ if and only if $\cu Q$ holds for $\cu M_e$.
By Example \ref{e-metric-expansion-trivial},  each pre-metric structure $\cu N$ is a pre-metric expansion of the downgrade of $\cu N$,
so $\cu Q$ holds for $\cu N$ if and only if $\cu P$ holds for the downgrade of $\cu N$.
In particular, for each $\cu M$ and $\cu M_e$,  $\cu Q$ holds for $\cu M_e$ if and only if $\cu P$ holds for the downgrade of $\cu M_e$.
 It follows that $\cu{P}$ is an absolute version of $\cu{Q}$, so (i) holds.
\end{proof}

Here are some examples of  pre-metric properties that are \emph{not} absolute.

\begin{ex}  \label{e-Lipschitz}
\noindent\begin{itemize}
\item  Let $0<r\le 1$.  The property of being a pre-metric structure whose distinguished distance has diameter $r$ is not absolute.
Similarly for diameter $\le r$, and for diameter $\ge r$.
\item  Let $d(x,y)$ be a $V$-formula that defines a pseudo-metric in every $V$-structure.
The property of being a pre-metric structure $\cu M_e$ with distinguished distance $D$
such that  $\cu M_e$ satisfies
$\sup_x\sup_y (d(x,y)\dotle D(x,y))$, is not absolute.  (Hint:
If $\cu M_e$ has an approximate distance $\langle d_m\rangle_{m\in\BN}$, then $(\cu M,\max(\sa D,d^{\cu M}))$ is
a pre-metric expansion of $\cu M$ with approximate distance $\langle \max(d_m,d)\rangle_{m\in\BN}$ and the same signature $L_e$.)
\item
Say that a pre-metric structure  is \emph{Lipschitz} if every predicate and function symbol $S$
has a modulus of uniform continuity $\triangle_S$ that is linear, that is, for each $S$ there is a  positive real $\ell$
such that  $\triangle_S(x)=\ell x$ for all $x\in(0,1]$.
The property of being a Lipschitz pre-metric structure is not absolute.  (Hint: Consider pre-metric expansions of
the unit ball in an $n$-dimensional vector space over the reals.)
\end{itemize}
\end{ex}

In the next section, we will show that many pre-metric properties $\cu Q$ in the literature have absolute versions, and also give
necessary and sufficient conditions for $\cu M$ to be a general structure with property $\cu Q$.
Note that the characterization of the absolute version of $\cu Q$ given by the proof of Proposition \ref{p-robust}
mentions pre-metric expansions.  In the results that follow, we will always give necessary and sufficient conditions
for a general structure $\cu M$ to have property $\cu Q$ that are about $\cu M$ itself, rather than
conditions that mention pre-metric expansions of $\cu M$.
This is desirable for potential applications, because a general structure $\cu M$
may have a very simple description even though all its pre-metric expansions are complicated.

\section{Properties of General Structures} \label{s-robust}

In this section we use absoluteness to extend known results about metric structures to general structures.

\subsection{Types in Pre-metric Expansions}  \label{s-types}

\begin{prop}  \label{p-expansion-equiv}
If $\cu M\models T$ and $\cu N\equiv \cu M$, then $\cu N_e\equiv\cu M_e$.
\end{prop}

\begin{proof}
Let $\kappa$ be a special cardinal such that $|V|+\aleph_0<\kappa$.
By  the Existence Theorem for Special Models,
there are $\kappa$-special $V_D$-structures $\cu{M}'\equiv \cu{M}_e$ and $\cu{N}'\equiv\cu{N}_e$.  Let $\cu{M}'', \cu{N}''$ be the $V$-parts of $\cu M',\cu N'$ respectively.
Then $\cu{M}'', \cu{N}''\models T$,
$\cu{M}'=\cu{M}''_e$ and $\cu{N}'=\cu{N}''_e$.
By Remark \ref{r-special-expansion}, the reductions of $\cu{M}''$ and $\cu{N}''$ are $\kappa$-special.
 By  the Uniqueness Theorem for Special Models, we have $\cu{M}''\cong\cu{N}''$.
By Lemma \ref{l-reduced-Me} (iii) we have $\cu{M}''_e\cong\cu{N}''_e$, or in other words, $\cu{M}'\cong\cu{N}'$.
Therefore $\cu{N}_e\equiv\cu{N}'\equiv\cu{M}'\equiv\cu{M}_e$.
\end{proof}

\begin{cor}  \label{c-expansion-over-A}  Suppose $A\subseteq M$ and $h\colon A\to N$.
Then $(\cu{M},a)_{a\in A}\equiv(\cu{N},ha)_{a\in A}$ if and only if
$(\cu{M}_e,a)_{a\in A}\equiv(\cu{N}_e,ha)_{a\in A}$
\end{cor}

\begin{proof} By  Remark \ref{r-definitional-expansion-parameters} and Proposition \ref{p-expansion-equiv}.
\end{proof}

\begin{cor}  \label{c-expansion-elem-ext}
$h\colon\cu{M}\prec\cu{N}$ if and only if $h\colon\cu{M}_e\prec\cu{N}_e$.
\end{cor}

\begin{proof}  By Corollary \ref{c-expansion-over-A} with $A=M$.
\end{proof}

\begin{cor}  \label{c-type-over-A}  The property of two tuples realizing the same type over $A$ is
absolute---$\tp_{\cu{M}}(\vec b/A)=\tp_{\cu{M}}(\vec c/A)$
if and only if $\tp_{\cu{M}_e}(\vec b/A)=\tp_{\cu{M}_e}(\vec c/A).$
Also, indiscernibility over $A$ is absolute.
\end{cor}

\begin{proof}  Apply Corollary \ref{c-expansion-over-A} with $\cu{N}=\cu{M}$ and where $h$  is the identity on $A$ and maps $\vec b$ to $\vec c$.
\end{proof}

The next corollary shows that for each complete theory $T$, $T_e$ has essentially the same complete types as $T$.

\begin{cor}  \label{c-homeo-types}
If $T=\Th(\cu M)$ and $T_e=\Th(\cu M_e)$, there is a  homeomorphism $h$ from $S_n(T)$ onto $S_n(T_e)$
such that for each  $\vec c\in M^n$, $h(\tp_{\cu M}(\vec c))=\tp_{\cu M_e}(\vec c)$.
\end{cor}

\begin{proof}  This follows from Corollary \ref{c-type-over-A}.
\end{proof}

\begin{prop}  \label{p-saturated-metric-expansion}  Let $\kappa$ be an infinite cardinal. The $\kappa$-saturation property is absolute---for
each general model $\cu{M}$ of $T$, $\cu{M}$ is $\kappa$-saturated if
and only if $\cu{M}_e$ is  $\kappa$-saturated.
\end{prop}

\begin{proof}  We prove the non-trivial direction.
Suppose $\cu{M}$ is $\kappa$-saturated.
By Remark \ref{r-reduction-sat} and Lemma \ref{l-reduced-Me} (i), we may assume without loss of generality that $\cu{M}$ is reduced.  Then $\cu{M}_e$ is reduced.
Let $A\subseteq M$ with $|A|<\kappa$, and let $\Gamma(x)$ be a set of $V_D$-formulas that is finitely satisfiable in $(\cu{M}_e)_A$.   By
the Existence Theorem for Special Models,
there is a reduced $\kappa$-saturated elementary extension
$\cu{N}'\succ\cu{M}_e$. Then $\cu{N}'$
is equal to $\cu{N}_e$ where $\cu{N}$ is the $V$-part of $\cu N'$, and $\cu N\succ\cu M$.  Since $\cu{N}_e$ is $\kappa$-saturated, some $c\in N$ satisfies $\Gamma(x)$ in $(\cu{N}_e)_A$.
Since $\cu{M}$ is $\kappa$-saturated, there exists $b\in M$ such that
$\tp_{\cu{M}}(b/A)=\tp_{\cu{N}}(c/A),$ so $\tp_{\cu{N}}(b/A)=\tp_{\cu{N}}(c/A).$
By Corollary \ref{c-type-over-A}, $\tp_{\cu{N}_e}(b/A)=\tp_{\cu{N}_e}(c/A).$  Therefore $b$ satisfies $\Gamma(x)$ in $\cu{N}_e$, and hence also satisfies $\Gamma(x)$ in $\cu{M}_e$.
\end{proof}

\begin{cor}  \label{c-special-metric-expansion}   $\cu{M}$ is $\kappa$-special if and only if $\cu{M}_e$ is $\kappa$-special.
\end{cor}

\begin{proof}  By Corollary \ref{c-expansion-elem-ext} and  Proposition \ref{p-saturated-metric-expansion}.
\end{proof}

\subsection{Definable Predicates}

The notion of a definable predicate in $\cu M$ was introduced in Definition \ref{d-definable}.
In Proposition \ref{p-definable-predicate} below we will show that the general property  [$\sa P$ is a definable predicate] is absolute.
The next lemma  will be used several times in this paper.

\begin{lemma}  \label{l-expansion-converge}  Let $T_e$ be a pre-metric expansion of $T$.
For every $V_D$-formula $\varphi(\vec x)$, there is a Cauchy sequence of $V$-formulas $\langle\varphi_m(\vec x)\rangle_{m\in\BN}$
in $T$ such that for every general model $\cu M$ of $T$,
$\varphi^{\cu{M}_e}=[\lim \varphi_m]^{\cu{M}}$.  Hence $\varphi^{\cu{M}_e}$ is definable in $\cu{M}$.
\end{lemma}

\begin{proof}
Let $\langle d\rangle=\langle d_m(u,v)\rangle_{m\in\BN}$ be an approximate distance for $T_e$,
such that  no bound variable in $d_m(u,v)$ occurs in $\varphi(\vec x)$.
Let $\Psi$ be the set of all subformulas of $\varphi(\vec x)$.
 For every $\psi\in\Psi$, let $\psi_m$ be the $V$-formula
obtained by replacing every subformula of $\psi$ of the form $D(\sigma,\tau)$ by $d_m(\sigma,\tau)$, where $\sigma,\tau$ are $V$-terms.  It then follows by induction
on complexity that for every $\psi\in\Psi$ we have:
$$\langle\psi_m^{\cu M}\rangle_{m\in\BN} \mbox{ is Cauchy  in } T \mbox{ and for each } \cu M\models T, \psi^{\cu {M}_e}=[\lim\psi_m]^{\cu M}.$$
In particular, this holds when $\psi=\varphi(\vec x)$, as required.
\end{proof}

\begin{prop}  \label{p-definable-predicate}
Let $\sa P\colon M^k\to[0,1]$.  The general property \emph{[$\sa P$ is a definable predicate]} is absolute.
\end{prop}

\begin{proof}
It is clear that if $\sa P$ is definable in $\cu M$, then  $\sa P$ is definable in $\cu M_e$.
Suppose that $\sa P$ is definable in $\cu M_e$.
Then $\sa P=[\lim \varphi_m]^{\cu M_e}$ for some sequence $\langle\varphi_m\rangle_{m\in\BN}$ of $V_D$-formulas
that is Cauchy in $\Th(\cu M_e)$.  By Lemma \ref{l-expansion-converge}, each $\varphi_m^{\cu M_e}$ is definable in $\cu M$,
so for each $m$ there is a $V$-formula $\psi_m(\vec x)$ such that
$$(\forall \vec b\in M^k)|\psi_m^{\cu M}(\vec b)-\varphi_m^{\cu M_e}(\vec b)|\le 2^{-m}.$$
Then $\sa P=[\lim\varphi_m]^{\cu M_e}=[\lim\psi_m]^{\cu M}$, so $\sa P$ is definable in $\cu M$.
\end{proof}

The following corollary shows that when we add countably many predicates that are definable in $\cu{M}$
to a pre-metric expansion of $\cu{M}$, we still have a pre-metric expansion.

\begin{cor} \label{c-expand-definable} Suppose $V'=V\cup W$ where $W$ is a countable set of new predicate symbols,  $\cu{M}$ is a $V$-structure,
$\cu{M}_e$ is a pre-metric expansion of $\cu{M}$, and $\cu{M}'=(\cu{M}, P^{\cu{M}'})_{P\in W}$ is a $V'$-structure such that
$P^{\cu{M}'}$ is a definable predicate in $\cu{M}$ for each $P\in W$.
Then the $V'_D$-structure $(\cu{M}_e,P^{\cu{M}'})_{P\in W}$ is a pre-metric expansion of $\cu{M}'$.
\end{cor}

\begin{proof}  Let $\langle d\rangle=\langle d_m\rangle_{m\in\BN}$ be an approximate distance for $T_e$.
By Lemma \ref{l-def-predicate-uniform},
for each predicate symbol $P\in W$, $P^{\cu {M}'}$ is uniformly continuous with respect to $D^{\cu{M}_e}$
with some modulus of uniform continuity $\triangle_P$.  Let $L_f$ agree with $L_e$
on $V$ and give each new predicate symbol $P\in W$ the modulus of uniform continuity $\triangle_P$.
 Then
$$ \cu{M}'_f := (\cu M',[\lim d_m]^{\cu M'})=(\cu M',D^{\cu M_e})=(\cu{M}_e,P^{\cu{M}'})_{P\in W}$$
is a pre-metric  expansion of $\cu M'$ with signature $L_f$.
\end{proof}

Given a set of parameters $A\subseteq M$, we say that a mapping $\sa P\colon M^k\to[0,1]$ is a \emph{definable predicate over}  $A$ in $\cu{M}$ if $\sa P$ is
a definable predicate in
$\cu{M}_A$.  In view of Remark \ref{r-definitional-expansion-parameters}, all of the results in this section hold for definable predicates over $A$ in $\cu{M}$.

\begin{df}  We say that a general structure $\cu N$ whose vocabulary $W$ contains $V$ \emph{admits quantifier elimination over} $V$ if for every
$W$-formula $\varphi(\vec x)$, $\varphi^{\cu N}$ is defined by a sequence of quantifier-free $V$-formulas in $\cu N$.
\end{df}

Note that admitting  quantifier elimination over $V$ is a stronger property than admitting  quantifier elimination over $V_D$.

\begin{cor}
\noindent\begin{itemize}
\item[(i)] $\cu M$ admits quantifier elimination over $V$ if and only if $\cu M_e$ admits quantifier elimination over $V$.
\item[(ii)] If $D^{\cu M_e}(x,y)$ is defined in $\cu M_e$ by a sequence of quantifier-free $V$-formulas, and $\cu M_e$ admits elimination of quantifiers over $V_D$,
then $\cu M$ admits elimination of quantifiers over $V$.
\end{itemize}
\end{cor}

\begin{proof} This follows easily from Lemma \ref{l-expansion-converge}.
\end{proof}

\subsection{Topological and Uniform Properties}   \label{s-top-unif}

The following definition was given in [BBHU] for metric structures, but makes sense for all general structures.

\begin{df}  A set $C\subseteq M^k$ is said to be \emph{type-defined} by $\Phi(\vec x)$ in a general structure $\cu M$,
and that $C$ is \emph{type-definable in} $\cu M$,  if $\Phi(\vec x)$ is a set of formulas
in the vocabulary of $\cu M$ with parameters in $M$, and
$$C=\{\vec c\in M^k\colon \cu M\models \Phi(\vec c)\}.$$
\end{df}

Note that type definability in $\cu M$  is preserved under finite unions and arbitrary intersections.

\begin{df}
Let $\cu N$ be a pre-metric structure with distinguished distance predicate $d$.  A set $C\subseteq N^k$ is \emph{closed} in $\cu N$ if
it is closed with respect to the pseudo-metric $d(\vec x,\vec y)=\max_{i\le k} d(x_i,y_i)$ on $N^k$.
\end{df}

\begin{lemma}  \label{l-closed-type-definable}
Let $\cu N$ be a pre-metric structure.  A set $C\subseteq N^k$ is closed in $\cu N$ if and only if
$C$ is type-definable in $\cu N$.
\end{lemma}

\begin{proof}
Assume $C\subseteq N^k$ is closed in $\cu N$.  For each $\vec b\in N^k\setminus C$ let $\varepsilon_{\vec b}=\inf_{\vec c\in C} \max_{i\le k} d(b_i,c_i),$
so $\varepsilon_{\vec b}>0$.  Then $C$ is type-defined in $\cu N$ by the set of formulas
$$\Phi(\vec x)=\{\varepsilon_{\vec b}\dotle \max_{i\le k}d(b_i,x_i)\colon \vec b\in N^k\setminus C\}.$$

Now suppose $C$ is type-defined in $\cu N$ by some set of formulas $\Psi(\vec x)$.
By Fact \ref{f-t3.5}, for each $\psi(\vec x)\in\Psi(\vec x)$,
the set $\{\vec c\in N^k\colon \cu N\models \psi(\vec c)\}$ is closed in $\cu N$.  Therefore $C$ is closed in $\cu N$.
\end{proof}

\begin{prop}  \label{p-closure}
The property of a set $C\subseteq M^k$ being closed in $\cu M$ is absolute.
$C$ is closed in $\cu M$ if and only if $C$ is type-definable in $\cu M$.
\end{prop}

\begin{proof}  By taking a subsequence if necessary, we can find an approximate distance $\langle d_m\rangle_{m\in\BN}$ for $T_e$ such that
$$T_e\models \sup_x\sup_y |d_m(x,y)-D(x,y)|\dotle 2^{-m}$$
for each $m$.  Let $C\subseteq M^k$.
By Corollary \ref{c-absolute-unique} and Lemma \ref{l-closed-type-definable}, it suffices to prove that   $C$ is type-definable in $\cu M_e$
if and only if $C$ is type-definable in $\cu M$.  It is clear that if $C$ is type-definable in $\cu M$ then $C$ is type-definable in $\cu M_e$

Suppose $C$ is type-defined  in $\cu M_e$ by  $\Phi(\vec x)$.
Let $\Psi(\vec x)$ be the set of all $V$-formulas $\psi(\vec x)$ with parameters in $M$ such that
$C\subseteq\psi^{\cu M_e}$, and let $B=\{\vec b\colon\cu M_e\models \Psi(\vec b)\}$.
Then $B=\{\vec b\colon\cu M\models \Psi(\vec b)\}$, so $B$ is type-definable in $\cu M$.
To prove that $C$ is type-definable in $\cu M$  we show that $B=C$.  Clearly $C\subseteq B$.
Let $\vec b\in B\setminus C$.  By Lemma \ref{l-closed-type-definable}, $C$ is closed in $\cu M_e$, so there is an $\varepsilon>0$
such that $\varepsilon\le\max_{i\le k} D^{\cu M_e}(b_i,c_i)$
for all $\vec c\in C$.  Hence for each $m\in\BN$, we have $\varepsilon\le(\max_{i\le k} d_m^{\cu M}(b_i,c_i)+2^{-m})$ for all $\vec c\in C$.  Therefore the
$V$-formula
$$\varepsilon\dotle(\max_{i\le k} d_m(b_i,x_i))\dotplus 2^{-m})$$
 belongs to $\Psi(\vec x)$.  But then
$$\cu M_e\models(\varepsilon\dotle\max_{i\le k} d_m(b_i,b_i))\dotplus 2^{-m}$$
 for each $m\in\BN$, which contradicts the fact that $D(x,y)=\lim_{m\to\infty} d_m(x,y)$ in $\cu M_e$.
\end{proof}

It follows that every pre-metric expansion $\cu M_e$  has the same topology as $\cu M$.
Hence all topological properties of subsets $C\subseteq M^k$ or sequences in $M^k$ are absolute.  For example, the properties that $C$ is dense, that $C$ is compact,
and that $\lim_{n\to\infty} c_n\doteq c$, are absolute.  We let $\cl_{\cu M}(C)$
denote the closure of $C$ in $\cu M$.  Then $\cl_{\cu M}(C)=\cl_{\cu M_e}(C)$.

We now obtain absolute versions of properties related to uniform convergence.

\begin{prop}  \label{d-cauchy-sequence}  The property \emph{[$\langle c_n\rangle_{n\in\BN}$ is Cauchy in $\cu M$]} is absolute.
A sequence $\langle c_n\rangle_{n\in\BN}$ is Cauchy in a general structure $\cu M$ if and only if
$\langle c_n\rangle_{n\in\BN}$
converges to some point in some elementary extension $\cu M'\succ\cu M$.
\end{prop}

\begin{proof}
 It suffices to show that the following are equivalent:
\begin{itemize}
\item[(a)]  $\langle c_n\rangle_{n\in\BN}$ is Cauchy in $\cu M_e$.
\item[(b)]  $\langle c_n\rangle_{n\in\BN}$ converges  in some  elementary extension of $\cu M_e$.
\item[(c)]  $\langle c_n\rangle_{n\in\BN}$ converges  in some  elementary extension of $\cu M$.
\end{itemize}

It is clear that (a) $\Leftrightarrow$ (b).

(b) $\Rightarrow$ (c):  Assume (b).
Then there is a general structure $\cu N$ and a point $c\in N$ such that
$\cu N_e\succ\cu M_e$ and $\lim_{n\to\infty} c_n\doteq c$ in $\cu N_e$.
By Corollary \ref{c-expansion-elem-ext}, $\cu N\succ\cu M$.
By absoluteness of convergence, $\lim_{n\to\infty} c_n\doteq c$ in $\cu N$, so (c) holds..

(c) $\Rightarrow$ (b): Suppose $\cu N\succ\cu M$, $c\in N$, and $\lim_{n\to\infty} c_n\doteq c$ in $\cu N$.
Since convergence of a sequence is absolute, $\lim_{n\to\infty} c_n\doteq c$ in $\cu N_e$.
By Corollary \ref{c-expansion-elem-ext}, $\cu N_e\succ\cu M_e$, so (b) holds.
\end{proof}

We say that a general structure $\cu M$ is \emph{complete} if every pre-metric expansion $\cu M_e$ of $\cu M$ is a metric structure.

\begin{cor}  \label{c-absolute-complete}
The property of being complete is absolute.  A  general structure $\cu M$ is complete
if and only if $\cu M$ is reduced and every Cauchy sequence in $\cu M$ converges to some point in $M$.
\end{cor}

\begin{cor}  \label{c-saturated-complete}
Every $\aleph_1$-saturated reduced structure is complete.
\end{cor}

\begin{proof}  Suppose $\cu M$ is reduced and $\aleph_1$-saturated.
By Lemma \ref{l-reduced-Me} (ii) and Proposition \ref{p-saturated-metric-expansion}, $\cu M_e$ is a reduced $\aleph_1$-saturated pre-metric structure,
and hence is complete.
Let $\langle c_n\rangle_{n\in\BN}$ be Cauchy in $\cu M$. By Proposition \ref{d-cauchy-sequence},
$\langle c_n\rangle_{n\in\BN}$ is Cauchy in $\cu M_e$.
Therefore $\langle c_n\rangle_{n\in\BN}$ converges to some point $c$ in $\cu M_e$.
By the absoluteness of convergence, $\langle c_n\rangle_{n\in\BN}$ converges to some point $c$ in $\cu M$.
\end{proof}

Recall that if $\cu M$ is a pre-metric structure, then a completion of $\cu M$ is a metric structure $\cu N$ such that the reduction of $\cu M$
is a dense elementary substructure of $\cu N$.

\begin{df}  We say that a general structure $\cu N$ is a \emph{completion} of a general structure $\cu M$ if $\cu N$ is complete, and the reduction of
$\cu M$ is a dense elementary substructure of $\cu N$.
\end{df}

The elementary substructure requirement in the above definition  ensures that the approximate distances are preserved when passing from $\cu M$ to $\cu N$.

\begin{cor}  \label{c-general-completion}
Let $\cu M, \cu N$ be general models of $T$ and let $T_e$ be a pre-metric expansion of $T$.  Then $\cu N$ is a completion of $\cu M$
if and only if $\cu N_e$ is a completion of $\cu M_e$.
\end{cor}

\begin{proof}  Let $\cu M'$ be the reduction of $\cu M$.  By Lemma \ref{l-reduced-Me} (ii), $\cu M'_e$ is the reduction of $\cu M_e$.
By Corollary \ref{c-absolute-complete},  $\cu N_e$ is reduced and complete if and only if $\cu N$ is reduced and complete.
By Proposition \ref{p-closure}, $\cu M'_e$ is dense in $\cu N_e$ if and only if $\cu M'$ is dense in $\cu N$.
By Corollary \ref{c-expansion-elem-ext}, $\cu M'_e\prec\cu N_e$ if and only if $\cu M'\prec\cu N$.
\end{proof}

\begin{prop}  \label{p-completion-unique}
Every general structure $\cu M$ has a completion, which is unique up to an isomorphism that is the identity on the reduction of $\cu M$.
\end{prop}

\begin{proof}
By Theorem  \ref{t-metric-expansion-exists}, there is a pre-metric expansion $T_e$ of $\Th(\cu M)$.   Let $\cu M'$ be the reduction of $\cu M$.
Then $\cu M_e$ is a pre-metric structure,  $\cu M'_e$ is the reduction of $\cu M_e$, and there exists a completion $\cu N_1$ of $\cu M_e$
that is unique up to an isomorphism that is the identity on $M'$.  Moreover, $\cu N_1\succ\cu M'_e$.  Then
$\cu N_1=\cu N_e$ where $\cu N$ is the $V$-part of $\cu N_1$.  By Corollary \ref{c-general-completion}, $\cu N$ is a completion of $\cu M$.

If $\cu N'$ is another completion of $\cu M$, then by Corollary \ref{c-general-completion}, $\cu N'_e$ is a completion of $\cu M_e$,
so there is an isomorphism $h\colon \cu N_e\cong \cu N'_e$ that is the identity on $M'$.  Then $h\colon \cu N\cong \cu N'$, as required.
\end{proof}

Let $\kappa$ be an infinite cardinal.  A complete metric theory $T$ is $\kappa$-categorical if every two complete models of $T$ of density character $\kappa$
are isomorphic.

\begin{cor}  \label{c-categorical}  The property of having a $\kappa$-categorical complete theory is absolute.  A complete general theory $T$ is
$\kappa$-categorical if and only if every two complete general models of $T$ of density character $\kappa$ are isomorphic.
\end{cor}

\begin{proof}  By Proposition \ref{p-closure} (closed is absolute), Corollary \ref{c-absolute-complete} (being complete is absolute),
and Lemma \ref{l-reduced-Me} (iii) (being isomorphic is absolute).
\end{proof}

\subsection{Infinitary Continuous Logic}  \label{s-infinitary}

We return to the infinitary  continuous formulas that were introduced in [Ea15] and discussed in Example \ref{e-Eagle}.
We assume in this subsection that $|V|\le\aleph_0$.
$\cu L_{\omega_1\omega}(V)$ denotes the set of all continuous $\cu L_{\omega_1\omega}$-formulas over the vocabulary $V$.

\begin{lemma}  \label{l-infinitary-approx}
Let $T_e$ be a pre-metric expansion of a $V$-theory $T$.  For every formula  $\psi(\vec x)\in\cu L_{\omega_1\omega}(V_D)$,
there is a formula $\varphi(\vec x)\in\cu L_{\omega_1\omega}(V)$ such that $\psi^{\cu M_e}=\varphi^{\cu M}$
for every general $V$-model $\cu M$ of $T$.
\end{lemma}

\begin{proof}
The pre-metric expansion $T_e$ has an exponentially Cauchy approximate distance $\langle d_m(u,v)\rangle_{m\in\BN}$
such that no bound variable in $d_m(u,v)$ occurs in $\psi(\vec x)$ or in $\vec z$.  Then for each $m$
and pair of $V$-terms $\sigma(\vec z), \tau(\vec z)$, $d_m(\sigma(\vec z),\tau(\vec z))$
is a $V$-formula, and for every general model $\cu M$ of $T$ and tuple $\vec z$ in $M$ we have
$$D^{\cu M_e}(\sigma(\vec z),\tau(\vec z))-2^{-m}\le d_m^{\cu M}(\sigma(\vec z),\tau(\vec z))\le D^{\cu M_e}(\sigma(\vec z),\tau(\vec z))+2^{-m}.$$
It follows that
$$D^{\cu M_e}(\sigma(\vec z),\tau(\vec z))=\inf_m\sup_k d_{m+k}^{\cu M}(\sigma(\vec z),\tau(\vec z))=\sup_m\inf_k d_{m+k}^{\cu M}(\sigma(\vec z),\tau(\vec z)).$$
Let $\psi\in\cu{L}_{\omega_1\omega}(V_D)$, and let $\varphi$ be the $\cu{L}_{\omega_1\omega}$-formula in the vocabulary $V$
obtained by replacing each atomic subformula of $\psi$ of the form $D(\sigma(\vec z),\tau(\vec z))$ by $\inf_m\sup_k d_{m+k}(\sigma(\vec z),\tau(\vec z))$.
It follows by induction on the complexity of $\psi$ that $\psi^{\cu M_e}=\varphi^{\cu M}$
for every general $V$-model $\cu M$ of $T$.
\end{proof}

We say that a mapping $\sa P\colon M^k\to[0,1]$ is \emph{$\cu{L}_{\omega_1\omega}$-definable in} a general structure $\cu M$
if $\sa P=\varphi^{\cu M}$ for some  $\cu{L}_{\omega_1\omega}$-formula $\varphi(\vec x)$ with $|\vec x|=k$ in the vocabulary of $\cu M$.
It is easily seen that if $\sa P$ is definable in $\cu M$ then $\sa P$ is $\cu{L}_{\omega_1\omega}$-definable in $\cu M$.

\begin{prop}  \label{p-infinitary} The property \emph{[$\sa P$ is $\cu{L}_{\omega_1\omega}$-definable in $\cu M$]} is absolute.
\end{prop}

\begin{proof}  Let $\cu M_e$ be a pre-metric expansion of $\cu M$.  It is easily seen by induction on complexity that for
each $\cu{L}_{\omega_1\omega}$-formula $\varphi(\vec x)$ in the vocabulary $V$ of $\cu M$ we have $\varphi^{\cu M}=\varphi^{\cu M_e}$,
so $\cu{L}_{\omega_1\omega}$-definability in $\cu M$ implies $\cu{L}_{\omega_1\omega}$-definability in $\cu M_e$.
Conversely, by Lemma \ref{l-infinitary-approx}, $\cu{L}_{\omega_1\omega}$-definability in $\cu M_e$ implies $\cu{L}_{\omega_1\omega}$-definability in $\cu M$.
\end{proof}

We now generalize several results in [Ea15] from metric theories to general theories.

\begin{prop}  \label{p-scott}  Let $\cu M$ be a separable complete $V$-structure.   There is an $\cu L_{\omega_1\omega}$-sentence $\varphi$
in the vocabulary $V$, called a \emph{Scott sentence} of $\cu M$, such that for every separable complete $V$-structure $\cu N$, $\varphi^{\cu N}=0$ if
$\cu M\cong\cu N$ and $\varphi^{\cu N}=1$ otherwise.
\end{prop}

\begin{proof} By Theorem 3.2.1 of [Ea15] (which follows from [BDNT]), the result holds for metric structures.  By the Expansion Theorem \ref{t-metric-expansion-exists},
the empty $V$-theory $T$ has a pre-metric expansion $T_e$ with signature $L_e$.  By the preceding section, $\cu M_e$ is a separable complete
pre-metric structure, and hence is a metric structure with signature $L_e$.
Therefore there is an $\cu L_{\omega_1\omega}$-sentence $\psi$ in the vocabulary $V_D$
such that for every  separable metric structure $\cu N'$ with signature $L_e$, $\psi^{\cu N'}=0$ if $\cu M_e\cong\cu N'$,
and $\psi^{\cu N'}=1$ otherwise.
By Lemma \ref{l-infinitary-approx}, there is an $\cu L_{\omega_1\omega}$-sentence $\varphi$ in the vocabulary $V$
such that $\psi^{\cu N_e}=\varphi^{\cu N}$ for every $V$-structure $\cu N$.
Moreover, every metric structure $\cu N'$ with signature $L_e$ is a general model of $T_e$, and hence is equal to $\cu N_e$ where $\cu N$ is the $V$-part of $\cu N'$,
and $\cu N$ is a separable complete $V$-structure.
By Lemma \ref{l-reduced-Me} (iii),  $\cu M\cong\cu N$ if and only if $\cu M_e\cong\cu N_e$.
Therefore $\varphi$ has the required property that for every separable complete $V$-structure $\cu N$,
$\varphi^{\cu N}=0$ if $\cu M\cong\cu N$ and $\varphi^{\cu N}=1$ otherwise.
\end{proof}

\begin{prop}  \label{p-scott-isom}  Let $\cu M$ be a separable complete general structure and let $\sa P$ be a mapping from $M^{n}$ into $[0,1]$.  The following are equivalent:
\begin{itemize}
\item[(i)]  $\sa P$ is $\cu L_{\omega_1\omega}$-definable in $\cu M$.
\item[(ii)]  $\sa P$ is fixed by all automorphisms of $\cu M$.
\end{itemize}
\end{prop}

\begin{proof}  By Theorem 3.2.3 of [Ea15], the  result holds for  metric structures.  By the preceding section, Lemma \ref{l-reduced-Me} (iii),
and Proposition \ref{p-infinitary}, both (i) and (ii) are absolute for complete separable general structures.
\end{proof}

Proposition \ref{p-ott} below extends the Omitting Types Theorem 3.3.4 of [Ea15] to general theories $T$.
The result in [Ea15] is about reduced pre-metric structures, which are pre-metric structures whose distinguished distance is a metric, but is not necessarily complete.

We first need the notion of a strong countable fragment of $\cu L_{\omega_1\omega}(V)$.  By [BBHU], Section 6, we may fix a countable set $\cu F$ of connectives such that $\cu F$
is closed under composition and projections, and for each $n\in\BN$ the set of $n$-ary connectives $C\in\cu F$ is uniformly dense in the
set of all $n$-ary connectives.  We also assume that $\cu F$ contains the connective $\dotle$ and each rational constant in $[0,1]$.  Following [Ea15],
by a \emph{countable fragment} of $\cu L_{\omega_1\omega}(V)$  we mean a countable set $\cu L\subseteq\cu L_{\omega_1\omega}(V)$ that
contains the set of atomic formulas over $V$, and is closed under the connectives in $\cu F$, $\sup_x$, $\inf_x$, subformulas, and substituting terms for free variables.
The smallest countable fragment of $\cu L_{\omega_1\omega}(V)$ is the set of all (finitary) $V$-formulas with connectives in $\cu F$.

\begin{df} \label{d-strongfragment}
We say that $\cu L$ is a \emph{strong countable fragment of} $\cu L_{\omega_1\omega}(V)$ if there is a countable fragment
$\cu L_D$ of $\cu L_{\omega_1\omega}(V_D)$ such that $\cu L=\cu L_D\cap\cu L_{\omega_1\omega}(V)$, and for each finitary $V$-formula
$\theta(x,y)$ with connectives in $\cu F$ and pair of $V$-terms $\sigma(\vec u),\tau(\vec v)$, $\cu L_D$ is closed under the operation
of replacing each subformula of the form $D(\sigma(\vec u),\tau(\vec v))$ by $\theta(\sigma(\vec u),\tau(\vec v))$.
\end{df}

Clearly, every strong countable fragment of $\cu L_{\omega_1\omega}(V)$ is a countable fragment of $\cu L_{\omega_1\omega}(V)$.
Also, the smallest countable fragment of $\cu L_{\omega_1\omega}(V)$ is a strong countable fragment of $\cu L_{\omega_1\omega}(V)$.
It is easily seen that every countable subset of $\cu L_{\omega_1\omega}(V)$
is contained in a strong countable fragment of $\cu L_{\omega_1\omega}(V)$.

\begin{df}
Let  $\cu L$ be a  countable fragment of $\cu L_{\omega_1\omega}(V)$, and let $T$ be a set of sentences in $\cu L$.
We say that a set $\Sigma(\vec x)\subseteq\cu L$  is  \emph{principal} over  $(T,\cu L)$ if
there is a formula $\varphi(\vec x)\in\cu L$, an $|\vec x|$-tuple of $V$-terms $\vec\tau(\vec y)$, and a rational $r\in(0,1)$, such that:
\begin{itemize}
\item $T\cup\{\varphi(\vec y)\}$ is satisfiable,
\item $T\cup\{\varphi(\vec y)\dotle r\}\models\Sigma(\vec\tau(\vec y))$.
\end{itemize}
\end{df}

\begin{prop}  \label{p-ott}
Let $\cu L$ be a strong countable fragment of $\cu L_{\omega_1\omega}(V)$. Let $T$ be a set of sentences in $\cu L$, and for each $n\in\BN$,
suppose $\Sigma_n(\vec x_n)\subseteq\cu L$ but $\Sigma_n(\vec x_n)$ is not principal over $(T,\cu L)$.
Then there is a reduced separable model of $T$ in which none of the sets $\Sigma_n(\vec x_n)$ is realized.
\end{prop}

\begin{proof}
By hypothesis, $\cu L=\cu L_D\cap \cu L_{\omega_1\omega}(V)$ for some countable fragment $\cu L_D$ of $\cu L_{\omega_1\omega}(V_D)$
as in Definition \ref{d-strongfragment}.  Let $U_e$ be a  pre-metric expansion of the empty $V$-theory $U$.
Then every pre-metric model of $T\cup U_e$ is of the form $\cu M_e$ for some general model $\cu M$ of $T$.
\medskip

\textbf{Claim.} For each $n\in\BN$, $\Sigma_n(\vec x_n)$ is not principal over $(T\cup U_e,\cu L_D)$.
\medskip

To prove this Claim, suppose that $\Sigma_n(\vec x_n)$ is  principal over $(T\cup U_e,\cu L_D)$,
witnessed by $\varphi(\vec x_n)\in\cu L_D$, a tuple of $V$-terms $\vec \tau(\vec y)$, and a rational $r\in(0,1)$.
By approximating the connectives by connectives in $\cu F$ and taking a subsequence, we see that $U_e$ has
an exponentially Cauchy approximate distance $\langle d_m\rangle_{m\in\BN}$ where each $d_m$ is a finitary continuous formula that belongs to $\cu L$.
Hence by Fact \ref{f-metric-theory} and Lemma \ref{l-Te-axioms}, we may take $U_e$ to be a set of sentences of $\cu L_D$.
For each $m\in\BN$, let $\varphi_m(\vec x)$ be the formula obtained from $\varphi(\vec x)$
by replacing every subformula of $\varphi$ of the form $D(\sigma_1(\vec x),\sigma_2(\vec x))$
by $d_m(\sigma_1(\vec x),\sigma_2(\vec x))$.  Then $\varphi_m\in\cu  L_D\cap\cu L_{\omega_1\omega}(V)=\cu L$.
It follows by induction on the complexity of formulas  that for each $\varepsilon>0$ there is an $m\in\BN$
such that for all $k\ge m$, $\cu N\models T_e$, and $\vec b\in N^{|\vec x|}$, $\varphi_m^{\cu N}(\vec b)$
is within $\varepsilon$ of $\varphi^{\cu N}(\vec b)$.  Taking $\varepsilon=r/2$, we obtain a formula $\varphi_m(\vec x)\in\cu L$
such that $\varphi_m(\vec x), \vec\tau(\vec y)$, and $r/2$ witness that $\Sigma_n(\vec x)$ is principal over $(T,\cu L)$.
This contradicts our hypothesis that $\Sigma_n(\vec x_n)$ is not principal over $(T,\cu L)$, and proves the Claim.
\medskip

Now, by Theorem 3.3.4 of [Ea15], there is a reduced separable
pre-metric model $\cu N$ of $T\cup U_e$ in which none of the sets $\Sigma_n$ is realized.  Then the $V$-part of $\cu N$ is a
reduced separable model of $T$ in which none of the sets $\Sigma_n$ is realized.
\end{proof}

Since the set $\cu L_0(V)$ of all $V$-formulas built from connectives in $\cu F$ is a strong countable fragment of $L_{\omega_1\omega}(V)$, we have:

\begin{cor}  Let $T\subseteq\cu L_0(V)$,
and for each $\in\BN$, let $\Sigma_n(\vec x_n)$ be a non-principal subset of $\cu L_0(V)$  over $(T,\cu L_0(V))$.
Then there is a reduced separable model of $T$ in which none of the sets $\Sigma_n(\vec x_n)$ is realized.
\end{cor}

\subsection{Many-sorted Metric Structures}  \label{s-manysorted}

In Definition \ref{d-sorted-metric} we defined a sorted metric structure $\cu M$ with sorted vocabulary $W$ to be
a  reduced structure that respects the sorts of $W$ and has a complete metric $d_\BS$ on each sort $\BS$.
As a sorted vocabulary, $W$ must have an unsorted constant symbol $u$ and a unary predicate symbol $U_{\BS}$ for each sort $\BS$.

\begin{df}  We say that $\Th(\cu M)$, $\Th(\cu M_e)$ have \emph{essentially the same types} if
there is a  homeomorphism $h$ from $S_n(\Th(\cu M))$ onto $S_n(\Th(\cu M_e))$
such that for each  $\vec c\in M^n$, $h(\tp_{\cu M}(\vec c))=\tp_{\cu M_e}(\vec c)$.
\end{df}

\begin{prop}  \label{p-unbounded-metric}
Let $W$ be a sorted vocabulary with sorts $\BS_1,\BS_2,\ldots$, and $\cu M$ be a sorted metric structure
 over $W$.  Then any pre-metric expansion $\cu M_e=(\cu M,\sa D)$ of $\cu M$
is a metric structure, and $\Th(\cu M)$, $\Th(\cu M_e)$ have essentially the same types.
\end{prop}

\begin{proof}  Note that in $\cu M_e$, the new distance symbol $D$ does not have sorts assigned to its arguments.
Since $\cu M$ is reduced,  $\cu M_e$ is a reduced pre-metric structure, so $\sa D$ is a metric.
Moreover, $u^{\cu M}$ is the only sortless element by Lemma \ref{l-sortless} (ii).  We show that $\sa D$ is complete.
Let $\langle d_m(x,y)\rangle_{m\in\BN}$ be an approximate distance for $\cu M_e$.
Suppose that $\langle a_k\rangle_{k\in\BN}$ is Cauchy convergent in $\cu M_e$.

Case 1: For some sort $\BS_n$, $a_k\in\BS_n^{\cu M}$ for infinitely many $k$.  By taking a subsequence, we may assume that $a_k\in\BS_n^{\cu M}$
for all $k$.  Since $\cu M_e$ is a pre-metric structure, $d_m^{\cu M}$ is uniformly continuous with respect to $\sa D$,
so $\langle a_k\rangle_{k\in\BN}$ is Cauchy with respect to $d_m^{\cu M}$.
$d_m^{\cu M}$ is a complete metric on $\BS_m^{\cu M}$, so $\langle a_k\rangle_{k\in\BN}$ converges to some point $b\in\BS_m^{\cu M}$
with respect to $d_m^{\cu M}$.  The set $\BS_n^{\cu M}$ is defined by the formula $U_{\BS_n}(x)$ in $\cu M$, and is therefore closed in $\cu M_e$.
By Proposition \ref{p-closure}, the topologies of $\cu M_e$ and $\cu M$ restricted to $\BS_n^{\cu M}$ are the same.  Since $\cu M$ is a sorted metric structure,
that topology is the metric topology of $d_m^{\cu M}$ restricted to $\BS^{\cu M}$.  Therefore $\langle a_k\rangle_{k\in\BN}$ converges to $b$
with respect to  $\sa D$.

Case 2.  For each sort $\BS_n$, $a_k\notin\BS_n^{\cu M}$ for all but finitely many $k$.
Then $\langle a_k\rangle_{k\in\BN}$ converges to some point $b$ in the completion $\cu M'_e\succ \cu M_e$.
For each $n$, $\BS_n$ is open in $\cu M'_e$, so $b\notin\BS_n^{\cu M'_e}$.  Therefore $b=u^{\cu M'_e}=u^{\cu M_e}$.
It follows that $\langle a_k\rangle_{k\in\BN}$ converges to $u^{\cu M_e}$ with respect to $\sa D$.

In both cases, $\langle a_k\rangle_{k\in\BN}$ converges to a point with respect to $\sa D$, so $\sa D$ is complete.
The existence of the homeomorphism $h$ follows from Corollary \ref{c-homeo-types}.
\end{proof}

\subsection{Bounded and Unbounded Metric Structures}  \label{s-unbounded}

Bounded vocabularies, metric structures, and general structures, were discussed in Example \ref{e-bounded}.
Like the metric case, the model theory of  $[0,1]$-valued general structures carries over to bounded general structures in a routine way.
When working with bounded continuous logic, we will freely use the same terminology that we use in the $[0,1]$-valued case.
In particular, the Expansion Theorem shows that every bounded general structure $\cu M$ has a pre-metric expansion $\cu M_e=(\cu M,\sa D)$, which is a bounded pre-metric structure.
By normalizing $\sa D$, one can even get a pre-metric expansion where $\sa D$ has values in $[0,1]$.

Unbounded metric structures were briefly discussed in Example \ref{e-unbounded}.
We mentioned three ways that an unbounded metric structure $\cu N$ has been treated in the literature
using the model theory of metric structures.
One way, as in [BBHU], [Fa], and [BY09], was  to look at the corresponding bounded many-sorted metric structure $\cu M$
with a sort for the closed $m$-ball around a distinguished constant $0$ for each positive $m\in\BN$.
$\cu M$ can be viewed either as a many-sorted structure with a metric in each sort, or as a single sorted general structure with a unary predicate for each sort.
By the Expansion Theorem, as a single-sorted general structure,
$\cu M$ has a pre-metric expansion $\cu M_e$, and by
Proposition \ref{p-unbounded-metric}, $\Th(\cu M_e)$ has essentially the same types as $\Th(\cu M)$.

The Expansion Theorem opens up additional possibilities, which have not yet been explored in the literature, for using the model theory of
metric structures to treat unbounded metric structures.
It allows the flexibility to look at general structures that are built in some way from an unbounded metric structure, and then use
the Expansion Theorem to get a related metric structure.

For instance, consider an unbounded metric structure $\cu N$ with a real-valued gauge predicate $\nu$, as in the paper [BY08] that was discussed in Example \ref{e-unbounded} above.
To study $\cu N$, one might look at the bounded general structure $\cu N'$ such that $\cu N'$
has the same universe and functions as $\cu N$ plus a point at $\infty$, and  for each predicate $P$ of $\cu N$ (including $d$ and $\nu$), and positive $m\in \BN$, $\cu N'$ has
the predicate $P_m$ formed by  truncating $P$ at $m$, that is,
\[ P_m^{\cu N'}(\vec v)=
\begin{cases}
-m & \mbox{ if } P^{\cu N}(\vec v) \le -m\\
m & \mbox{ if } P^{\cu N}(\vec v) \ge m\\
P^{\cu N}(\vec v) & \mbox{otherwise.} \\
\end{cases}
\]
This construction merely truncates the predicates of $\cu N$ rather than adding sorts for the closed $m$-balls.
Again, $\cu N'$ will have a pre-metric expansion $\cu N'_e=(\cu N',\sa D)$, and the theories of $\cu N'$ and $\cu N'_e$ will have essentially the same types.
As usual, the distinguished metric $\sa D$ of $\cu N'_e$ may be wildly behaved and unnatural, but its existence does make the model theory of metric structures available
for the study of the unbounded metric structure $\cu N$.

Note that for any $x$ in $\cu N$, $\nu_m^{\cu N'}(x)=m$ whenever $m\le \nu^\cu N(x)$, and $\nu_m^{\cu N'}(x) = \nu^{\cu N}(x)$ whenever $m> \nu^\cu N(x)$.
So in any model of $\Th(\cu N')$,
$\lim_{m\to\infty} \nu_m(x)$ is finite if and only if $\nu_m(x)$ is eventually constant as $m\to\infty$.

Instead of looking at an arbitrary model of $\Th(\cu N')$, it may be useful to look at the substructure consisting of those elements $y$ such that $\lim_{m\to\infty} \nu_m(y)$ is finite, plus the point at $\infty$. This corresponds to looking at the modified pre-ultraproduct of unbounded metric structures that was discussed in Example \ref{e-unbounded}.
The universe of the modified pre-ultraproduct $\cu M=\prod^\cu D\cu N_i$
is the set of elements $y$ of the ordinary pre-ultraproduct $\cu M'=\prod^\cu D\cu N'_i$  such that the set $\{ \nu^{\cu N_i}(y_i)\colon i\in I\}$ has a finite bound,
plus the point at $\infty$.

To see the connection, one can check that for each element $x$ of $\cu M'$,
the following are equivalent, where $x=_\cu D y$ means $\{i\in I\colon x_i=y_i\}\in\cu D$:
\begin{itemize}
\item For some $y =_\cu D x$, $y$ belongs to $\cu M$.
\item For some $y =_\cu D x$, the set $\{ \nu^{\cu N_i}(y_i)\colon i\in I\}$ has a finite bound.
\item For some $y =_\cu D x$, the set $\{  \lim_{m\to\infty} \nu_m^{\cu N'_i}(y_i)\colon i\in I\}$ has a finite bound.
\item For some $y =_\cu D x$, $\lim_{m\to\infty} \nu^{\cu M'}_m(y)$ is finite.
\end{itemize}

\subsection{Imaginaries}

For each metric structure $\cu M$, [BU] (Section 5) introduces a sorted
metric structure $\cu M^{eq}$ that has $\cu M$ as its home sort and infinitely many sorts of imaginary elements.
In a similar way, we will  now introduce imaginary elements for
any reduced general structure $\cu M$ with the  vocabulary $V$.

In the following, we let $T$ be a complete $V$-theory,
 $\vec x$ be a tuple of variables and $\vec y$ be a countable sequence of variables, and $\langle\varphi\rangle=\langle\varphi_m(\vec x,\vec y)\rangle_{m\in\BN}$ be
 an exponentially Cauchy sequence of formulas  in $T$.
Let $V^{\langle\varphi\rangle}$
be the two-sorted vocabulary obtained from $V$ by adding a \emph{home sort} $\BS$, an \emph{imaginary sort} $\BS_{i}$,  unary predicate symbols
$U_\BS$ for the home sort and $U_{\BS_{i}}$ for the imaginary sort,  a predicate symbol $d_{\langle\varphi\rangle}$ of sort
$\BS_{i}\times \BS_{i}$, a predicate symbol $P_{\langle\varphi\rangle}$ of sort $\BS^{|\vec x|}\times\BS_i$,
and the unsorted constant symbol $u$.

Intuitively, the imaginary elements will be equivalence classes of infinite sequences of parameters in the home sort,
$P_{\langle\varphi\rangle}(\vec x,\vec y)$ will be $\lim_{m\to\infty}\varphi_m(\vec x,\vec y)$,
and $d_{\langle\varphi\rangle}(\vec y,\vec z)$ will be $\lim_{m\to\infty}(\sup_{\vec x}|\varphi_m(\vec x,\vec y)-\varphi_m(\vec x,\vec z)|).$

\begin{df}  Let  $T^{\langle\varphi\rangle}$ be the set of $V^{\langle\varphi\rangle}$-sentences
$$ \sup_{wz}[|\sup_{\vec x}|P_{\langle\varphi\rangle}(\vec x,w)-P_{\langle\varphi\rangle}(\vec x,z)|-d_{\langle\varphi\rangle}(w,z)|\dotle 0],$$
$$\{\sup_z\inf_{\vec y}\sup_{\vec x}[|\varphi_m(\vec x,\vec y)-P_{\langle\varphi\rangle}(\vec x,z)|\dotle 2\cdot 2^{-m}\colon m\in\BN\},$$
$$\{\sup_{\vec y}\inf_z\sup_{\vec x}[|\varphi_m(\vec x,\vec y)-P_{\langle\varphi\rangle}(\vec x,z)|\dotle 2\cdot 2^{-m}\colon m\in\BN\}.$$
\end{df}

$T^{\langle\varphi\rangle}$ is the natural analogue, for general structures, of the theory $T_\psi$ that is defined in [BU].

\begin{prop}  \label{p-imaginary-exists-unique}
For every reduced model $\cu M$ of $T$, there is, up to isomorphism, a unique reduced model $\cu M^{\langle\varphi\rangle}$
 of $T\cup T^{\langle\varphi\rangle}$ that respects sorts in $V^{\langle\varphi\rangle}$ and agrees with $\cu M$ in the home sort.
\end{prop}

\begin{proof} The argument is the same as in [BU], pages 29-30.
\end{proof}

 The elements of sort $\BS_{i}$ in $\cu M^{\langle\varphi\rangle}$
are called \emph{imaginary elements}, or \emph{canonical parameters.}
By Lemma \ref{l-sortless},  $u^{\cu M^{\langle\varphi\rangle}}$ is  the only element without a sort.

\begin{prop}  \label{p-expansion-imaginaries}  For each reduced model $\cu M$ of $T$, $d_{\langle\varphi\rangle}$
is a metric on the imaginary sort in $\cu M^{\langle\varphi\rangle}$.
\end{prop}

\begin{proof}   We work in $\cu M^{\langle\varphi\rangle}$.  It is clear from the definition that $d_{\langle\varphi\rangle}$ is the limit of a uniformly convergent sequence
of pseudo-metrics on $\BS_{i}$, and hence is itself a pseudo-metric on $\BS_i$.
 Suppose that  $b,c$ have sort $\BS_i$ and $d_{\langle\varphi\rangle}(b,c)=0$.  Since $\cu M^{\langle\varphi\rangle}$ is reduced, it suffices to show that $b\doteq c$.
Let $\vec z$ be a tuple in ${M}^{\langle\varphi\rangle}$, and let $\alpha(b,\vec z)$ be an atomic formula.  If $\alpha$ begins with a predicate symbol other than
$d_{\langle\varphi\rangle}$ or $P_{\langle\varphi\rangle}$ or $U_{\BS_{i}}$, then $b$ and $c$ have the wrong sort, so  $\alpha(b,\vec z)=\alpha(c,\vec z)=1$.

Suppose $\alpha$ begins with the predicate symbol $d_{\langle\varphi\rangle}$, so $\alpha(b,\vec z)$ has the form
$$d_{\langle\varphi\rangle}(\sigma(b,\vec z),\tau(b,\vec z))$$
where $\sigma, \tau$ are terms.  No function or constant symbol in $V^{\langle\varphi\rangle}$
has  value sort $\BS_i$.  So if either $\sigma$ or $\tau$ starts with a function or constant symbol,
then $\alpha(b,\vec z)=\alpha(c,\vec z)=1$.  The only other possibilities are that for some variable $z_0$, $\alpha(b,\vec z)$ is either
$d_{\langle\varphi\rangle}(b,b)$, $d_{\langle\varphi\rangle}(z_0,z_0)$, $d_{\langle\varphi\rangle}(b,z_0)$, or $d_{\langle\varphi\rangle}(z_0,b)$.
In each of those cases, since $d_{\langle\varphi\rangle}(b,c)=0$, we have  $\alpha(b,\vec z)=\alpha(c,\vec z).$

If $\alpha$ begins with the predicate symbol $P_{\langle\varphi\rangle}$ or  $U_{\BS_{i}}$, then
an argument that is similar to  the preceding paragraph again shows that
 $\alpha(b,\vec z)=\alpha(c,\vec z)$.  Thus in all cases we have
$\alpha(b,\vec z)=\alpha(b,\vec z)$ and hence $b\doteq c$.
\end{proof}

\begin{prop}  \label{p-imaginary-special}  If $\cu M$ is $\kappa$-special, then
${\cu M}^{\langle\varphi\rangle}$ is $\kappa$-special.  If in addition $\kappa$ has uncountable cofinality,
then $d_{\langle\varphi\rangle}$
is a complete metric on the imaginary sort.
\end{prop}

\begin{proof}
By the Existence Theorem for Special Models, there is a $\kappa$-special model $\cu N$ of $T\cup T^{\langle\varphi\rangle}$.
It is easily checked that the $V$-part $\cu M'$ of $\cu N$ is $\kappa$-special, and $\cu N=\cu M'^{\langle\varphi\rangle}$.
By the Uniqueness Theorem for Special Models, $\cu M'\cong\cu M$, so $\cu M'^{\langle\varphi\rangle}\cong \cu M^{\langle\varphi\rangle}$.
If $\kappa$ has uncountable cofinality, then $\cu M^{\langle\varphi\rangle}$ is $\aleph_1$-saturated, and as in Remark \ref{r-saturated-metric},
it follows that $d_{\langle\varphi\rangle}$
is a complete metric on the imaginary sort.
\end{proof}

With more work, one can show that if $\cu M$ is $\kappa^+$-saturated, then ${\cu M}^{\langle\varphi\rangle}$ is $\kappa^+$-saturated,
but we will not need that fact here.

\begin{df}
We say that two exponentially Cauchy sequences $\langle\varphi\rangle$ and $\langle\psi\rangle$
are \emph{equivalent} in $T$ if $[\lim\varphi_m]^{\cu M}=[\lim\psi_m]^{\cu M}$ for all general models $\cu M$ of $T$.
\end{df}

\begin{rmk}  \label{r-exponential-equivalent}
If $\cu M$ is a reduced model of $T$ and $\langle\varphi\rangle$ and $\langle\psi\rangle$ are equivalent exponentially Cauchy sequences in $T$,
then $\cu M^{\langle\varphi\rangle}=\cu M^{\langle\psi\rangle}$,
 in the sense that they have
the same universes and the identity function is an isomorphism between them.
\end{rmk}

We now define the sorted vocabulary $V^{eq}$ and structure $\cu M^{eq}$.  Let $\Phi$ be a set of
exponentially Cauchy sequences $\langle\varphi\rangle$ for $\cu M$ that contains exactly one member of each equivalence class.
For each $\langle\varphi\rangle\in\Phi$, take a copy of ${\cu M}^{\langle\varphi\rangle}$ in such a way that
the home sorts are all the same, and the imaginary
sorts of ${\cu M}^{\langle\varphi\rangle}$ and ${\cu M}^{\langle\psi\rangle}$ are disjoint when $\langle\varphi\rangle\ne\langle\psi\rangle$.
Denote the imaginary sort of ${\cu M}^{\langle\varphi\rangle}$ by $\BS_{i}^{\langle\varphi\rangle}$.

Let $V^{eq}$ be the union of the sorted vocabularies
$\{V^{\langle\varphi\rangle}\colon  \langle\varphi\rangle\in\Phi\}$.
Finally, let $\cu M^{eq}$ be a $V^{eq}$-structure that respects sorts and is a common expansion of the structures
$\{{\cu M}^{\langle\varphi\rangle}\colon \langle\varphi\rangle\in\Phi\}$.

\begin{rmk}  \label{r-Meq-unique} It follows from Remark \ref{r-exponential-equivalent} that
for each reduced  $V$-structure $\cu M$, $\cu M^{eq}$ is unique up to isomorphism.
\end{rmk}

Note that $V^{eq}$ may have uncountably many sorts and hence uncountably many predicate symbols.  To avoid that difficulty,
one can restrict things to countable sets of sorts.  For any  set $\Theta\subseteq \Phi$, let $V^\Theta$ and
$\cu M^\Theta$ be the parts of $V^{eq}$ and
$\cu M^{eq}$ with only the main sort and the imaginary sorts $\BS_{i}^{\langle\varphi\rangle}$ where $\langle\varphi\rangle \in \Theta$ .
Whenever $\Theta$ is countable, the vocabulary $V^{\Theta}$   has countably many predicate and function symbols.

The next result concerns pre-metric expansions of $\cu M$.

\begin{prop}  \label{p-Meeq}  Suppose $\cu M$ is a reduced  $V$-structure, and $\cu M_e=(\cu M,\sa D)$ is a pre-metric expansion of $\cu M$ with
vocabulary $V_D$.  Then $(V_D)^{eq}$ has the same sorts
as $V^{eq}$, and $(\cu M_e)^{eq}$ is a sorted pre-metric structure whose signature  has:
\begin{itemize}
\item the distinguished distance $D$ in the home sort,
\item the distinguished distance $d_{\langle\varphi\rangle}$ in each imaginary sort $\BS_{i}^{\langle\varphi\rangle}$,
\item the same modulus of uniform continuity as $\cu M_e$ has for each symbol of $V$,
\item the modulus of uniform continuity for $P_{\langle\varphi\rangle}$ that $\cu M_e$ has for $[\lim \varphi_m]$,
\item the trivial modulus of uniform continuity for  $U_{\BS}$ and each $U_{\BS_{i}^{\langle\varphi\rangle}}$.
\end{itemize}
The same holds with a set $\Theta\subseteq \Phi$ in place of $eq$.
\end{prop}

\begin{proof}   Since every exponentially Cauchy sequence of formulas in $T$ is also
exponentially Cauchy in $T_e$,  every sort of $\cu M^{eq}$ is also a sort of $(\cu M_e)^{eq}$.
Also, for each $\langle\varphi\rangle\in\Phi$, $({\cu M}_e)^{\langle\varphi\rangle}=({\cu M}^{\langle\varphi\rangle},\sa D)$ is
the structure obtained by adding the predicate $\sa D$ in the home sort.
If $\langle\varphi\rangle$ and $\langle\psi\rangle$
are equivalent in $T$ then they are equivalent in $T_e$.  The proof of Lemma \ref{l-expansion-converge} shows that
every exponentially Cauchy sequence $\langle\varphi\rangle$ in $T_e$ is equivalent in $T_e$ to an exponentially Cauchy sequence $\langle\psi\rangle$ in $T$,
so by Remark \ref{r-exponential-equivalent} we have $({\cu M}_e)^{\langle\varphi\rangle}=({\cu M}_e)^{\langle\psi\rangle}$.
It follows that $(V_e)^{eq}$ has the same sorts as $V^{eq}$, and $(\cu M_e)^{eq}$ with the signature described above is a pre-metric expansion of $\cu M^{eq}$.
\end{proof}

Assume the hypotheses of Proposition \ref{p-Meeq}, and let $\Theta\subseteq \Phi$.
One can combine all the distinguished distances in $\cu M_e^{\Theta}$ into a single metric $\sa D^{\Theta}$ in the following way.
If $x,y$ are in the home sort, $\sa{D}^{\Theta}(x,y)=\sa{D}(x,y)$.  If $x,y$ have the same imaginary sort $\BS_{i}^{\langle\varphi\rangle}$,
then $\sa{D}^{\Theta}(x,y)=d_{\langle\varphi\rangle}(x,y)$.  If neither $x$ nor $y$ has a sort, then $\sa{D}^{\Theta}(x,y)=0$.
Otherwise, $\sa{D}^{\Theta}(x,y)=1$. According to our definition, the structure $(\cu M^{\Theta},\sa{D}^{\Theta})$ does not respect sorts.
However,  $\cu M^{\Theta}$ is a general structure whose vocabulary is $V^\Theta$,
and $(\cu M^{\Theta},\sa D^{\Theta})$ is a $V_{D}^{\Theta}$-structure.

\begin{cor}  \label{c-S-premetric}  Assume the hypotheses of  Proposition \ref{p-Meeq}, and $\Theta\subseteq \Phi$ is finite.
Then $(\cu M^{\Theta},\sa D^{\Theta})$ is a pre-metric expansion of $\cu M^{\Theta}$.
\end{cor}

\begin{proof}  The vocabulary $V^{\Theta}$ has countably many predicate and function symbols.
It is easily seen that there is a $V^{\Theta}$-formula $\theta(x,y)$ that defines
$\sa{D}^{\Theta}(x,y)$ in $\cu M^{\Theta}$.
\end{proof}

\begin{prop}  \label{p-eq-saturated}  If $\cu M$ is $\kappa$-special, then  $\cu M^{eq}$ is $\kappa^+$-special,
and $\cu M^{\Theta}$ is $\kappa$-special for each $\Theta\subseteq \Phi$.
\end{prop}

\begin{proof}  The proof is similar to the proof of Proposition \ref{p-imaginary-special}, but with many imaginary sorts.
\end{proof}

\begin{cor}  Suppose $T_e$ is a pre-metric expansion of $T$,  $\cu M$ is $\kappa$-special, and $\kappa$ has uncountable cofinality.
 Then $(\cu M_e)^{eq}$ is a sorted metric structure, and $(\cu M_e)^{\Theta}$ is a sorted metric structure for each $\Theta\subseteq\Phi$.
\end{cor}

\begin{proof}  By Propositions \ref{p-imaginary-special}, \ref{p-eq-saturated}, and Remark \ref{r-saturated-metric}.
\end{proof}

\subsection{Definable Sets}

Recall that in [BBHU], in a metric structure $\cu{M}_e$, the \emph{distance} between a $k$-tuple $\vec x\in M^k$ and a closed set $\sa S\subseteq M^k$
is the mapping
$$\dist^{\cu M_e}(\vec x,\sa S)=\inf\{\max_{i\le k} D^{\cu M_e}(x_i,y_i)\colon \vec y\in\sa S\},$$
and
 $\sa S$ is a \emph{definable set} over $A$ in $\cu M_e$ if $\sa S$
is a closed subset of $M^k$ and $\dist^{\cu M_e}(\vec x,\sa S)$ is a definable predicate over $A$ in $\cu M_e$.
Note that the empty set $\emptyset$ is always definable over $\emptyset$, and $\dist^{\cu M_e}(\vec x,\emptyset)=1$.

The following result uses the fact that the Expansion Theorem \ref{t-metric-expansion-exists} gives a pseudo-metric approximate distance.

\begin{prop}  \label{p-definable-set-absolute}
The general property \emph{[$\cu M$ is complete and $\sa S$ is a definable set in $\cu M$ over $A$]} is absolute.
For each complete general structure $\cu M$, set $A\subseteq M$, and set $\sa S\subseteq M^k$, the following are equivalent:
\begin{itemize}
\item[(a)] $\sa S$ is definable over $A$ in $\cu M$.
\item[(b)] $\sa S$ is closed in $\cu M$, and for each  $V$-formula $\varphi(\vec x,\vec y)$,
the mapping
$$\dist^{\cu M}_\varphi(\vec x,\sa S):=\inf\{ \varphi^{\cu M}(\vec x,\vec y)\colon \vec y\in\sa S\}$$
is a definable predicate over $A$ in $\cu M$.
\item[(c)]  $\sa S$ is closed in $\cu M$, and for each $V$-formula $\varphi(\vec x,\vec y)$ that is pseudo-metric in $\Th(\cu M)$,
$\dist^{\cu M}_\varphi(\vec x,\sa S)$ is a definable predicate over $A$ in $\cu M$.
\end{itemize}
 \end{prop}

\begin{proof}  We first consider an arbitrary pre-metric expansion $\cu M_e$ of $\cu M$.
By Proposition \ref{c-absolute-complete}, $\cu M$ is complete if and only if $\cu M_e$ is complete.
By Proposition \ref{p-closure}, a subset of $M^k$ is closed in $\cu M_e$ if and only if it is closed in $\cu M$.
Let (a')--(c') be the statements (a)--(c) with $\cu M_e$ in place of $\cu M$.
By Theorem 9.17 of [BBHU], (a'), (b'), and (c') are equivalent. We now show that (b) is equivalent to (b').

(b') $\Rightarrow$ (b):  Assume (b').  Let $\varphi(\vec x,\vec y)$ be a $V$-formula.  By (b'), $\dist_\varphi^{\cu M_e}$
is definable in $\cu M_e$ over $A$. $\dist_\varphi^{\cu M_e}$ is definable in $\cu M$ over $A$ by Proposition \ref{p-definable-predicate}.
Since $\psi^{\cu M}=\psi^{\cu M_e}$ for every $V$-formula $\psi$, $\dist_\varphi^{\cu M_e}=\dist_\varphi^{\cu M}$.  This proves (b).

(b) $\Rightarrow$ (b'):  Assume (b).
Let $\varphi_m(\vec x,\vec y)$ be the $V$-formula $\max_{i\le k} d_m(x_i,y_i)$. By (b),
$$\dist^{\cu M}_{\varphi_m}(\vec x,\sa S)=\inf\{\max_{i\le k} d_m^{\cu M}(x_i,y_i)\colon \vec y\in \sa S\}$$
 is a definable predicate over $A$ in $\cu M$.  Moreover, $\langle \varphi_m\rangle_{n\in\BN}$ is Cauchy in $T$,
 and converges uniformly to $\dist^{\cu M_e}(\vec x,\sa S)$.  Therefore $\dist^{\cu M_e}(\vec x,\sa S)$ is a definable predicate over
 $A$ in $\cu M_e$, so (a') holds and hence (b') holds.

Since (b) is equivalent to (b') for all $\cu M_e$, the property [$\cu M$ is complete and (b)] is the absolute version of
the property [$\cu M_e$ is complete and (b')].  The absolute version is unique and (a') is equivalent to (b'), so (a) is equivalent to (b).
It is trivial that (b) implies (c).

(c) $\Rightarrow$ (b): Assume (c).
By Theorem \ref{t-metric-expansion-exists}, there exists a pre-metric expansion $\cu M_f$ of $\cu M$ with a pseudo-metric approximate distance
$\langle d_m(x,y)\rangle_{m\in\BN}$.
 By hypothesis, each formula $d_m(x,y)$ is pseudo-metric in $\Th(\cu M)$.
Let $\varphi_m(\vec x,\vec y)$ be the $V$-formula $\max_{i\le k} d_m(x_i,y_i)$, which is also pseudo-metric in $\Th(\cu M)$. By (c),
$$\dist^{\cu M}_{\varphi_m}(\vec x,\sa S)=\inf\{\max_{i\le k} d_m^{\cu M}(x_i,y_i)\colon \vec y\in \sa S\}$$
 is a definable predicate over $A$ in $\cu M$.  Moreover, $\langle \varphi_m\rangle_{m\in\BN}$ is Cauchy in $\Th(\cu M)$,
 and converges uniformly to $\dist^{\cu M_f}(\vec x,\sa S)$.  Therefore $\dist^{\cu M_f}(\vec x,\sa S)$ is a definable predicate over
 $A$ in $\cu M_f$.  It follows that (a') holds for $\cu M_f$, and therefore (b) holds.
\end{proof}

We now turn to the notions of definable and algebraic closure.
The following definitions agree with the corresponding definitions in [BBHU] in the case that $\cu M$ is a metric structure.

\begin{df}  Let $\cu M$ be a complete general structure.  An element $b$ belongs to the \emph{definable closure of} $A$ in $\cu M$, in symbols $b\in\dcl_{\cu M}(A)$,
if the singleton $\{b\}$ is definable over $A$ in $\cu M$.  $b$ belongs to the \emph{algebraic closure of} $A$ in $\cu M$, $b\in\acl_{\cu M}(A)$,
if there is a compact set $C$ in $\cu M$ such that $b\in C$ and $C$ is definable over $A$ in $\cu M$.  A tuple $\vec b$ belongs to $\dcl_{\cu M}(A)$ or
$\acl_{\cu M}(A)$ if each term $b_i$ does.
\end{df}

One can use Proposition \ref{p-definable-set-absolute} to obtain necessary and sufficient conditions for $b\in\dcl_{\cu M}(A)$ and $b\in\acl_{\cu M}(A)$.

\begin{cor}  \label{c-dcl-acl-absolute}
If $\cu M$ is a complete general structure,
then $\dcl_{\cu M}(A)=\dcl_{\cu M_e}(A)$ and $\acl_{\cu M}(A)=\acl_{\cu M_e}(A)$.
\end{cor}

\begin{proof}  By Propositions \ref{p-closure} and \ref{p-definable-set-absolute}.
\end{proof}

\begin{cor} \label{c-acl-elem-ext}
Suppose $\cu M, \cu N$ are complete general structures and $A\subseteq\cu M\prec\cu N$.
Then $\dcl_{\cu M}(A)=\dcl_{\cu N}(A)$ and $\acl_{\cu M}(A)=\acl_{\cu N}(A)$.
\end{cor}

\begin{proof}   By Proposition \ref{c-absolute-complete},
$\cu M_e$ and $\cu N_e$ are metric structures.  By Corollary \ref{c-expansion-elem-ext} we have $\cu M_e\prec\cu N_e$.
By Corollary \ref{c-dcl-acl-absolute} above and Corollary 10.5 of [BBHU],
$\dcl_{\cu M}(A)=\dcl_{\cu M_e}(A)=\dcl_{\cu N_e}(A)=\dcl_{\cu N}(A)$, and similarly for $\acl$.
\end{proof}

\begin{cor}  \label{c-acl-saturated}
Suppose $\cu M$ is a reduced $\aleph_1$-saturated general structure, $A\subseteq M$, and $b\in M$.
\begin{itemize}
\item[(i)] $b\in\dcl_{\cu M}(A)$ if and only if $b$ is the only realization of $\tp(b/A)$ in $\cu M$.
\item[(ii)]  $b\in\acl_{\cu M}(A)$ if and only if  the set of realizations of $\tp(b/A)$ is compact in $\cu M$.
\end{itemize}
\end{cor}

\begin{proof}
By Exercises 10.7 (4) and 10.8 (4) of [BBHU], (i) and (ii) hold with an $\aleph_1$-saturated metric structure $\cu N$ in place of $\cu M$.
By Corollary  \ref{c-saturated-complete}, $\cu M$ is complete.
 By Proposition \ref{p-saturated-metric-expansion}, $\cu M_e$ is an $\aleph_1$-saturated metric structure.  By Corollary \ref{c-dcl-acl-absolute},
 $\dcl_{\cu M}(A)=\dcl_{\cu M_e}(A)$ and $\acl_{\cu M}(A)=\acl_{\cu M_e}(A)$.
By Corollary \ref{c-type-over-A}, an element $c\in M$ realizes $\tp_{\cu M}(b/A)$ if and only if it realizes $\tp_{\cu M_e}(b/A)$.
By Proposition \ref{p-closure}, a subset of $M$ is compact in $\cu M$ if and only if it is compact in $\cu M_e$.
Therefore (i) and (ii) hold for $\cu M$.
\end{proof}

The paper [EG] introduced the notion of a metric structure admitting weak elimination of finitary imaginaries.  We now introduce the analogous notion
for general complete structures.

An exponentially Cauchy sequence $\langle\varphi\rangle$ in $\cu M$ is called \emph{finitary} if there is an $\ell\in\BN$
such that for each $m\in\BN$, $\varphi_m$ has at most the free variables $(\vec x,y_0,\ldots,y_\ell)$.  If $\langle\varphi\rangle$
is finitary, the elements of sort $\BS_{i}^{\langle\varphi\rangle}$ in ${\cu M}^{\langle\varphi\rangle}$ are called \emph{finitary imaginaries}.

\begin{df}  Let  $\cu M$ be a reduced $\aleph_1$-saturated general structure
\begin{itemize}
\item[(i)]
 $\cu{M}$ \emph{admits  elimination of finitary imaginaries} if for every finitary imaginary $b\in {\cu M}^{\langle\varphi\rangle}$,
there is a finite tuple $\vec c$ from $M$ such that in ${\cu M}^{\langle\varphi\rangle}$, $b\in\dcl(\vec c)$ and $\vec c\in\dcl(b)$.
\item[(ii)]   $\cu{M}$ \emph{admits weak elimination of finitary imaginaries} if for every finitary imaginary $b\in {\cu M}^{\langle\varphi\rangle}$,
there is a finite tuple $\vec c$ from $M$ such that in ${\cu M}^{\langle\varphi\rangle}$, $b\in\dcl(\vec c)$ and $\vec c\in\acl(b)$.
\end{itemize}
\end{df}

The next result shows that the property of being reduced and admitting (or weakly admitting) elimination of finitary imaginaries is absolute.

\begin{cor}  For every reduced $\aleph_1$-saturated general structure $\cu{M}$ and pre-metric expansion $\cu{M}_e$ of $\cu M$, $\cu{M}_e$ admits (or weakly admits) elimination of
finitary imaginaries if and only if $\cu{M}$ admits  (or weakly admits) elimination of finitary imaginaries.
\end{cor}

\begin{proof}  By Corollaries \ref{c-S-premetric}  and  \ref{c-dcl-acl-absolute}.
\end{proof}

\subsection{Stable Theories}

As is often done the literature on stable theories, we will work in a monster structure of inaccessible cardinality.
The axioms of ZFC do not imply the existence of an inaccessible cardinal.  However, one can avoid
inaccessible cardinals by working in a universal domain, at the cost of some minor complications (see [BBHU]).

\emph{We assume hereafter that $T$ is a complete $V$-theory, and that $\upsilon$ is an inaccessible cardinal greater than $|V|+\aleph_0$.
By a \emph{monster structure} we mean a reduced $\upsilon$-saturated structure of cardinality $\upsilon$.
  We let $\BM$ be a monster model of $T$, and let $\BM_e$ be a pre-metric expansion of $\BM$.}

For the rest of this paper we will work exclusively within  $\BM$ and $\BM_e$.
By a \emph{small set} we mean a set of cardinality $<\upsilon$.  $A, B$ will always denote small subsets of $\BM$.

\begin{rmk} By Remark \ref{r-special-inaccessible}, $\BM$ is $\upsilon$-special. By the Uniqueness Theorem for Special Models, $T$ has a unique monster model up to
isomorphism.  By Corollary \ref{c-special-metric-expansion}, $\BM_e$ is also monster structure.
By Corollary \ref{c-saturated-complete}, $\BM$ and $\BM_e$ are complete, so $\BM_e$ is a metric structure.
\end{rmk}

Let us recall the definition of a $\lambda$-stable metric theory in [BBHU], where $\lambda$ is an infinite cardinal.
For each small set $A$, the $D$-metric on the type space $S_1^{\BM_e}(A)$ is defined by
$$D^{\BM_e}(p,q)=\inf\{D^{\BM_e}(b,c)\colon \tp_{\BM_e}(b/A)=p, \tp_{\BM_e}(c/A)=q\}.$$
The complete metric theory $\Th(\BM_e)$ is \emph{$\lambda$-stable} if for each $A$ of cardinality $\le\lambda$
there is a dense subset of $S_1^{\BM_e}(A)$ of cardinality $\le \lambda$ with respect to the $D$-metric.
And $\Th(\BM_e)$ is stable if it is $\lambda$-stable for some small cardinal $\lambda$.
Here, it will be convenient to say that $\BM_e$ is stable instead of saying that $\Th(\BM_e)$ is stable.
So, by a $\lambda$-stable metric structure we mean a monster metric structure whose complete metric theory is $\lambda$-stable.

\begin{prop}  \label{p-stable-absolute}
The property of being $\lambda$-stable is absolute.  $\BM$ is $\lambda$-stable if and only if for each $A$ of cardinality $\le\lambda$
there is a set $B\subseteq S_1^{\BM}(A)$ of cardinality $\le\lambda$ such that the set $\{b\colon \tp_{\BM}(b/A)\in B\}$ is dense in $\BM$.
\end{prop}

\begin{proof}  By  Corollary \ref{c-type-over-A}, for all elements $b,c$ of $\BM$, $\tp_\BM(b/A)=\tp_\BM(c/A)$ if and only if
$\tp_{\BM_e}(b/A)=tp_{\BM_e}(c/A)$.  By Proposition \ref{p-closure}, a subset of $\BM$ is dense in $\BM$ if and only if it is dense in $\BM_e$.
So it suffices to prove that a set $B\subseteq S_1^{\BM_e}(A)$ is dense in $S_1^{\BM_e}(A)$ with respect to the $D$-metric if and only
if the set $B':=\{b\colon \tp_{\BM_e}(b/A)\in B\}$ is dense in $\BM_e$.
We prove the non-trivial direction here.  Let $c\in \BM_e$ and $p=\tp_{\BM_e}(c/A)$.  Since $B$ is dense in $S_1^{\cu M_e}(A)$,
there is a sequence $\langle p_n\rangle_{n\in\BN}$ in $B$ that converges to $p$ in the $D$-metric.
Since $\BM_e$ is $\kappa^+$-saturated, for each $n$ there exists $b_n$ such that $\tp_{\BM_e}(b_n/A)=p_n$ and
$D^{\BM_e}(b_n,c)=D^{\BM_e}(p_n,p)$.  Then $b_n\in B'$ for each $n$ and $\lim_{n\to\infty} b_n=c$ in $\BM_e$, so $B'$ is dense in $\BM_e$.
\end{proof}

\begin{cor}  \label{c-omega-stable}
$\BM$ is $\aleph_0$-stable if and only if $\BM$ is $\lambda$-stable for all infinite $\lambda$.
\end{cor}

\begin{proof}  This follows from the corresponding result for metric structures (Remark 14.8 in [BBHU]) and the fact that being $\lambda$-stable is absolute
(Proposition \ref{p-stable-absolute}).
\end{proof}

\begin{exercise}  Suppose every predicate symbol and every function symbol in $V$ is unary.  Then $T$ is $\aleph_0$-stable.
Hint: Show that for all $b.c\in\BM$ and all $A\subseteq\BM$, $\tp_\BM(b)=\tp_\BM(c)$ if and only if $\tp_\BM(b/A)=\tp_\BM(c/A)$.
\end{exercise}

\begin{cor}  \label{c-stable-discrete}
Suppose $\lambda=\lambda^{\aleph_0}$. The following are equivalent:
\begin{itemize}
\item[(i)] $T$ is stable.
\item[(ii)] $T$ is $\lambda$-stable.
\item[(iii)]  For every set $A$ of cardinality $\le\lambda$, $|S_{1}^{\BM}(A)|\le\lambda$.
\end{itemize}
\end{cor}

\begin{proof}  The equivalence of (i) and (ii) follows from the corresponding result for metric structures
(Theorem 8.5 in [BU]) and the fact that $\lambda$-stability is absolute.
It is trivial that (iii) implies (ii).
Assume (ii).  Suppose $A\subseteq\BM$, and $|A|\le\lambda$.  By (ii) and Proposition \ref{p-stable-absolute}, there is a set
$B\subseteq S_1^{\BM}(A)$ of cardinality $\le\lambda$ such that the set $B':=\{b\colon \tp_{\BM}(b/A)\in B\}$ is dense in $\BM$.
By Remark \ref{r-saturated-types}, every complete type in $S_{1}^{\BM}(A)$ is realized in $\BM$.  Therefore,
since $B'$ is dense in $\BM$, $B$ is dense in $S_1^{\BM}(A)$.  So $|S_1^{\BM}(A)|\le \lambda^{\aleph_0}=\lambda$, and (iii) follows.
\end{proof}

In [BBHU], a \emph{stable independence relation} on a monster metric structure $\BM=\BM_e$ is a ternary relation $A\ind_C B$ on small subsets of  $\BM$ that has the following
properties\footnote{In naming these properties, we follow Adler [Ad], rather than [BBHU]}.
\begin{itemize}
\item \emph{Invariance under automorphisms of $\BM$.}
\item \emph{Symmetry:} $A\ind_C B$ if and only if $B\ind_C A$.
\item \emph{Transitivity:} $A\ind_C BD$ if and only if $A\ind_C B$ and $A\ind_{BC} D$.
\item \emph{Finite character:} $A\ind_C $ if and only if $\vec a\ind_C B$ for all finite $\vec a\subseteq A$.
\item \emph{Full existence:}  For all $A,B,C$ there exists $A'$ such that
$(\BM_{A'})_C\equiv(\BM_A)_C$  and $A'\ind_C B$.
\item \emph{Strong local character:} For each finite $\vec a$, there exists $B_0\subseteq B$ of cardinality $\le |V|+\aleph_0$ with $\vec a\ind_{B_0} B$.
\item \emph{Stationarity:} For all small complete $\cu{M}^0\prec\BM$ and all small $A, A', B$, if
$$ (\BM_A)_{M^0}\equiv(\BM_{A'})_{M^0},  \quad A\ind_{M^0} B, \quad A'\ind_{M^0} B,$$
then
$$(\BM_A)_{B\cup M^0}\equiv (\BM_{A'})_{B\cup M^0}$$
\end{itemize}

We define a \emph{stable independence relation} on a monster general structure $\BM$ to be a ternary relation  that satisfies the same seven properties with respect to $\BM$.

\begin{thm}
\noindent\begin{itemize}
\item[(i)]  A relation $\ind$ is a stable independence relation on $\BM$ if and only if it is a stable independence relation on $\BM_e$.
\item[(ii)] $\BM$ is stable if and only if it has a stable independence relation, and also if and only if it  has a unique stable independence relation.
\end{itemize}
\end{thm}

\begin{proof}  (i):  Let $\ind$ be a ternary relation on small subsets of the universe of $\BM$.  We show that each of the properties listed above for
a stable independence relation is absolute, that is, holds for $\BM$ if and only if it holds for $\BM_e$.
It is trivial that Symmetry, Transitivity, Finite Character, and Local Character are absolute.
By Lemma \ref{l-reduced-Me} (iii),  Invariance under automorphisms is absolute.
To show that Full Existence is absolute, use  Corollary \ref{c-type-over-A},
which says that two tuples realize the same complete type over $C$
in $\BM$ if and only if they realize the same complete type over $C$ in $\BM_e$.  To show that Stationarity is absolute, use Corollary \ref{c-type-over-A} again and
Corollary \ref{c-expansion-elem-ext}, which says that $\cu{M}^0\prec\BM$ if and only if $\cu{M}^0_e\prec\BM_e$.

(ii) follows immediately from (i) and Theorem 14.1 of [BBHU].
\end{proof}

We now consider the approach to stability theory via definable types.

\begin{df}  We say that a complete $k$-type $\tp_\BM(\vec a/B)$  is \emph{definable over} $C$ in $\BM$ if for each $V$-formula
$\varphi(\vec x,\vec y)$ with parameters in $C$ there is a mapping $\sa Q\colon M^{|\vec y|}\to[0,1]$ that is definable over $C$ in $\BM$
such that for all $\vec b\in B^{|\vec y|}$ we have $\varphi^{\BM}(\vec a,\vec b)=\sa Q(\vec b).$
\end{df}

In the case that $\BM=\BM_e$ (so $\BM$ is a  metric structure), the above definition is the same as the corresponding definition in [BBHU].
 We now show that the general property [$\tp(\vec a/B)$ is definable over $C$] is absolute.

\begin{prop} \label{p-definable-type-absolute}
Let $\vec a\in M$ and $B,C$ be small subsets of $M$.  $\tp_\BM(\vec a/B)$ is definable over $C$ in $\BM$ if and only if $\tp_{\BM_e}(\vec a/B)$ is definable over $C$ in $\BM_e$.
\end{prop}

\begin{proof}
Suppose first that $\tp_{\BM_e}(\vec a/B)$ is definable over $C$ in $\BM_e$.  Let $\varphi(\vec x,\vec y)$ be a $V$-formula with parameters in $C$.
By hypothesis there is a mapping $\sa Q\colon M^{|\vec y|}\to[0,1]$ that is definable over $C$ in $\BM_e$
such that for all $\vec b\in B^{|\vec y|}$ we have $\varphi^{\BM_e}(\vec a,\vec b)=\sa Q(\vec b).$  By Proposition \ref{p-definable-predicate},
$\sa Q$  is definable over $C$ in $\BM$, and by Remark \ref{r-formula-expansion}, $\varphi^{\BM}(\vec a,\vec b)=\varphi^{\BM_e}(\vec a,\vec b)=\sa Q(\vec b)$,
so $\tp_\BM(\vec a/B)$ is definable over $C$ in $\BM$.

Now suppose that  $\tp_{\BM}(\vec a/B)$ is definable over $C$ in $\BM$.  Let $\varphi(\vec x,\vec y)$ be a $V_D$-formula with parameters in $C$.
By Lemma \ref{l-expansion-converge}, $\varphi^{\BM_e}$ is definable in $\BM$ over $C$.
Then there is a sequence of $V$-formulas $\langle\varphi_m(\vec x,\vec y)\rangle_{m\in\BN}$ with parameters in $C$
such that for each $m\in\BN$,
$$(\forall\vec a\in M^{|\vec x|})(\forall\vec b\in M^{|\vec y|})|\varphi_m^{\BM}(\vec a,\vec b)-\varphi^{\BM_e}(\vec a,\vec b)|\le 2^{-m}.$$
By hypothesis, for each $m\in\BN$ there is a mapping $\sa Q_m\colon \BM^{|\vec y|}\to[0,1]$ definable over
$C$ in $\BM$ such that for all $\vec b\in B^{|\vec y|}$ we have $\varphi_m^{\BM}(\vec a,\vec b)=\sa Q_m(\vec b).$
By Proposition \ref{p-definable-predicate}, each $\sa Q_m$ is definable in $\BM$ over $C$, so
there is a  sequence of $V$-formulas $\langle\psi_m(\vec y)\rangle_{m\in\BN}$ with parameters in $C$  such that
$$ (\forall \vec b\in \BM^{|\vec y|})(|\psi_m^\BM(\vec b)-\sa Q_m(\vec b)|\le 2^{-m})$$
and
$$(\forall \vec b\in B^{|\vec y|})(|\psi_m^{\BM}(\vec b)-\sa Q_m(\vec b)|=|\psi_m^{\BM}(\vec b)-\varphi_m^{\BM}(\vec a,\vec b)|\le 2^{-m}).$$
Then
$$(\forall \vec b\in B^{|\vec y|})(|\psi_m^{\BM}(\vec b)-\varphi^{\cu M_e}(\vec a,\vec b)|\le 2\cdot 2^{-m}.$$
We now modify $\langle\psi_m(\vec y)\rangle_{m\in\BN}$ to a sequence of $V$-formulas that is Cauchy in $\Th(\BM)$ using the forced convergence trick of [BU].
Note that for all $\vec b\in B^{|\vec y|}$,
$$|\psi^{\BM}_m(\vec b)-\psi^{\BM}_{m+1}(\vec b)|\le 3\cdot 2^{-m}.$$
We inductively define $\theta_0=\psi_0$, and
$$\theta_{m+1}=\max(\theta_m-3\cdot 2^{-m},\min(\theta_m+3\cdot 2^{-m},\psi_{m+1})).$$
Then $\theta_m^{\BM}(\vec b)=\psi_m^{\BM}(\vec b)$
for all $m$ and all $\vec b\in B^{|\vec y|}$.  Moreover,
$\langle\theta_m(\vec y)\rangle_{m\in\BN}$ is a sequence of $V$-formulas with parameters in $C$ such that
$$ (\forall \vec b)|\theta_m^{\BM}(\vec b)-\theta_{m+1}^{\BM}(\vec b)|\le 3\cdot 2^{-m},$$
so $\langle\theta_m(\vec y)\rangle_{m\in\BN}$ in Cauchy in $\Th(\BM)$.  Therefore by Proposition \ref{p-definable-predicate},
$\sa Q :=[\lim \theta_m]^{\BM}$ is a definable predicate over $C$ in $\BM_e$, and $\varphi^{\BM_e}(\vec a,\vec b)=\sa Q(\vec b)$
for all $\vec b\in B^{|\vec y|}$.  Therefore $\tp_{\BM_e}(\vec a/B)$ is definable over $C$ in $\BM_e$.
\end{proof}

\begin{df}  We say that a complete type $p=\tp_{\BM}(\vec a/B)$ \emph{does not fork over} $C$ in $\BM$ if
$p$ is definable over $\acl_{\BM}(C)$.
\end{df}

In the case that $\BM=\BM_e$, the above definition is the same as the definition in [BBHU].
We now show that the property [$\tp(\vec a/B)$ does not fork over $C$] is absolute.

\begin{cor}  For every $a, B$, and $C$, $\tp_{\BM}(a/B)$ does not fork over $C$ in $\BM$ if and only if $\tp_{\BM_e}(a/B)$ does not fork over $C$ in $\BM_e$.
\end{cor}

\begin{proof}
By Proposition \ref{p-definable-type-absolute},
$\tp_{\BM}(\vec a/B)$ is definable over $\acl_\BM(C)$ in $\BM$ if and only $\tp_{\BM_e}(\vec a/B)$ is definable over $\acl_{\BM}(C)$ in $\BM_e$.
We have $\acl_{\BM}(C)=\acl_{\BM_e}(C)$ by Corollary \ref{c-dcl-acl-absolute}.
\end{proof}

\begin{cor}
$\BM$ is stable if and only if for every small $\cu{N}\prec\BM$ and every tuple $\vec a$ of elements of $\BM$, $\tp_\BM(\vec a/N)$ is definable over $N$ in $\BM$.
\end{cor}

\begin{proof}  By Theorem 14.16 of [BBHU], the result holds with the monster metric structure $\BM_e$ in place of $\BM$.
Corollary \ref{c-expansion-elem-ext}, $\cu{N}\prec\BM$ if and only if $\cu{N}_e\prec\BM_e$.  By Proposition \ref{p-definable-type-absolute},
$\tp_\BM(\vec a/N)$ is definable over $N$ in $\BM$ if and  only if $\tp_{\BM_e}(\vec a/N)$ is definable over $N$ in $\BM_e$.
By Proposition \ref{p-stable-absolute}, $\BM$ is stable if and only if $\BM_e$ is stable, and the result follows.
\end{proof}

\subsection{Building Stable Theories}

In both first order and continuous logic, stable structures are of particular interest because they are well-behaved and can be analyzed.
The  literature contains a wide variety of examples of stable first order structures, and several examples of stable
metric structures in areas such as Banach spaces and probability algebras.  We now present a way to build many examples of stable general structures
from first order structures that are stable, or even  stable for positive formulas.
As in [HI] and [BY03a], we exploit a connection between $[0,1]$-valued structures and positive formulas in first order structures.
We need some definitions.

A first order formula is \emph{positive} if it is built from atomic formulas using only quantifiers and the connectives $\wedge,\vee$.
A first-order structure is \emph{positively $\kappa$-saturated} if every set of positive formulas with the free variable $x$ and fewer than $\kappa$
parameters that is finitely satisfiable is satisfiable in the structure.  Note that every $\kappa$-saturated first-order structure is positively $\kappa$-saturated.
A \emph{positive monster structure} is a positively $\upsilon$-saturated first order structure of cardinality $\upsilon$.

In a positive monster structure $\BK$, the \emph{complete positive type} of a $k$-tuple $\vec b$ over a set $A$ is the set of all positive formulas $\varphi(\vec x)$
with parameters in $A$ satisfied by $\vec b$.
We say that a first-order structure $\BK$ is \emph{positively $\lambda$-stable} if $\BK$ is a positive monster structure
and, for every set $A$ of cardinality $\lambda$,
the set of complete positive $1$-types over $A$ in $\BK$ has cardinality $\le\lambda$.
In particular, every $\lambda$-stable first order monster structure is a positively $\lambda$-stable.

By a $\bigwedge$-\emph{formula} we mean a finite or countable conjunction of positive formulas, possibly with parameters in $\BK$.
We also allow the empty conjunction, whose truth value is always true.
Let $\BD$ be the set of all dyadic rationals in $[0,1]$.  Hereafter we let $q,r,s$ vary over $\BD$.  Let $\BJ$ be the set of all intervals of the form
$[0,r]$ or $[r,1]$ where $r\in\BD$.

Let $\BK$ be a first-order positive monster structure with universe $M$ whose vocabulary $W$  contains at least all the function and constant symbols of $V$.
  By a \emph{positive interpretation of $V$ in} $\BK$
we mean a function $\mathbf{I}$ that associates, with each $k$-ary predicate symbol $P\in V$ and interval $J\in\BJ$,  a  $\bigwedge$-formula
$\mathbf{I}(P,J)$ in the vocabulary $W$ with $k$ free variables, such that whenever $r<s$ we have:
\begin{itemize}
\item[(a)] $\mathbf{I}(P,[0,r])^{\BK}\subseteq \mathbf{I}(P,[0,s])^{\BK}$ and  $\mathbf{I}(P,[r,1])^{\BK}\supseteq \mathbf{I}(P[s,1])^{\BK}$.
\item[(b)]  $\mathbf{I}(P,[0,r])^{\BK}\cap \mathbf{I}(P,[s,1])^{\BK}=\emptyset$.
\item[(c)]  $\mathbf{I}(P,[r,1])^{\BK}\cup \mathbf{I}(P,[0,s])^{\BK}=M^k$.
\end{itemize}

\begin{thm}  \label{t-interpret-upgrade}
Suppose $\mathbf{I}$ is a positive interpretation of $V$ in a positive monster structure $\BK$.
There is a unique  $V$-structure $\BM=\mathbf{I}(\BK)$ with universe $M$ such that $\BM$ agrees with $\BK$ on all function and constant symbols in $V$,
and for each $k$-ary predicate symbol $P\in V$, $r\in\BD$, and  $\vec b\in M^k$,
$$ P^{\BM}(\vec b)\in [0,r] \mbox{ if and only if } \BK\models\bigwedge_{s>r} \mathbf{I}(P,[0,s])(\vec b)$$
and
$$ P^{\BM}(\vec b)\in [r,1] \mbox{ if and only if } \BK\models\bigwedge_{s<r} \mathbf{I}(P,[s,1])(\vec b)$$
(with the empty conjunction being the true sentence).
Moreover, $\BM$ is a monster structure, and for each $\lambda$, if $\BK$ is positively $\lambda$-stable then $\BM$ is $\lambda$-stable.
\end{thm}

Theorem \ref{t-interpret-upgrade} shows that one can build a $\lambda$-stable general structure $\BM$ by starting with a positively $\lambda$-stable
monster structure $\BK$ and taking any positive interpretation $\mathbf{I}$ of $V$ in $\BK$.
In fact, it is not hard to show that if $\lambda=\lambda^{\aleph_0}$, then every $\lambda$-stable general structure $\BM$ can be obtained in this way.
Since $\lambda$-stability is absolute, one can then get a $\lambda$-stable metric structure by taking any pre-metric expansion $\BM_e$.

\begin{proof} [Proof of Theorem \ref{t-interpret-upgrade}]
It is clear that there is at most one such $\BM$.  We first prove that such an $\BM$ exists.  Let
$$\mathbf{I}'(P,[0,r])=\bigwedge_{s>r}\mathbf{I}(P,[0,s]), \qquad\mathbf{I}'(P,[r,1])=\bigwedge_{s<r}\mathbf{I}(P,[s,1]).$$
Using the fact that $\BD$ is dense, one can easily check that $\mathbf{I}'$ is a positive interpretation of $V$ in $\BK$, and that $\mathbf{I}''=\mathbf{I}'$.
For each $k$-ary $P\in V$ and $\vec b\in M^k$, let
$$X=X(P,\vec b)=\{r\colon \BK\models \mathbf{I}'(P,[r,1])(\vec b)\},$$
$$  Y=Y(P,\vec b)=\{s\colon \BK\models \mathbf{I}'(P,[0,s])(\vec b)\}.$$
By (c), $X$ and $Y$ are non-empty.
 Therefore $\sup X$ exists.  We define $\BM$ to be the $V$-structure with universe $M$ that agrees with
$\BK$ on every function and constant symbol in $V$, and such that $P^{\BM}(\vec b)=\sup X$ for each $P$ and $\vec b$.
We first show that:
\begin{equation} \label{e-claim}
\mbox{For each $P$ and $\vec b\in M^k$, $\sup X=P^{\cu M}(\vec b)=\inf Y$.}
\end{equation}
If $r\in X$ and $s\in Y$, then  by (b), $\neg s<r$ and thus $r\le s$.  Hence $\sup X\le\inf Y$.
By (a),  $q\le r\in X\Rightarrow q\in X$,
and $q\ge s\in Y\Rightarrow q\in Y$.  If $\sup X < \inf Y$, then there is a  $q\in\BD$ with $\sup X<q<\inf Y$, which
would contradict (c).  This proves (\ref{e-claim}).

It follows that for all $r<1$ in $\BD$,
\begin{equation}  \label{e-Y1}
P^{\BM}(\vec b)\in[0,r]\Leftrightarrow \inf Y\in[0,r]\Leftrightarrow(\forall s>r)\BK\models\mathbf{I}'(P,[0,s])(\vec b).
\end{equation}
Then using $\mathbf{I}''=\mathbf{I}'$, one can see that whenever $r<1$,
$$P^{\cu M}(\vec b)\in[0,r]\Leftrightarrow P_{[0,r]}^{\cu{N}}(\vec b).$$
By (c), this also holds when $r=1$.  By a similar argument,
$$P^{\BM}(\vec b)\in[r,1]\Leftrightarrow \BK\models\mathbf{I}'(P,[r,1])(\vec b)$$
for all $r\in \BD$.

We next prove that $\BM$ is $\upsilon$-saturated, and thus is a monster structure.
Note that any small set of $\bigwedge$-formulas with parameters in $M$ that is finitely satisfiable in $\BK$ is satisfiable in $\BK$.
Finite disjunctions and countable conjunctions of $\bigwedge$-formulas are logically equivalent to $\bigwedge$-formulas.
Moreover, since $\BK$ is $\aleph_1$-saturated, for every  $\bigwedge$-formula $\Theta(\vec x,y)$, $(\exists y )\Theta$ and $(\forall y)\Theta$
are equivalent to $\bigwedge$-formulas in $\BK$.

By a \emph{$\BD$-interval} we mean  either the empty set or an interval  $[r,s]$ where $0\le r\le s\le 1$ and $r,s\in\BD$.
A \emph{$\BD$-rectangle} is a finite cartesian product of $\BD$-intervals.
Note that for any continuous connective $C\colon[0,1]^k\to[0,1]$ and $\BD$-interval $[r,s]$, $C^{-1}([r,s])$ is a countable intersection of finite unions
of $\BD$-rectangles.

Using the above two paragraphs, one can show by induction on the complexity of formulas that for every $V$-formula $\varphi(\vec x)$ and $\BD$-interval $[r,s]$,
there is a $\bigwedge$-formula $\Theta(\vec x)$ such that for all $\vec b\in M^k$, $\varphi^{\BM}(\vec b)\in[r,s]$ if and only if $\BK\models\Theta(\vec b)$.
For every small set $A$, every finitely satisfiable set of $\bigwedge$-formulas $\Theta(x)$ with parameters in $A$ is satisfiable in $\BK$.
It follows that every finitely satisfiable set of $V$-formulas $\varphi(x)$ with parameters in $A$ is satisfiable in $\BM$, so $\BM$ is a monster structure.

It also follows that if $\vec b, \vec c$ have the same positive type over $A$ in $\BK$, then $\tp_{\BM}(\vec b/A)=\tp_{\BM}(\vec c/A)$.  Therefore if $\BK$
is positively $\lambda$-stable, then $\BM$ is $\lambda$-stable.
\end{proof}

\subsection{Simple and Rosy Theories}

The notion of a simple theory was introduced in the context of cats in [BY03b].  In the literature (see [EG], for example), the definition
of a simple complete metric theory is obtained by translating the definition of a simple cat in [BY03b] into the context of continuous model theory.
Fact \ref{f-simple} below is a necessary and sufficient condition for a  complete metric theory to be simple that is proved in [BY03b].
Using that fact, we will show here that the property of being simple is absolute.  We need the relation $\equiv^{Ls}_C$
(for Lascar equivalence) from [BY03b].

\begin{df}  In a monster general structure $\BM$, we write $A\equiv^{Ls}_C B$ if $A, B$ are small sequences of the same length and there exist finitely many sequences
$A_1,\ldots,A_n$ such that $A=A_1, B=A_n$, and for each $k<n$, $A_k$ and $A_{k+1}$  both occur on some $C$-indiscernible sequence.
\end{df}

\begin{fact} \label{f-simple} (By Theorem 1.51 of [BY03b])
A complete metric theory  is simple if and only if its monster metric model $\BM$
has a ternary relation $A\ind_C B$ on small sets that has the following properties:
\begin{itemize}
\item Invariance under automorphisms of $\BM$.
\item Symmetry.
\item Transitivity.
\item Finite character.
\item For every $A$ and $C$, $A\ind_C C$.
\item Local character: For every $A$ there exists a small cardinal $\lambda$ such that for every $B$
there exists $B_0\subseteq B$ with $B_0\le\lambda$ and $A\ind_{B_0} B$.
\item Extension: If $A\ind_C B$ and $\hat{B}\supseteq B$, there exists $A'\equiv_{BC} A$ such that $A'\ind_C \hat{B}$.
\item Independence theorem: Whenever $A_0\ind_C A_1$, $B_0\ind_C A_0$, $B_1\ind_C A_1$, and $B_0\equiv^{Ls}_C B_1$,
there exists $B$ such that $B\ind_C A_0 A_1$, $B\equiv^{Ls}_{CA_0}B_0$, and $B\equiv^{Ls}_{CA_1} B_1$.
\end{itemize}
\end{fact}

As in first order model theory, every stable metric theory is simple.

By a \emph{simple metric structure} we mean a monster metric structure whose complete theory is simple.
The following corollary shows that the property of being a simple metric structure has an absolute version.
By Definition \ref{d-general-Q}, this tells us that the right definition of a simple general structure is a monster general structure that satisfies that absolute version.

\begin{cor}  \label{c-simple-absolute}  The property of being simple is absolute.   A general monster structure
$\BM$ is simple if and only if there exists a ternary relation $A\ind_C B$ on small subsets of $\BM$
that satisfies the properties listed in Fact \ref{f-simple}.
\end{cor}

\begin{proof}  It suffices to show that for each ternary relation $\ind$  on the universe of $\BM$, each of the condit6ions listed in
Fact \ref{f-simple} is absolute, that is, it holds for $\BM$ if and only if it holds for $\BM_e$.
We have already seen that Invariance under automorphisms is absolute.  Absoluteness for the other conditions
follow easily from  Corollary  \ref{c-type-over-A}, which says that the property $\tp(\vec b/A)=\tp(c/A)$
is absolute, and hence that indiscernibility over $C$ is absolute.
\end{proof}

The notion of a rosy metric theory is introduced by Ealy and Goldbring  [EG].

\begin{df}  \label{d-rosy} A metric structure $\BM$, or its complete metric theory $\Th(\BM)$, is
\emph{real rosy} if $\BM$ is a monster metric structure and has a ternary relation $\ind$ on small sets with the following properties:
\begin{itemize}
\item \emph{Invariance under automorphisms of $\BM$.}
\item \emph{Monotonicity:} If $A\ind_C B$, $A'\subseteq A$, and $B'\subseteq B$, then $A'\ind_C B'$.
\item \emph{Base monotonicity:} Suppose $C\in[D, B]$.  If $A\ind_D B$, then $A\ind_C B$.
\item \emph{Transitivity.}
\item \emph{Normality:} $A\ind_CB$ implies $AC\ind_C B$.
\item \emph{Extension.}
\item \emph{Countable character:}  $A\ind_C B$ if and only if $\vec a\ind_C B$ for all countable $\vec a\subseteq A$.
\item \emph{Local character.}
\item \emph{Anti-reflexivity:} $a\ind_B a$ implies $a\in \acl_\BM(B)$.
\end{itemize}
\end{df}

A ternary relation with the above properties is called a \emph{strict independence relation}.  Every strict independence relation
also satisfies Symmetry, Full Existence, and $A\ind_C C$.

\begin{cor}  \label{c-rosy-absolute}  The property of being real rosy is absolute.   A general monster structure
$\BM$ is real rosy if and only if it has a strict independence relation.
\end{cor}

\begin{proof}  It is enough to show that for each ternary relation $\ind$ on the universe $M$, each of the conditions in Definition \ref{d-rosy}
is absolute, that is, it holds for $\BM$ if and only if it holds for $\BM_e$.
We have already observed that Invariance under automorphisms is absolute.  The absoluteness of Anti-reflexivity
follows from Corollaries \ref{c-type-over-A} about types, and Corollary \ref{c-dcl-acl-absolute} about algebraic closure.  It is trivial
that the other properties for a strict independence relation are absolute.
\end{proof}

By Proposition \ref{p-eq-saturated}, if $\BM$ is a monster  general structure, then $\BM^{eq}$ is a monster structure, and $\BM^\Theta$
is a monster structure for each $\Theta\subseteq\Phi$.

\begin{df} A metric structure $\BM$, or its theory $\Th(\BM)$, is called \emph{rosy} if $\BM$ is a monster metric structure and  $\BM^{eq}$
is real rosy.
\end{df}

\begin{lemma}  \label{l-rosy-finite}  Let $\BM$ be a monster general structure, and let $\BM_e$ be a pre-metric expansion of $\BM$.
\begin{itemize}
\item[(i)] $(\BM_e)^{eq}$ has the same sorts, and the same universe in each sort, as $\BM^{eq}$,
\item[(ii)] If $\vec a,\vec b$ are tuples and $C$ is a small set in $\BM^{eq}$,
then
$$\tp_{\BM^{eq}}(\vec a/C)=\tp_{\BM^{eq}}(\vec b/C) \Leftrightarrow\tp_{(\BM_e)^{eq}}(\vec a/C)=\tp_{(\BM_e)^{eq}}(\vec b/C).$$
\end{itemize}
\end{lemma}

\begin{proof}   By Proposition \ref{p-Meeq}, $(\BM_e)^{eq}$ is a pre-metric expansion of $\BM^{eq}$.  Therefore (i) holds.
The proof of Corollary \ref{c-type-over-A} also works in this case, even though the vocabulary $V^{eq}$ has uncountably many predicate symbols,
and shows that (ii) holds.
\end{proof}

\begin{cor}  \label{c-rosy-absolute}  The property of being rosy is absolute.  A general monster structure $\BM$ is rosy if and only if
$\BM^{eq}$ has a strict independence relation.
\end{cor}

\begin{proof}  By Lemma \ref{l-rosy-finite} (i), $\BM^{eq}$ and $(\BM_e)^{eq}$ have the same sorts and universe sets.
Therefore they have the same ternary relations.  It follows from Lemma \ref{l-rosy-finite} (i) that for each  ternary relation $\ind$,
each of the conditions in Definition \ref{d-rosy} holds for $\BM^{eq}$ if and only if it holds for $(\BM_e)^{eq}$.
\end{proof}

\section{Conclusion}

We have shown that every general structure with truth values in $[0,1]$ can be made into a metric structure
by adding a distance predicate that is a uniform limit of pseudo-metric formulas, and completing the metric.
We used that result to extend many notions and results  from the class of metric structures to the class of general structures.
Thus the model theory of metric structures is considerably more broadly applicable than it initially appears.

\end{document}